%% file: KM-1.tex
\documentclass[11pt]{amsart}

\usepackage{graphicx}

\usepackage{graphics,color}

\usepackage{psfrag}

\newcommand{\PSLR}{{\bf {PSL}}(2,\R)}
\newcommand{\PSLC}{{\bf {PSL}}(2,\C)}
\newcommand{\D} {\mathbb {D}}
\newcommand{\Ha}{{\mathbb {H}}^{2}}
\newcommand{\Ho}{ {\mathbb {H}}^{3}}
\newcommand{\R} {\mathbb {R} }
\newcommand{\Z} {\mathbb {Z}}
\newcommand{\N}{\mathbb {N}}
\newcommand{\C} {\mathbb {C}}
\newcommand{\TT} {\mathbb {T}}

\newcommand{\ph}{\widehat{\partial}}

\newcommand{\wt}{\widetilde}
\newcommand{\wh}{\widehat}

\newcommand{\M}{{\bf {M}}^{3} }
\newcommand{\Su}{{\bf S}}

\newcommand{\Pant}{{\bf {\Pi}}}
\newcommand{\hl}{{\bf {hl}}}

\newcommand{\flow}{{\bf {g}}}
\newcommand{\abs}{{\bf a}}
\newcommand{\taso}{{\bf b}}
\newcommand{\len}{{\bf l}}
\newcommand{\dis}{{\bf d}}
\newcommand{\Ang}{{\bf {\Theta}}}
\newcommand{\ft}{{\bf {f}}}

\newcommand{\Col}{{\mathcal {C}}}
\newcommand{\Lab}{{\mathcal {L}}}

\newcommand{\TB}{T^{1}}
\newcommand{\NB}{N^{1}}
\newcommand{\Mes}{{\mathcal {M}}}
\newcommand{\A}{{\mathcal {A}}}
\newcommand{\Ne}{{\mathcal {N}}}
\newcommand{\KG}{{\mathcal {G}}}
\newcommand{\dista}{{\mathcal {D}}}
\newcommand{\FR}{{\mathcal {F}}}

\newcommand{\refl}{{\mathcal {R}}}

\newcommand{\IM}{\operatorname{Im}}
\newcommand{\RE}{\operatorname{Re}}
\newcommand{\lab}{\operatorname{lab}}
\newcommand{\rot}{\operatorname{rot}}
\newcommand{\foot}{\operatorname{foot}}

\newcommand{\Rel}{{\operatorname {Rel}}}
\newcommand{\dist}{{\operatorname {dis}}}

\newcommand{\Eucl}{{\operatorname {Eucl}}}

\newcommand{\tr}{{\operatorname {tr}}}
\newcommand{\id}{{\operatorname {id}}}
\newcommand{\Orb}{{\operatorname {Orb}}}

\newcommand{\Jac}{{\operatorname {Jac}}}

\newcommand{\at}{\mathbin{@}}
\newcommand{\from}{\colon}

\newtheorem{corollary}{Corollary}[section]
\newtheorem{theorem}{Theorem}[section]
\newtheorem{definition}{Definition}[section]
\newtheorem{lemma}{Lemma}[section]

\newtheorem{proposition}{Proposition}[section]
\newtheorem{assumption}{Assumption}[section]
\newtheorem{claim}{Claim}[section]

\theoremstyle{remark}
\newtheorem*{remark}{Remark}
\newtheorem*{example}{Example}

\begin{document}

\title[Nearly geodesic surfaces] {Immersing almost geodesic surfaces in a closed hyperbolic three manifold}
\author[Kahn and Markovic]{Jeremy Kahn and Vladimir Markovic}
\address{\newline Mathematics Department \newline  Stony Brook University \newline Stony Brook, 11794 \newline  NY, USA \newline  and
\newline University of Warwick \newline Institute of Mathematics \newline Coventry,
CV4 7AL, UK}
\email{kahn@math.sunysb.edu,  v.markovic@warwick.ac.uk}

\today


\begin{abstract} Let $\M$ be a closed hyperbolic three manifold.  We construct  closed
surfaces which map by immersions into $\M$ so that for each one the
corresponding mapping on the universal covering spaces  is an  embedding,  or, in
other words,  the corresponding induced mapping on fundamental groups is  an injection.

\end{abstract}

\maketitle

\let\johnny\thefootnote
\renewcommand{\thefootnote}{}

\footnotetext{JK is supported by NSF grant number DMS 0905812}
\let\thefootnote\johnny

\section{Introduction}  The purpose of this paper  is to prove the following theorem.

\begin{theorem}\label{main} Let $\M=\Ho/ \KG$ denote a closed hyperbolic three manifold where $\KG$ is a Kleinian group and let $\epsilon>0$. Then there exists
a Riemann surface $S_{\epsilon}=\Ha/F_{\epsilon}$ where $F_{\epsilon}$ is a Fuchsian group and a $(1+\epsilon)$-quasiconformal  map 
$g:\partial{\Ho} \to \partial{\Ho}$, such that the quasifuchsian group $g\circ F_{\epsilon} \circ g^{-1}$ is a subgroup of $\KG$ (here we identify the hyperbolic plane $\Ha$ with an oriented geodesic plane in $\Ho$ and the circle $\partial{\Ha}$ with the corresponding circle on the sphere $\partial{\Ho}$).
\end{theorem}
\begin{remark} In the above theorem the Riemann surface $S_{\epsilon}$ has a pants decomposition where all the cuffs have a fixed large length  and they are glued by twisting for $+1$.
\end{remark}

One can extend the map $g$ to an equivariant diffeomorphism  of the hyperbolic space. This extension defines the map   
$f:S_{\epsilon} \to \M$ and the surface $f(S_{\epsilon}) \subset \M$ is an immersed  $(1+\epsilon)$-quasigeodesic surface. 
In particular, the surface $f(S_{\epsilon})$ is essential which means that the induced map $f_{*}:\pi_1(S_{\epsilon}) \to \pi_1(\M)$ is an injection.
We summarise this in the following theorem.

\begin{theorem}\label{thm-intro-1} Let $\M$ be a closed hyperbolic 3-manifold. Then we can find a closed hyperbolic surface $S$ and a continuous map $f:S \to \M$
such that the induced map between fundamental groups is injective. 
\end{theorem}
\vskip .1cm

Let $S$ be an oriented closed topological surface with a given pants decomposition $\Col$, where $\Col$ is a maximal collection of disjoint (unoriented) simple closed curves that cut $S$ into the corresponding pairs of pants. Let $f:S \to \M$  be a continuous map and let $\rho_f:
\pi_1(S) \to \pi_1(\M)$  be the induced map between the fundamental groups. Assume that $\rho_f$ is injective on $\pi_1(\Pi)$, for every pair of pants  $\Pi$ from the pants decomposition of $S$. Then to each curve $C \in \Col$  we can assign a complex half-length $\hl(C) \in (\C/2\pi i \Z)$ and a complex twist-bend $s(C) \in \C / (\hl(C) \Z +2\pi i \Z ) $. 
We prove the following in Section 2.

\begin{theorem}\label{thm-intro-2} There are universal constants $\wh{\epsilon}, K_0>0$ such that the following holds.
Let  $\epsilon$ be such that  $\wh{\epsilon}>\epsilon>0$. Suppose $(S,\Col)$ and $f:S \to \M$ are as above, and for every $C \in \Col$ we have
$$
|\hl(C)-\frac{R}{2}|<\epsilon, \, \, \text{and} \,\, |s(C)-1|<{{\epsilon}\over{R}},
$$
\noindent
for some  $R>R(\epsilon)>0$. Then $\rho_f$ is injective and the map $\partial{\wt{f}}:\partial{\wt{S}} \to \partial{\wt{\M}}$ extends to a $(1+K_0\epsilon)$-quasiconformal map from $\partial{\Ho}$ to 
itself (here $\wt{S}$ and $\wt{\M}$  denote the corresponding universal covers).
\end{theorem}

\vskip .1cm 
It then remains to construct such a pair $(f,(S,\Col))$. If $\Pi$ is a (flat) pair of pants, we say $f:\Pi \to \M$ is a skew pair of pants  
if $\rho_f$ is injective, and $f(\partial{\Pi})$ is the union of three closed geodesics. Suppose we are given a collection 
$\{f_{\alpha}: \Pi_\alpha \to \M \}_{\alpha \in A}$ of skew pants, and suppose for the sake of simplicity that no $f_\alpha$ maps two components of $\partial{\Pi}$ to the same geodesic. 
\vskip .1cm
For each closed geodesic $\gamma$ in $\M$ we let $A_{\gamma}=\{\alpha \in A: \gamma \in  f_{\alpha} (\partial{\Pi_{\alpha} } ) \}$. Given permutations $\sigma_{\gamma}:A_{\gamma} \to A_{\gamma}$ for all such $\gamma$, we can build a closed surface in $\M$ as follows. For each $(f_{\alpha},\Pi_{\alpha})$
we make two pairs of skew pants in $\M$, identical except for their orientations. For each $\gamma$ we connect via the permutation $\sigma_{\gamma}$ the pants that induce one orientation on $\gamma$ to the pants that induce the opposite orientation on $\gamma$. We show in Section 3 that if the pants are 
``evenly distributed' around each geodesic $\gamma$ then we can build a surface this way that satisfies the hypotheses of  Theorem \ref{thm-intro-2}.
\vskip .1cm
We can make this statement more precise as follows: for each $\gamma \in f_{\alpha}(\partial{\Pi}_{\alpha})$ we define an unordered pair $\{n_1,n_2\} \in \NB(\gamma)$, the unit normal bundle to $\gamma$. The two vectors satisfy $2(n_1-n_2)=0$ in the torus $\C/(2\pi i \Z+l(\gamma)\Z)$, where $l(\gamma)$ is the complex length of $\gamma$. So we write 
$$
\foot_{\gamma}(\Pi_{\alpha})\equiv \foot_{\gamma}(f_{\alpha},\Pi_{\alpha})=\{n_1,n_2\} \in \NB(\sqrt{\gamma})=\C/(2\pi i \Z+\hl(\gamma)\Z).
$$
\noindent

We let 
$\foot(A)=\{\foot_\gamma(\Pi_{\alpha} ):  \alpha \in A,    \gamma \in \partial \Pi_\alpha\}$
(properly speaking $\foot(A)$ is a labelled set (or a multiset) rather than a set; see Section 3 for details).
We then define $\foot_{\gamma}(A)= \foot(A)|_{\NB(\sqrt{\gamma})}$.
\noindent
We let $\tau:\NB(\sqrt{\gamma}) \to \NB(\sqrt{\gamma})$ be defined by $\tau(n)=n+1+i\pi$. If for each $\gamma$ we can define a permutation $\sigma_{\gamma}\from A_\gamma \to A_\gamma$ such that 
$$
|\foot_{\gamma}(\Pi_{\sigma_{\gamma}(\alpha)})-\tau(\foot_{\gamma}(\Pi_{\alpha}))| <{{\epsilon}\over{R}},
$$
\noindent
and $|\hl(\gamma)- \frac{R}{2}|<\epsilon$ for all $\gamma \in \partial{\Pi_{\alpha}}$, then the resulting surface will satisfy the assumptions of 
Theorem \ref{thm-intro-2}. The details of the above discussion are carried out in Section 3.
\vskip .1cm
In Section 4 we construct the measure on  skew pants that after rationalisation will give us the collection $\Pi_\alpha$ we mentioned above. This is the heart of the paper. Showing that there exists a single skew pants that satisfy the first inequality in Theorem \ref{thm-intro-2} is a non-trivial theorem and the only known proofs use the ergodicity of either the horocyclic or the frame flow. This result was first formulated and proved by L. Bowen \cite{bowen}, where he used the horocyclic flow to construct such skew pants. Our construction is different. We use the frame flow to construct a measure on skew pants whose equidistribution properties follow from the exponential mixing of the frame flow.   This exponential mixing is a result of Moore \cite{moore} (see also \cite{pollicott})  (it has been shown by Brin and Gromov \cite{brin-gromov} that for a much larger class of negatively curved manifolds the frame flow is strong mixing). The detailed outline of this construction is given at the beginning of Section 4.
\vskip .1cm
We point out that Cooper-Long-Reid \cite{c-l-r} proved the existence of essential surfaces in cusped finite volume hyperbolic 3-manifolds. Lackenby 
\cite{lackenby} proved the existence of such surfaces in all closed hyperbolic 3-manifolds that are arithmetic. 

\subsection{Acknowledgement} We would like to thank the following people for their interest in our work and suggestions for writing the paper: 
Ian Agol, Nicolas Bergeron, Martin Bridgeman, Ken Bromberg, Danny Calegari, Dave Gabai, Bruce Kleiner, Francois Labourie, Curt McMullen, Yair Minsky, Jean Pierre Otal, Peter Ozsvath, Dennis Sullivan, Juan Suoto, Dylan Thurston, and Dani Wise. In particular, we are grateful to the referee for numerous comments and suggestions that have improved the paper.

\section{Quasifuchsian representation of a surface group}

\subsection{The Complex Fenchel-Nielsen coordinates}

Below we define the Complex Fenchel-Nielsen coordinates. For a very detailed account we refer to \cite{series} and \cite{kour}.
Originally the coordinates were defined in \cite{ser-tan} and \cite{kour}. 

A word on notation. By $d(X,Y)$ we denote the hyperbolic distance between sets $X,Y \subset \Ho$. If $\gamma^{*} \subset \Ho$ is an oriented geodesic and $p,q \in \gamma^{*}$ then $\dis_{\gamma^{*}}(p,q)$ denotes the signed real distance between $p$ and $q$. Let $\alpha^{*},\beta^{*}$ be two  oriented geodesics in $\Ho$, and let $\gamma^{*}$ be the geodesic that is orthogonal to both $\alpha^{*}$ and $\beta^{*}$, with an orientation. Let $p=    \alpha^{*} \cap \gamma^{*}$ and $q=\beta^{*} \cap \gamma^{*}$.
Let $u$ be the tangent vector to $\alpha^{*}$ at $p$, and $v$ be  the tangent vector to $\beta^{*}$ at $q$. We let $u'$ be the parallel transport of $u$ to $q$. By $\dis_{\gamma^{*}}(\alpha^{*},\beta^{*})$ we denote the complex distance between $\alpha^{*}$ and $\beta^{*}$ measured along $\gamma^{*}$. The real part is given by $\RE(\dis_{\gamma^{*}}(\alpha^{*},\beta^{*} ) )=\dis_{\gamma^{*}}(p,q)$. The imaginary part $\IM(\dis_{\gamma^{*}}(\alpha^{*},\beta^{*}))$ is the oriented angle from $u'$ to $v$, where the angle is oriented by $\gamma^{*}$ which is orthogonal to both $u'$ and $v$. 
The complex distance is well defined $\pmod {2k\pi i}$, $k \in \Z$. In fact, every identity we write in terms of complex distances is therefore assumed to be true $\pmod {2k\pi i}$. We have the following identities
$\dis_{\gamma^{*}}(\alpha^{*},\beta^{*})=-\dis_{\gamma^{*}}(\beta^{*},\alpha^{*})$; $\dis_{-\gamma^{*}}(\alpha^{*},\beta^{*})=-\dis_{\gamma^{*}}(\alpha^{*},\beta^{*})$, and  
$\dis_{\gamma^{*}}(-\alpha^{*},\beta^{*})=\dis_{\gamma^{*}}(\alpha^{*},\beta^{*})+i \pi$.

We let $\dis(\alpha^*,\beta^*)$ (without a subscript for $\dis$) denote the unsigned complex distance equal to $\dis_{\gamma^{*}}(\alpha^*,\beta^*)$
modulo $<z \to -z>$. We will write $\dis(\alpha^*,\beta^*) \in (\C   / 2 \pi i \Z )/  \Z_2$, where $\Z_2$ of course stands for $<z \to -z>$.
We observe that $\dis(\alpha^*,\beta^*)= \dis(\beta^*,\alpha^*)=\dis(-\alpha^*,-\beta^*)= \dis(-\beta^*,-\alpha^*)$.

For a loxodromic element  $A \in \PSLC$, by $\len(A)$ we denote its complex translation length. The number $\len(A)$ has a positive real part and it is defined  $\pmod {2k\pi i}$, $k \in \Z$. By $\gamma^{*}$ we denote the oriented axis of $A$, where $\gamma^{*}$ is oriented so that the attracting fixed point of $A$ follows the repelling fixed point.

Let $\Pi^{0}$ be a topological pair of pants (a three holed sphere). We consider $\Pi^0$ as a manifold with boundary, that is we assume that $\Pi^0$ contains its cuffs.  
We say that a pair of pants in a closed hyperbolic 3-manifold $\M$  is an injective homomorphism 
$\rho:\pi_1(\Pi^{0}) \to \pi_1(\M)$, up to conjugacy. This induces a representation
$$
\rho:\pi_1(\Pi^{0}) \to \PSLC, 
$$
up to conjugacy, which in general we also call a free-floating pair of pants. A pair of pants in $\M$ is determined by (and determines) a continuous map $f:\Pi^{0} \to \M$, up to homotopy, and free-floating pair of pants likewise determines a map 
$$
f:\Pi^{0} \to \Ho / \rho(\pi_1(\Pi^0))=M_{\rho},
$$
up to homotopy.

Suppose $\rho:\pi_1(\Pi^0) \to \PSLC$ is a free-floating pair of pants, and $\rho=f_{*}$, where $f:\Pi^0 \to M_{\rho}$. We orient the components $C_i$ of $\partial{\Pi}^0$ so that $\Pi^0$ is on the left of each $C_i$. For each $i$, there is a unique oriented closed geodesic $\gamma_i$ in $M_{\rho}$ freely homotopic to $f(C_0)$. Now let $a_i$ be the simple non-separating arc on $\Pi^0$ 
connecting $C_{i-1}$ and $C_{i+1}$ (we take the subscript $\pmod{3}$). We can homotop $f$ so that $f$ maps each $C_i$ to $\gamma_i$, and maps $a_i$ to an arc $\eta_i$ from $\gamma_{i-1}$ to $\gamma_{i+1}$ that is orthogonal at its endpoints to $\gamma_{i-1}$ and $\gamma_{i+1}$. 

While such an $f$ is not unique, the 1-complex made of the $\gamma_i$ and the $\eta_i$ together divide $f(\Pi^0)$ into two singular regions whose boundaries are geodesic right-angled hexagons.
Because  the geometry of each of these two hexagons is determined by these unsigned complex distances $\dis_{\eta_{i}}(\gamma_{i-1},\gamma_{i+1})$, the two right-angled hexagons are isometric.

Let us fix for the moment $i \in \{0,1,2\}$. We then orient $\eta_{i-1}$ and $\eta_{i+1}$ to point away from $\gamma_i$ (so the signed complex distance $\dis_{\eta_{i \pm 1}}(\gamma_i,\gamma_{i \mp 1})$ has positive real part). Recall that  $\dis_{\gamma_{i}}(\eta_{i-1},\eta_{i+1})$ denotes the signed complex distance from $\eta_{i-1}$ to $\eta_{i+1}$, along $\gamma_i$.
Because the two hexagons are isometric, 
$$
\dis_{\gamma_{i}}(\eta_{i-1},\eta_{i+1})=\dis_{\gamma_{i}}(\eta_{i+1},\eta_{i-1}).
$$
We let

$$
\hl(\gamma_i)=\dis_{\gamma_{i}}(\eta_{i-1},\eta_{i+1}).
$$

We can also think of this definition on the universal cover $\Ho$ as follows: We conjugate $\rho$ so that there is a lift $\wt{\gamma}_i$ of $\gamma_i$ to $\Ho=\{(x,y,z):z>0\}$ that connects $0$ and $\infty$. We let $A_{\gamma_{i}} \in \PSLC$ be such that $\gamma_{i}=\wt{\gamma}_i /\langle A_{\gamma_{i}} \rangle $. Then $A_{\gamma_{i}}:\Ho \to \Ho$ extends to  map $\wh{\C}=\partial{\Ho}$ to itself
by $z \mapsto e^{\len(\gamma_{i})} \cdot z$.

Moreover, the lifts of $\eta_{i-1}$ and $\eta_{i+1}$ that intersect $\wt{\gamma}_{i}$ will alternate along $\wt{\gamma}_{i}$ (so we can define $\dis_{\gamma_{i}}(\eta_{i-1},\eta_{i+1})$ as 
$\dis_{\wt{\gamma}_{i}}(\wt{\eta}_{i-1},\wt{\eta}_{i+1})$, where $\wt{\eta}_{i-1}$ is a lift of $\eta_{i-1}$ that intersects $\wt{\gamma}_{i}$ and $\wt{\eta}_{i+1}$ is the next lift of $\eta_{i+1}$ along $\wt{\gamma}_{i}$). If we define $\sqrt{A_{\gamma_{i}}} \in \PSLC$ so that it maps $z \mapsto e^{\hl(\gamma_{i})} \cdot z$, then it will map the lifts of $\eta_{i-1}$ to the lifts of $\eta_{i+1}$, and vice verse.

Moreover, the unit normal bundle $\NB(\wt{\gamma}_i)$ is a torsor for $\C^{*} \equiv \C / 2\pi i \Z$, and the unit normal bundle $\NB(\gamma_i)$ is a torsor for 
$$
\C^{*}/\langle A_{\gamma_{i}} \rangle=\C/ 2\pi i \Z+\len(\gamma_{i}) \cdot \Z.
$$

\begin{remark} Let $G$ be a group and let $X$ be a space on which $G$ acts. We say that $X$ is a torsor for $G$ (or that $X$ is a $G$-torsor) if for any two elements $x_1$ and $x_2$ of $X$
there exists a unique group element $g \in G$  with $g(x_1) = x_2$.
\end{remark}

By a mild abuse of notation, we let 
$$
\NB(\sqrt{\gamma_{i}})=\NB(\wt{\gamma}_{i})/\langle \sqrt{A_{\gamma_{i}}} \rangle.
$$
This is a torsor for 
$$
\C^{*}/\langle \sqrt{A_{\gamma_{i}}} \rangle=\C/ 2\pi i \Z+\hl(\gamma_{i}) \cdot \Z.
$$

For $i \ne j$, $i,j=0,1,2$, we let $n(i,j) \in \NB(\gamma_{i})$ be the unit vector at $\gamma_i \cap \eta_j$ pointing along $\eta_j$. Then $\sqrt{A_{\gamma_{i}}}$ 
interchanges $n(i,i-1)$ and $n(i,i+1)$, so we can think of the unordered pair $\{n(i,i-1),n(i,i+1)\}$ as an element of $\NB(\sqrt{\gamma_{i}})$. We call this element $\foot_{\gamma_{i}}(\rho)$ or 
$\foot_{\gamma_{i}}(f)$ where $f:\Pi^0 \to M_{\rho}$ is a map whose homotopy class is determined by $\rho$.

If $\rho:\pi_1(\Pi^0) \to \PSLC$ is a representation for which $\hl(C) \in \R^{+}$ for each $C \in \partial{\Pi}^0$, then, after conjugation, $\rho(\pi_1(\Pi^0)) \in \PSLR < \PSLC$, and
$\Ha /\rho(\pi_1(\Pi^0))$ is a topological pair of pants (homeomorphic to the interior of $\Pi^0$). Also the converse is true: if we are given  $\rho:\pi_1(\Pi^0) \to \PSLR$ and $\Ha / \rho(\pi_1(\Pi^0))$ is homeomorphic to the interior of $\Pi^0$, then $\hl(C) \in \R^+$ for each cuff $C \in \partial{\Pi}^0$.

Now suppose that $S^0$ is a closed surface (of genus at least 2), and $\Col^0$ a maximal set of simple closed curves on $S^0$ (the curves in $\Col^0$ are disjoint, non-isotopic and nontrivial). By $\Col^*$ we denote the set of oriented curves from $\Col^0$ (each curve is taken with both orientations).  A pair of pants $\Pi$ for $(S^0,\Col^0)$ is the closure of a component of $S^0 \setminus \bigcup \Col^0 $, and a marked pair of pants is a pair $(\Pi,C)$, where $C \in \Col^*$ is an oriented closed curve such that $C \in \partial{\Pi}$, and $C$ lies to the left of $\Pi$. 
For any marked pair of pants $(\Pi,C)$, there is a unique marked pair of pants $(\Pi',C')$ such that $C'=-C$ (where $-C$ denotes the curve $C$ but with the opposite orientation). We observe in passing that $\Pi$ can be equal to $\Pi'$. 

Now suppose that 
$$
\rho:\pi_1(S^0) \to \PSLC
$$
is a representation that is discrete and faithful when restricted to $\pi_1(\Pi)$, for each pair of pants $\Pi$ in  $S^0 \setminus \bigcup \Col^0 $. By $M_{\rho}$ we again denote the quotient $\Ho /\rho(\pi_1(S^0))$.  Suppose that $\rho=f_*$ for some continuous map $f:S^0 \to M_{\rho}$. Then for each marked pair of pants $(\Pi,C)$ we let $\gamma$ be the oriented geodesic freely homotopic to $f(C)$. As before,  we define $\hl_{\Pi}(\gamma)$  using $f|_{\Pi}$.

Let $(\Pi',C')$ be the marked pair of pants such that $C'=-C$.  Then $\hl_{\Pi}(C)=\hl_{\Pi'}(C)$, or  $\hl_{\Pi}(C)=\hl_{\Pi'}(C)+i\pi$. In the former case, $\langle \sqrt{A_{\gamma}} \rangle=\langle \sqrt{A_{\gamma'}} \rangle$, so $\NB(\sqrt{\gamma})=\NB(\sqrt{\gamma'})$ literally. In this case we write  $\hl(C)=\hl_{\Pi}(C)=\hl_{\Pi'}(C)$.

\begin{definition} Let $S^0$ and $\Col^0$ be as above. We say that a representation 
$$
\rho:\pi_1(S^0) \to \PSLC
$$
is viable if 
\begin{itemize}
\item  $\rho$ is discrete and faithful when restricted to $\pi_1(\Pi)$, for each pair of pants $\Pi$ in  $S^0 \setminus \bigcup \Col^0 $,
\item $\hl(C)=\hl_{\Pi}(C)=\hl_{\Pi'}(C)$, for each $C \in \Col^0$, where $\Pi$ and $\Pi'$ are two pairs of pants that contain $C$.
\end{itemize}
\end{definition}
 
Given a viable representation $\rho:\pi_1(S^0) \to \PSLC$, we let 
$$
s(C)=\foot_{\gamma}(\rho|_{\Pi})-\foot_{\gamma'}(\rho|_{\Pi'})-i\pi.
$$
Then $s(C) \in \C/ 2\pi i \Z+\hl(C) \cdot \Z$. If we reverse the roles of $(\Pi,C)$ and $(\Pi',C')$, we negate the difference of the two feet, but we also reverse the orientation of $\gamma$, so we get the same element  $s(C) \in \C/ 2\pi i \Z+\hl(C) \cdot \Z$.  The coordinates $(\hl(C),s(C))$ are called the reduced complex Fenchel-Nielsen coordinates for $\rho$.

The following is the main result of this section and it will be used later in the paper.

\begin{theorem}\label{geometry} Let  $0<\epsilon< \wh{\epsilon}$ where $\wh{\epsilon}>0$ is a universal constant. Then there exists $R_0=R_0(\epsilon)>0$ such that the following holds.
Let  $S^{0}$ be a closed topological surface with a pants decomposition $\Col^0$. Suppose that  $\rho:\pi_1(S^{0}) \to \PSLC$ is a viable representation such that 
$$
|\hl(C)-\frac{R}{2}|<\epsilon, \, \, \text{and} \,\, |s(C)-1|<{{\epsilon}\over{R}},
$$
\noindent
for some  $R>R_0>0$. Then there exists a viable representation $\rho_0:\pi_1(S^{0}) \to \PSLC$ such that $\hl(C)=R$ and $s(C)=1$ for all $C \in \Col^0$, and a $K$-quasisymmetric map
$h:\partial{\Ho} \to \partial{\Ho}$ so that $h^{-1}\rho_0(\pi_1(S^0)) h=\rho(\pi_1(S^0))$, where $K=K(\epsilon)$ and $K(\epsilon) \to 1$ uniformly when $\epsilon \to 0$.
In particular, the representation $\rho$ is injective and  the group $\rho(\pi_1(S^{0}) )$ is quasifuchsian.
\end{theorem}

\subsection{Holomorphic families of representations} In this subsection we state Theorem \ref{geometry-1} that will imply Theorem \ref{geometry}. The rest of Section 2 is devoted to proving  Theorem \ref{geometry-1}.

Fix a closed surface $S^0$ with a pants decomposition $\Col^0$. Fix a pair of pants $\Pi$ from $S^0 \setminus \Col^0$, and let $C_0,C_1,C_2 \in \Col^0$ denote the cuffs  of $\Pi$. The inclusion $\Pi \to S^0$ induces an embedding $\pi_1(\Pi) \to \pi_1(S^0)$ (such embedding is well defined up to conjugation). Let $c_0,c_1 \in \pi_1(\Pi) \subset \pi_1(S^0)$ be elements in the conjugacy classes corresponding to $C_0$ and $C_1$ respectively.

Let $\rho:\pi_1(S^0) \to \PSLC$ be a viable representation. After conjugating $\rho$ by an element of $\PSLC$, we may assume  that the axis of $\rho(c_0)$ is the geodesic in $\Ho$ that connects $0$ and $\infty$ (such that $0$ is the repelling point) and that the point $1 \in \partial{\Ho}$ is the repelling point of $\rho(c_1)$ (such a conjugation exists since $\rho$ is viable and the restriction of $\rho$ to $\pi_1(\Pi)$ is injective). Such $\rho$ is said to be normalized (the normalization depends on the choice of $c_0$ and $c_1$ but we suppress this).

Let $R>0$, and  we let $\Omega$ denote the set of all pairs $(z_C,w_C)$, $C \in \Col^0$, where for each $C$ we have 
\begin{enumerate}
\item $z_C \in \C / 2\pi i \Z$ and  $|z_C-\frac{R}{2}|<1$,
\item $w_C \in \C/ 2\pi i \Z+z_C \cdot \Z$ and $|s(C)-1|<{{1}\over{R}}$.
\end{enumerate}

For simplicity we let $z=(z_C)_{C \in \Col^{0}}$ and  $w=(w_C)_{C \in \Col^{0}}$. It follows  from \cite{kour} and \cite{series} that  when $R$ is large enough (say $R>2$),  for each $(z,w) \in \Omega$ there exists a normalized viable  representation $\rho:\pi_1(S^{0}) \to \PSLC$ such that $\hl(C)=z_C$ and $s(C)=w_C$. 

\begin{remark} Such a normalized representation $\rho$ is not unique  since $(\hl(C),s(C))$ are the reduced complex Fenchel-Nielsen coordinates and they determine the normalized representation only if we specify the marking of the cuffs (that is, a normalized viable representation is uniquely determined by the choice of the (non-reduced) Fenchel-Nielsen coordinates). 
\end{remark}

Suppose that we are given a normalized viable representation  $\rho':\pi_1(S^{0}) \to \PSLC$ such that  $|\hl(C)-\frac{R}{2}|<1$ and  $|s(C)-1|<{{1}\over{R}}$, where $(\hl(C),s(C))$ are the reduced complex Fenchel-Nielsen coordinates for $\rho'$. Let $z'_C=\hl(C)$ and $w'_C=s(C)$.  Then $(z',w') \in \Omega$. It then follows from  \cite{kour} and \cite{series} that for each $(z,w) \in \Omega$, there exists a unique normalized viable representation  $\rho_{z,w}:\pi_1(S^{0}) \to \PSLC$ such that 
\begin{itemize}
\item  $z_C=\hl(C)$ and $w_C=s(C)$, where $(\hl(C),s(C))$ are the reduced complex Fenchel-Nielsen coordinates for $\rho_{z,w}$,
\item The family of representations $\rho_{z,w}$ varies holomorphically in $(z,w)$,
\item $\rho'=\rho_{z',w'}$.
\end{itemize}

\begin{definition}\label{def-C}
For $C \in \Col^0$ let $\zeta_C, \eta_C \in \D$, where $\D$ denotes the unit disc in the complex plane.  Let $\tau \in \D$ be a complex parameter. Fix $R>1$ and let $\hl(C)(\tau)=\frac{1}{2}(R+ \tau \zeta_C)$ and  $s(C)(\tau)=1+{{\tau \eta_C } \over{R}}$. By $\rho_{\tau}$ we denote the corresponding normalized viable representation with the reduced Fenchel-Nielsen coordinates $(\hl(C)(\tau),s(C)(\tau))$. 
Note that $\rho_{\tau}$ depends on $\zeta_C,\eta_C$ but we suppress this. 
\end{definition} 

It follows that $\rho_{\tau}$ depends holomorphically on $\tau$. The remainder of this section is devoted to proving the following theorem.

\begin{theorem} \label{geometry-1} There exist constants $\wh{R},\wh{\epsilon}>0$, such that the following holds. Let $S^{0}$ be any closed topological surface with a pants decomposition $\Col^0$ and fix $\zeta_C, \eta_C \in \D$ for $C \in \Col^{0}$.  Then for every $R \ge \wh{R}$ and $|\tau| <\wh{\epsilon}$, the group $\rho_{\tau} (\pi_1(S^{0}))$  is quasifuchsian and the induced quasisymmetric map  $f_{\tau}:\partial \Ha \to \partial \Ho$ (that conjugates $\rho_{0}(\pi_1(S^{0}))$ to  $\rho_{\tau}(\pi_1(S^{0}))$)  is  $K(\tau)$-quasisymmetric, where 
$$
K(\tau)={{\wh{\epsilon} +|\tau|}\over{\wh{\epsilon}-|\tau|}}.
$$
\end{theorem}

\begin{figure}
	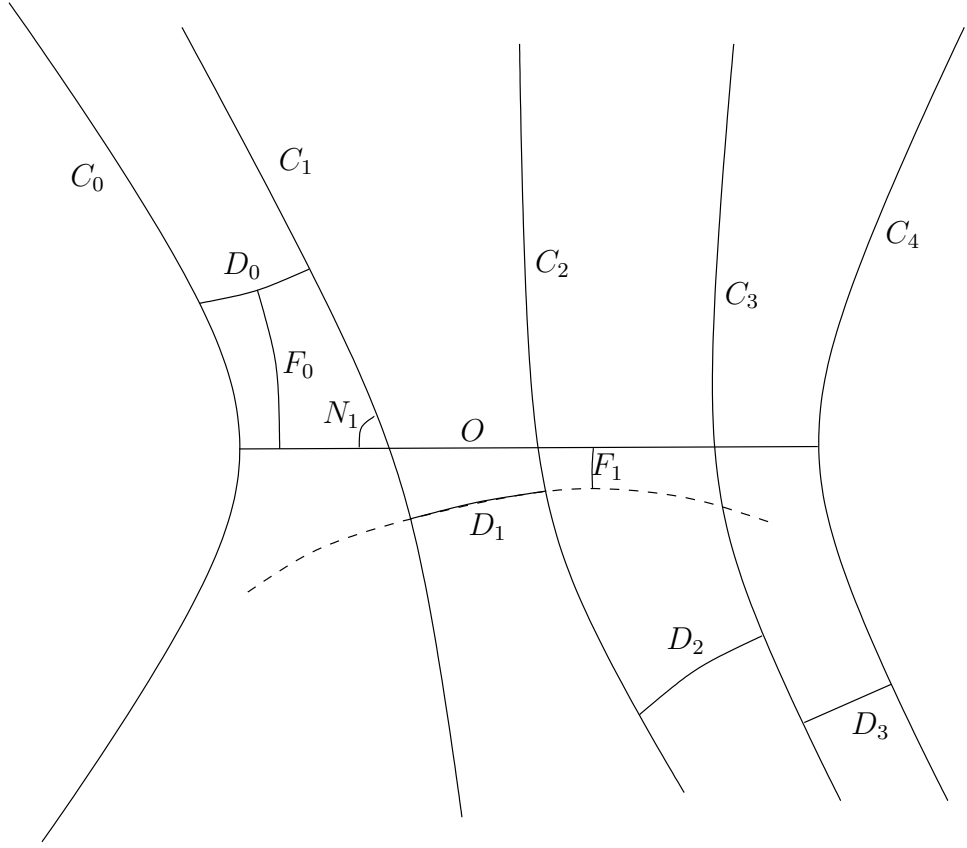
	\label{strans}
	\caption{The geodesics $O$, $C_i$, $N_i$, $F_i$, and $D_i$}
\end{figure}

\subsection{Notation and the brief outline of the proof of Theorem \ref{geometry-1}} The following notation remains valid through the section. Fix $S^{0}$, $\Col^0$ and $\zeta_C,\eta_C \in \D$  as above.  Denote  by $\Col_{\tau}(R)$ the collection of translation axes in $\Ho$ of all the elements $\rho_{\tau}(c)$, where $c \in \pi_1(S^{0})$ is in the conjugacy class of some cuff $C \in \Col^{0}$. Fix two such axes $C(\tau)$ and $\wh{C}(\tau)$ and let $O(\tau)$ be their common orthogonal in $\Ho$.  Since  $C(\tau)$ and $\wh{C}(\tau)$ vary holomorphically in $\tau$ so does $O(\tau)$ (this means that the endpoints of $O(\tau)$ vary holomorphically on $\partial \Ho$). Note that the endpoints of $O(\tau)$ might 
not belong to the limit set of the group $\rho_{\tau}(\pi_1(S^{0}))$.
\vskip .1cm
Let $C_0(0),C_1(0),...,C_{n+1}(0)$ be the ordered collection of geodesics from $\Col_0(R)$ that $O(0)$ intersects (and in this order) and so that $C_0(0)=C(0)$ and $C_{n+1}(0)=\wh{C}(0)$. The geodesic segment on $O(0)$ between $C_0(0)$ and $C_{n+1}(0)$ intersects $n \ge 0$ other geodesics from $\Col_0(R)$ (until the end of this section $n$ will have the same meaning).  We orient $O(0)$ so that it goes from  $C_0(0)$ to $C_{n+1}(0)$. We orient each  $C_i(0)$ so that the angle from $O(0)$ to $C_i(0)$ is positive (recall that we fix in advance an orientation on the initial plane $\Ha \subset \Ho$ so this angle is positive with respect to this orientation of the plane $\Ha$). Then the oriented geodesics $C_i(\tau)$ vary holomorphically in $\tau$.
\vskip .1cm 
Let $N_i(\tau)$ be the common orthogonal between $O(\tau)$ and $C_i(\tau)$ that is oriented so that  the imaginary part of the complex distance 
$\dis_{N_{i}(\tau)}(O(\tau),C_i(\tau))$ is positive. Let $D_i(\tau)$, $i=0,...,n$ be the common orthogonal between $C_i(\tau)$ and $C_{i+1}(\tau)$, that is oriented so that the angle from $D_i(0)$ to $C_i(0)$ is positive. Also, let $F_i(\tau)$ be the common orthogonal between $O(\tau)$ and $D_i(\tau)$, for $i=0,...,n$. We orient $F_i(\tau)$ so that the angle from $O(0)$ to $F_i(0)$  is positive. Observe that $F_0(\tau)=C_0(\tau)$ and $F_{n}(\tau)=C_{n+1}(\tau)$.
\vskip .1cm
For simplicity, in the rest of this section  we suppress the dependence on $\tau$, that is we write $C_i(\tau)=C_i$, $O(\tau)=O$ and so on. 
However, we still write $C_i(0)$, $O(0)$, to distinguish the case $\tau=0$.
\vskip .1cm
For Theorem \ref{geometry-1} we need to estimate the quasisymmetric constant of the map $f_{\tau}$, when $\tau$ belongs to some small, but definite neighbourhood of the origin in $\D$.   In order to do that we want to estimate the derivative (with respect to $\tau$) 
$\dis'_{O} (C_0,C_{n+1})$ of the complex distance   $\dis_{O} (C_0,C_{n+1})$ between any two geodesics $C_0,C_{n+1} \in \Col_{\tau}(R)$. We will 
compute an upper bound of $|\dis'_{O} (C_0,C_{n+1})|$ in terms of $\dis_{O} (C_0,C_{n+1})$.
This will lead to an inductive type of argument that will finish the proof. We will offer more explanations as we go along.

\subsection{The Kerckhoff-Series-Wolpert type formula} In \cite{series} C. Series has derived  the formula for the derivative of the complex translation length of a (not necessarily simple) closed  curve on $S^{0}$ under the representation $\rho_{\tau}$. Using the same method (word by word) one can obtain the appropriate formula for the derivative of the complex distance $\dis_{O(\tau)}(C_0(\tau),C_{n+1}(\tau))$.

\begin{theorem}\label{thm-ksw} Letting $'$ denote the derivative with respect to $\tau$ we have

\begin{align}
\dis'_{O} (C_0,C_{n+1}) &= \sum_{i=0}^{n} \cosh( \dis_{F_{i}}(O,D_i))  \dis'_{D_i}(C_i,C_{i+1})+  \label{S-formula} \\
&+ \sum_{i=1}^{n} \cosh( \dis_{N_{i}}(O,C_i)) \dis'_{C_i}(D_{i-1},D_i) \notag.
\end{align}

\end{theorem}

\begin{proof} For each $i=1,...,n$ consider the skew right-angled hexagon with sides $O,F_{i},D_i,C_i,D_{i-1},F_{i-1}$. Since each hexagon varies holomorphically in $\tau$ we have the following derivative formula in each hexagon (this is the formula $(7)$ in \cite{series})

\begin{align}
\dis'_{O}(F_{i-1},F_i) &= \cosh(\dis_{N_{i}}(O,C_i)) \dis'_{C_i}(D_{i-1},D_i)+ \notag \\
&+ \cosh(\dis_{F_{i-1}}(O,D_{i-1}) ) \dis'_{D_{i-1}}(F_{i-1},C_{i}) \label{S-formula-1} \\
& + \cosh(\dis_{F_{i}}(O,D_{i})) \dis'_{D_{i}}(C_{i}, F_i) \notag .
\end{align}
\noindent
The following relations (\ref{S-formula-2}), (\ref{S-formula-3}) and  (\ref{S-formula-4}) are direct corollaries of the identities 
$F_0=C_0$ and $F_{n}=C_{n+1}$. We have 
\begin{equation}\label{S-formula-2}
\sum_{i=1}^{n} \dis_{O}(F_{i-1},F_i)=\dis_{O}(C_0,C_{n+1}).
\end{equation}
\noindent
Also
\begin{equation}\label{S-formula-3}
\dis'_{D_{0}}(F_{0},C_{1})= \dis'_{D_{0}}(C_{0},C_{1}),
\end{equation}
\noindent
and 
\begin{equation}\label{S-formula-4} 
\dis'_{D_{n}}(C_{n}, F_n)=\dis'_{D_{n}}(C_{n}, C_{n+1}).
\end{equation}
\noindent
Also for $1=1,...,n$ we observe the identity 
$$
\dis_{D_{i}}(C_{i}, F_{i})+\dis_{D_{i}}(F_{i},C_{i+1})=
\dis_{D_{i}}(C_{i},C_{i+1}).
$$
\noindent

Putting all this together and summing up the formulae (\ref{S-formula-1}), for $i=1,...,n$  we obtain (\ref{S-formula}). 

\end{proof}

Let $H$ be a consistently oriented skew right-angled hexagon with sides $L_k$, $k \in \Z$, and $L_k=L_{k+6}$. Set 
$\sigma(k)=\dis_{L_{k} } (L_{k-1},L_{k+1})$. Recall the cosine formula
$$
\cosh(\sigma(k))={{ \cosh(\sigma(k+3)) -\cosh(\sigma(k+1)) \cosh(\sigma(k-1)) } \over{\sinh(\sigma(k+1)) \sinh(\sigma(k-1)) }}.
$$
\noindent
Assume that $\sigma(2j+1)={{1}\over{2}}(R+a_{2j+1})+i\pi$, $j=0,1,2$, and $a_{2j+1} \in \D$.
A hexagon with this property is called a $thin$ $hexagon$. 
From the cosine formula for a skew right-angled hexagon we have (see also Lemma 5.1 in \cite{bowen})  
\begin{equation}\label{hex-1}
\sigma(2j)=2e^{ {{1}\over{4}}[-R+a_{2j+3}-a_{2j+1}-a_{2j-1}] }+i\pi+O(e^{ -{{3R}\over{4}} }).
\end{equation}
\noindent
From the pentagon formula the hyperbolic distance between opposite sides in the hexagon can be estimated as (see Lemma 5.4 in \cite{bowen} and Lemma 2.1 in \cite{series})
\begin{equation}\label{hex-2}
{{R}\over{4}} -10 <d(L_k,L_{k+3})<{{R}\over{4}}+10,
\end{equation}
\noindent
for $R$ large enough.

\begin{lemma}\label{derivative-est-0} Suppose that  $|\dis_{O} (C_0,C_{n+1})|<{{R}\over{5}}$. Then for $R$ large enough the following estimate holds
$$
|\dis'_{O} (C_0,C_{n+1})| \le 20e^{ -{{R}\over{4}} } \sum_{i=0}^{n} e^{d(O,D_i) }+
{{n}\over{R}} \left( \max_{1 \le i \le n}  e^{d(O,C_i)} \right).
$$

\end{lemma}

\begin{proof} Let $\gamma$ be the geodesic segment on $O(0)$ that runs between $C_j(0)$ and $C_{j+1}(0)$. Then $\gamma$ is a lift of a geodesic  arc connecting two cuffs in the  pair of pants whose all three cuffs have  length $R$. Since the length of $\gamma$ is at most ${{R}\over{5}}$ we have from (\ref{hex-2}) that $\gamma$ connects two different cuffs in this pair of pants and is freely homotopic to the shortest orthogonal arc between these two cuffs in this pair of pants. This implies that there exists $\wt{C} \in \Col(R)$ such that the hexagon determined by $C_j,C_{j+1}$ and $\wt{C}$ is a thin hexagon. Then $D_i$ is a side of this hexagon since it is the common orthogonal for $C_i$ and $C_{i+1}$.  Taking into account that the orientation of $D_i$ that comes from this hexagon is the opposite to the one we defined above in terms of  $O$ and applying (\ref{hex-1}) we obtain

\begin{equation}\label{cuff-distance}
\dis_{D_j}(C_j,C_{j+1})=2e^{{{1}\over{4}}[-R+  \wt{\zeta} \tau  -  \zeta_j \tau - \zeta_{j+1} \tau ] }+O(e^{ -{{3R}\over{4}} }).
\end{equation}
\noindent
where $\zeta_j,\zeta_{j+1},\wt{\zeta} \in \D$ are the complex numbers associated to the corresponding $C \in \Col^{0}$ in Definition \ref{def-C}. Differentiating the cosine formula for the skew right-angled hexagon we get

$$
|\dis'_{D_j}(C_j,C_{j+1})|<20 e^{-{{R}\over{4} } }.
$$
\noindent
for $R$ large enough (here we use $|\zeta_C|, |\tau|<1$).
\vskip .1cm
On the other hand  $\dis_{C_j}(D_{j-1},D_j)=1+{{\tau \eta_{j-1}}\over{R}}$, where $\eta_{j-1} \in \D$ is the corresponding number. 
Differentiating this identity gives $|\dis'_{C_j}(D_{j-1},D_j ) | \le {{1}\over{R}}$ (we use $|\eta_{j-1}|<1$). 
Combining these estimates with the equality of Theorem \ref{thm-ksw} proves the lemma.

\end{proof}

\subsection{Preliminary estimates}  The purpose of the next two subsections is to estimate the two terms on the right-hand side of the inequality of Lemma 
\ref {derivative-est-0} in terms of the complex distance $\dis_{O} (C_0,C_{n+1})$. We will show that
$$
|\dis'_{O} (C_0,C_{n+1})| \le C F(\dis_{O} (C_0,C_{n+1})),
$$
\noindent
where $C$ is a constant and $F$ is the function $F(x)=xe^{x}$. We will obtain this estimate under some natural assumptions (see Assumption \ref{ass} below).
\vskip .1cm
Let $\alpha, \beta$ be two oriented geodesics in $\Ho$  such that $d(\alpha,\beta)>0$ and let $O$ be their common orthogonal (with either orientation). Let $q_0=\beta \cap O$ . Let $t \in \R$ and let $q:\R \to \beta$ be the   parametrisation by arc length such 
that $q(0)=q_0$.  The following trigonometric formula follows directly from the $\cosh$ and $\sinh$ rules for  right angled triangles in the hyperbolic plane
(the planar case of this formula was stated in Lemma 2.4.7 in \cite{epstein}) 
\begin{equation}\label{epstein}
\sinh^{2}(d(q(t),\alpha))=\sinh^{2}(d(\alpha,\beta) ) \cosh^{2}(t)+\sinh^{2}(t)\sin^{2}(\IM[\dis_{O}(\alpha,\beta)]).
\end{equation}
\noindent
This yields the following inequality that will suffice for us
\begin{equation}\label{epstein-0}
\sinh(d(\alpha,\beta) ) \cosh(t) \le \sinh(d(q(t),\alpha)).
\end{equation}

From this we derive:
\begin{equation}\label{epstein-1}
|t| \le d(q(t), \alpha) - \log d(\alpha, \beta),\,\, \text{for every}\,\, t \in \R,
\end{equation}
\noindent
and
\begin{equation}\label{epstein-2} 
|t| \le \log d(q(t), \alpha) + 1 - \log d(\alpha, \beta),  \,\, \text{when} \,\,  d(q(t),\alpha)  \le 1 .
\end{equation}
 
 \begin{figure}
	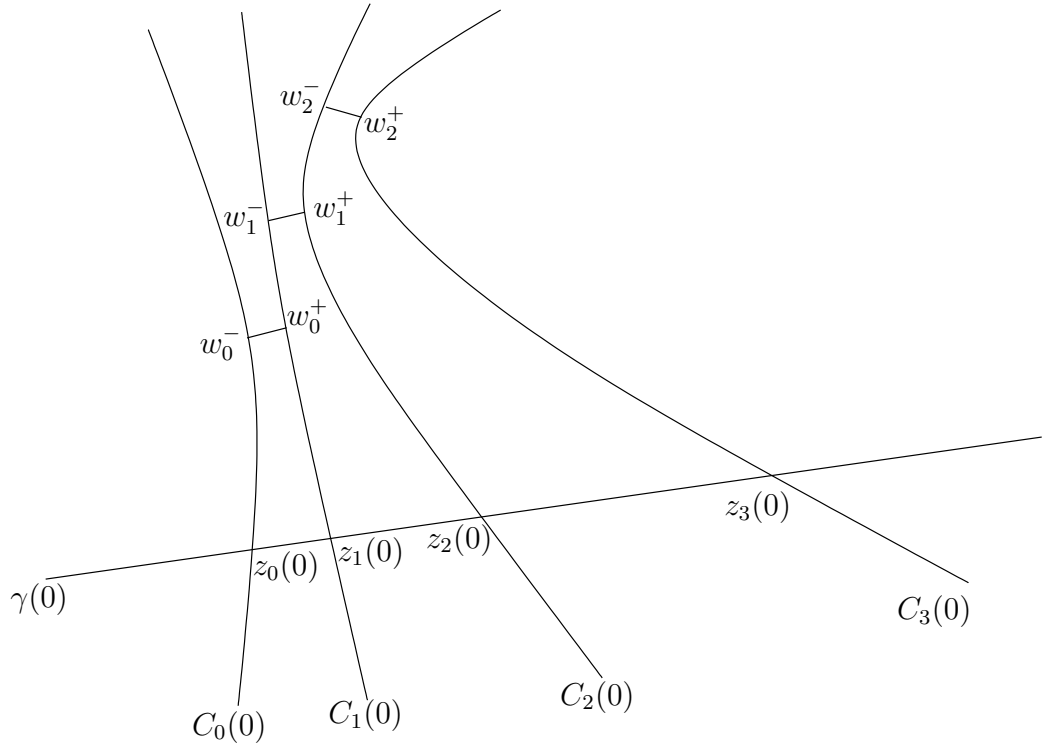
	\caption{The $z$'s and the $w$'s}
  	\label{fig:transversal}
\end{figure}
\vskip .1cm
Let $\gamma=\gamma(\tau)$, $\tau \in \D$, be an oriented geodesic  in $\Ho$ that varies continuously in $\tau$, and such that $\gamma(0)$ belongs to the plane $\Ha \subset \Ho$ that contains the lamination $\Col_{0}(R)$ (the common orthogonal  $O$ from the previous subsection is an example of $\gamma$ but there is no need to restrict ourselves to $O$ in order to prove the estimates below). 
Let $C_1(0),...,C_k(0)$ be an ordered subset of geodesics from $\Col_0(R)$ that $\gamma(0)$ consecutively intersects (this means that the segment of $\gamma(0)$ between $C_i(0)$ and $C_{i+1}(0)$ does not intersect any other geodesic from  $\Col_0(R)$). 
Orient each $C_i$ so that the angle from $\gamma(0)$ to $C_i(0)$ is positive. Let $N_i$ be the common orthogonal between $\gamma$ and $C_i$, and let $z_i=N_i \cap C_i$ and  $z'_i=N_i \cap \gamma$. (See Figure \ref{fig:transversal}). Let $D_i$, $i=1,...,k$  be the common orthogonal between $C_i$ and $C_{i+1}$ and let $w^{-}_i=D_i \cap C_i$ and $w^{+}_i=D_i \cap C_{i+1}$. As long as the distance between $z_i$ and $z_{i+1}$ is at most ${{R}\over{5}}$, then (as seen in the previous subsection)  for $R$ large enough we have 

\begin{equation}\label{cuff-distance-1}
\dis_{D_{i} }(C_i,C_{i+1})=(2+o(1))e^{-{{R}\over{4}}+ \tau \mu} \le e^{-{{R}\over{4}}+2}, 
\end{equation} 
\noindent
where $\mu \in \C$ and $|\mu|\le {{3}\over{4}}$ (see (\ref{cuff-distance}) ). Then it follows from the definition of $\Col_{\tau}(R)$ that 

\begin{equation}\label{twist}
\dis_{C_{i}}(w^{+}_{i-1},w^{-}_i)=1+ \RE[ {{\tau \eta}\over{R}} +j{{ (R+ \tau \zeta)}\over{2}} ],
\end{equation}
\noindent
for some $j \in \Z$, where $\eta=\eta_C$ and $\zeta=\zeta_C$ are the complex numbers from the unit disc that correspond to the cuff in 
$C \in \Col^{0}$ whose lift is $C_i(0)$. Here $\dis_{C_{i}}(w^{+}_{i-1},w^{-}_i)$ denotes the signed hyperbolic distance.

\begin{lemma}\label{lemma-est-1} Assume that $d(z_i,z_{i+1})< e^{-5}$, for  $i=1,...,k-1$. Set  $a_i=\dis_{C_{i}}(z_i,w^{-}_i)$. Then for $R$ large enough the following inequalities hold
\begin{enumerate}
\item $a_{i+1}-a_i<1+e^{-1}$, $i=1,...,k-2$ 
\item $k < R$.
\end{enumerate}

\end{lemma}

\begin{proof}   Since the distance  between each pair $z_i$ and $z_{i+1}$ is at most $e^{-5}$, applying (\ref{epstein-2}) and (\ref{cuff-distance-1}) to all pairs $\alpha=C_i$ and $\beta=C_{i+1}$ yields the inequality
\begin{equation}\label{est-1}
d(z_i,w^{+}_{i-1}),  d(z_i,w^{-}_{i}) \le {{R}\over{4}}-2,
\end{equation}
\noindent
for each $i=1,...,k-1$.  By the triangle inequality we have
\begin{equation}\label{pm} 
\dis_{C_{i}}(w^{+}_{i-1},w^{-}_i) \le {{R}\over{2}}-4.
\end{equation}
\noindent
On the other hand, from  (\ref{twist}) we obtain
$$
|j|(1-{{|\tau \zeta|}\over{R}} ) \le {{2}\over{R}} (  \dis_{C_{i}}(w^{+}_{i-1},w^{-}_i) +1+{{|\tau \eta|}\over{R}} ) \le {{2}\over{R}} ({{R}\over{2}}-4+2).
$$
\noindent 
Since $|\tau|,|\zeta|,|\eta|<1$ and from (\ref{pm}) we get
$$
|j|\le {{1-{{4}\over{R}} }\over{1-{{1}\over{R}} }},
$$
\noindent
which shows that $j=0$ in (\ref{twist}).
\vskip .1cm
From (\ref{est-1}) we have 
\begin{equation}\label{est-2}
|a_i| < {{R}\over{4}}.
\end{equation}
\noindent
We write (using the triangle inequality) 
$$
a_{i+1} - a_i - 1 \le d(w_i^-, w_i^+) + d(z_i, z_{i+1} ) + |d(w_i^+, w_{i+1}^-) - 1|.
$$
\noindent
By (\ref{cuff-distance-1}) we have
$$ 
d(w_i^{-} , w_i^+) \le  e^{-{{R}\over{4}} + 2}.
$$
\noindent
The assumption of the lemma is $d(z_i, z_{i+1}) \le e^{-5}$. It follows from (\ref{twist}) (and the established fact that in this case $j=0$) 
that
$$ 
| d(w_i^+, w_{i+1}^-) -1 | \le | \RE({{\tau \eta}\over{R}} )  | \le {{1}\over{R}}.
$$
\noindent
Therefore 
$$
a_{i+1}-a_i -1 < e^{-1},
$$
\noindent
which proves the first part of the lemma.
\vskip .1cm
From (\ref{est-2}) we have $-{{R}\over{4}}< a_1$, which implies that $a_{k-1}>(k-1)(1-e^{-1})-{{R}\over{4}}$. Again from (\ref{est-2}) we have $a_{k-1}<{{R}\over{4}}$, which proves 
$$
k<{{R}\over{2(1-e^{-1}) } }+1<R.
$$
\end{proof}

The following lemma is a corollary of the previous one.

\begin{lemma}\label{number} Let $\gamma$ be a geodesic segment in $\Ha$ that is transverse to the lamination $\Col_0(R)$. For $R$ large enough, the number of geodesics from $\Col_0(R)$ that $\gamma$ intersects is at most $(2+R)e^{5}|\gamma|$.

\end{lemma}

\begin{proof} As above denote by $C_i(0)$, $i=1,...,k$, the geodesics from $\Col_0(R)$ that $\gamma$ intersects. Using the above notation, let $j_1,...,j_l \in \{0,..,k\}$, be such that $d(z_{j_{i}}(0),z_{j_{i}+1}(0))>  e^{-5}$. Then 
$$
l<{{|\gamma|}\over{e^{-5}} }=e^{5}|\gamma|.
$$
\noindent
By definition, the open segment between $z_{j_{i}}(0)$ and $z_{j_{i}+1} (0)$ does not intersect any geodesics from $\Col_0(R)$.   
\vskip .1cm
On the other hand, by the previous lemma the number of geodesics from $\Col_0(R)$ that the subsegment of $\gamma$ between $z_{j_{i}+1 } (0)$ and $z_{j_{i+1} } (0)$ intersects is at most $R$ (because the distance between any $z_i(0)$ and $z_{i+1}(0)$ in this range is at most $e^{-5}$). 
Since there are at most $l$ such segments we have that the total number of geodesics from $\Col_0(R)$ that $\gamma$ intersects is at most 
$2l+lR<(2+R) e^{5}|\gamma|$.
\end{proof}

\subsection{Estimating the derivative $|\dis'_{O} (C_0,C_{n+1})|$}
We now combine the notation of the previous two subsections (and set $\gamma=O$). In the following  lemmas we prove estimates for the two terms on the right-hand side in the inequality of Lemma \ref{derivative-est-0}, that are independent of $R$.
\vskip .1cm
We first estimate the second term in  the inequality of Lemma \ref{derivative-est-0}.

\begin{lemma}\label{derivative-est-1} We have 

$$
{{n}\over{R}} \left( \max_{1 \le i \le n}  e^{d(O,C_i)} \right) \le 1000
d(C_0(0),C_{n+1}(0)) \left( \max_{1 \le i \le n}  e^{d(O,C_i)} \right),
$$
\noindent
where $n$ is the number of geodesics that $O(0)$ intersects between $C_0(0)$ and $C_{n+1}(0)$.
\end{lemma}

\begin{proof} From Lemma \ref{number} we have 
$$
n \le (2+R)e^{5}d(C_0(0),C_{n+1}(0))<1000 R d(C_0(0),C_{n+1}(0)),
$$
\noindent
which proves the lemma.

\end{proof}

We now bound the first term in the inequality of Lemma \ref{derivative-est-0} under the following assumption.

\begin{assumption}\label{ass} Assume that for some $\tau \in \D$ the following estimates hold for $i=0,...,n+1$,
$$
d(z_i,z_{i}(0)),\, d(O,C_i)< {{1}\over{4}} e^{-5}. 
$$
\end{assumption}

We have

\begin{lemma}\label{derivative-est-2}  Under  Assumption \ref{ass} and for $R$ large enough we have
$$
20e^{ -{{R}\over{4}} }\sum_{i=0}^{n+1} e^{d(O,D_i) } \le  10^{8}d(C_0(0),C_{n+1}(0) ) e^{d(C_0(0),C_{n+1}(0)) }.
$$

\end{lemma}

\begin{proof} Recall $z'_i=N_{i} \cap O$ (note $z_0=z'_0$ and $z_{n+1}=z'_{n+1}$ since $O$ is the common orthogonal between $C_0$ and $C_{n+1}$).  Observe

\begin{equation}\label{distance-1}
d(O,D_i) \le d(z'_i,z_i)+  d(z_i,w^{-}_i) =d(O,C_i)+ |a_i|<1+|a_i|.
\end{equation}
\noindent
It follows from  (\ref{epstein-1}) that 
$$
|a_i|=d(z_i,w^{-}_i) \le d(z_i,C_{i+1}) - \log d(C_{i},C_{i+1})
$$
We observe the estimate $d(z_i,C_{i+1}) \le d(z_i,z_{i+1})$. On the other hand, by  (\ref {cuff-distance-1}) we have 
$$
\dis_{D_{i} }(C_i,C_{i+1})=(2+o(1))e^{-{{R}\over{4}}+ \tau \mu},
$$
so for $R$ large enough (such that $|o(1)|<1$) we find that (using the estimate $|\tau \mu|<1$)
$$
d(C_i,C_{i+1}) \ge e^{-\frac{R}{4}-1},
$$
that is $-\log d(C_{i},C_{i+1}) \le \frac{R}{4}+1$. It follows that 
$$
|a_i| \le d(z_i,z_{i+1}) + {{R}\over{4}}+1.
$$
\noindent
From
\begin{equation}\label{est-z}
|d(z_i,z_{i+1})-d(z_i(0),z_{i+1}(0))| \le d(z_i,z_{i}(0))+d(z_{i+1},z_{i+1}(0))  \le {{e^{-5}}\over{2}}
\end{equation}
\noindent
and  $d(z_i(0),z_{i+1}(0)) \le  d(C_{0}(0),C_{n+1}(0))$, we obtain

\begin{equation}\label{a-est}
|a_i|<{{R}\over{4}}+d(C_0(0),C_{n+1}(0) ) +2.
\end{equation}

\vskip .1cm
Let $j_1,...,j_{l} \in \{1,..,n-1 \}$, be such that $d(z_{j_{i}},z_{j_{i}+1})> e^{-5}$ (note that $l=l(\tau)$ depends on $\tau$). Set $j_0=0$ and $j_{l+1}=n$. 
From (\ref{est-z}) we have $d(z_{j_{i}}(0),z_{j_{i}+1}(0))> {{e^{-5}}\over{2}}$ for each $1 \le i \le l$. 
The intervals $(z_i(0),z_{i+1}(0))$ partition the arc between $z_0(0)$ and $z_{n+1}(0)$ so we get

\begin{equation}\label{distance-2}
l<{{ d(C_0(0),C_{n+1}(0))  }\over{{{e^{-5}}\over{2}} }}=2 e^{5}  d(C_0(0),C_{n+1}(0)).
\end{equation}
\vskip .1cm
Let $0 \le i \le l+1$. For $j_{i}+1 \le t <  j_{i+1}$ we have $d(z_{t},z_{t+1}) \le e^{-5}$. It follows from Lemma \ref{lemma-est-1} that 
$$
{{1}\over{2}}<a_{t+1}-a_t. 
$$
\noindent
We see that in  this interval the sequence $a_t$ is an increasing sequence.  
Combining this with (\ref{a-est}) and (\ref{distance-1}) we obtain
\begin{align}
\sum_{t= j_{i}+1}^{j_{i+1}} e^{d(O,D_t) } &\le 2e^{ {{R}\over{4}} + d(C_0(0),C_{n+1}(0)) +3 } \sum_{t=0}^{\infty} e^{-{{t}\over{2}} }
\label{est-mo}  \\
&<200e^{  {{R}\over{4}} + d( C_0(0),C_{n+1}(0) )  } \notag. 
\end{align}
\noindent
We have
$$
\sum_{i=0}^{n+1} e^{d(O,D_i) }  \le    (l+1) \max_{i=0,...,n+1} e^{d(O,D_i) }+
\sum_{i=0}^{l+1} \sum_{t= j_{i}+1}^{j_{i+1}} e^{d(O,D_t) } 
$$
By  (\ref{distance-1}),  (\ref{a-est}) we have  
$$ 
e^{d(O,D_i)} \le e^{ {{R}\over{4}}+d(C_0(0),C_{n+1}(0) ) +2 }.
$$
\noindent
Also, by   (\ref{distance-2}) and (\ref{est-mo}) we have 
\begin{align*}
\sum_{i=0}^{l+1} \sum_{t= j_{i}+1}^{j_{i+1}} e^{d(O,D_t) } &\le \big( 2 e^{5} d(C_0(0),C_{n+1}(0))+1 \big) \times 200e^{  {{R}\over{4}} + d( C_0(0),C_{n+1}(0) ) } \\
&<
10^{6} d(C_0(0),C_{n+1}(0))e^{  {{R}\over{4}} + d( C_0(0),C_{n+1}(0) ) }
\end{align*}

Combining all this gives
$$
20e^{ -{{R}\over{4}} }\sum_{i=0}^{n+1} e^{d(O,D_i) } \le   10^{8} d(C_0(0),C_{n+1}(0)) e^{d(C_0(0),C_{n+1}(0)) }.
$$

\end{proof}

The previous two lemmas together with Lemma \ref{derivative-est-0} imply

\begin{lemma} \label{derivative-3} Under Assumption \ref{ass} and assuming that $d(C_0,C_{n+1})<\frac{R}{5}$,  for $R$ large enough we have
$$
|\dis'_{O} (C_0,C_{n+1})| <  10^{9} d(C_0(0),C_{n+1}(0)) e^{d(C_0(0),C_{n+1}(0)) }.
$$
\end{lemma}

\begin{proof} By Lemma \ref{derivative-est-0}   the estimate 
$$
|\dis'_{O} (C_0,C_{n+1})| \le 20e^{ -{{R}\over{4}} } \sum_{i=0}^{n} e^{d(O,D_i) }+
{{n}\over{R}} \left( \max_{1 \le i \le n}  e^{d(O,C_i)} \right)
$$
holds for $R$ large enough (recall that  $n$ is the number of geodesics that $O(0)$ intersects between $C_0(0)$ and $C_{n+1}(0)$). 
By Lemma \ref{derivative-est-1} we have 
$$
{{n}\over{R}} \left( \max_{1 \le i \le n}  e^{d(O,C_i)} \right) \le 1000
d(C_0(0),C_{n+1}(0)) \left( \max_{1 \le i \le n}  e^{d(O,C_i)} \right).
$$
By Assumption \ref{ass} we have that 
$$
d(O,C_i) \le  {{1}\over{4}}  e^{-5},
$$
for every $0\le i \le n+1$, so we obtain
$$
{{n}\over{R}} \left( \max_{1 \le i \le n}  e^{d(O,C_i)} \right) \le 3000
d(C_0(0),C_{n+1}(0)).
$$

On the other hand, by Lemma \ref{derivative-est-2} we have
$$
20e^{ -{{R}\over{4}} }\sum_{i=0}^{n+1} e^{d(O,D_i) } \le  10^{8}d(C_0(0),C_{n+1}(0) ) e^{d(C_0(0),C_{n+1}(0)) }.
$$
Putting the above estimates together proves the lemma.

\end{proof}

\subsection{The family of surfaces $\Su(R)$}  
We will consider geodesic laminations on a closed hyperbolic surface, and on its universal cover, the hyperbolic plane, which we will identify with the unit disk. By recording the endpoints of the leaves of a lamination of the unit disk, we can think of the lamination as a symmetric subset of $\partial\D \times \partial\D$, and by adding the diagonal, we obtain a closed subset of $\partial \D \times \partial \D$. The Hausdorff topology on such closed subsets will give us what we will call the Hausdorff topology on geodesic laminations of the unit disk. 

We have
\begin{definition} Let $R>1$ and let $P(R)$ be the pair of pants whose all three cuffs have the length $R$. We define the surface $\Su(R)$ to be the genus two surface that is obtained by gluing two copies of $P(R)$ alongside the cuffs with the twist parameter equal to $+1$ (these are the Fenchel-Nielsen coordinates for $\Su(R)$). The surface $\Su(R)$  can also be obtained by first doubling $P(R)$ and then applying the right earthquake of length $1$, where the lamination that supports the earthquake is the union of the three cuffs of $P(R)$. 
\end{definition}
By $\Orb(R)$ we denote the quotient orbifold  of the surface $\Su(R)$ (the quotient of $\Su(R)$ by the group of automorphisms of $\Su(R)$).
Observe that the Riemann surface $\Ha / \rho_{0}(\pi_1(S^{0}))$ is a regular finite degree cover of the orbifold $\Orb(R)$. In particular there exists a Fuchsian group $G(R)$ such that $\Orb(R)=\Ha / G(R)$ and  that $\rho_{0}(\pi_1(S^{0}))<G(R)$ is a finite index subgroup. It is important to point out that the lamination $\Col_0(R)$ is invariant under the group $G(R)$. In fact, one can define the group $G(R)$  as the group of all elements of $\PSLR$ that leave invariant the lamination $\Col_0(R) \subset \Ha$. Observe that the group $G(R)$ acts transitively on the geodesics from  $\Col_0(R)$, that is the $G(R)$-orbit of a geodesic from  $\Col_0(R)$ is equal to $\Col_0(R)$.
\vskip .1cm
Although the marked family of surfaces $S(R)$ (marked by its Fenchel-Nielsen coordinates defined above) tends to $\infty$ in the Teichm\"uller space of genus two surfaces, the unmarked family  $S(R)$ stays in some compact set in the moduli space of genus two surfaces. We prove this fact below.

\begin{lemma}\label{short} 
For $R$ large enough, the length of the shortest closed geodesic on the surface $\Su(R)$ is at least $e^{-5}$.
\end{lemma}

\begin{proof} Suppose that the length of the shortest closed geodesic on $\Su(R)$ is less than  $e^{-5}$ and let $\gamma$ be a lift of this geodesic to $\Ha$ (this geodesic is transverse to the lamination $\Col_0(R)$ because otherwise $\gamma \in \Col_0(R)$ which implies that the length of the shortest closed geodesic on $\Su(R)$ is equal to $R$). Then by Lemma \ref{lemma-est-1} every subsegment of $\gamma$ can intersect at most $R$ geodesics from $\Col_0(R)$, which means that $\gamma$ intersects at most $R$ geodesics from $\Col_0(R)$. This is impossible since $\gamma$ is a lift of a closed geodesic that is transverse to $\Col_0(R)$ so it has to intersect infinitely many geodesics from $\Col_0(R)$. This proves the lemma.

\end{proof}

The conclusion is that the family of (unmarked) Riemann surfaces $\Su(R)$ stays in some compact set in the moduli space of genus two surfaces. 
One can describe the accumulation set of the family $\Su(R)$  in the moduli space as follows. 
Let $P$ be a pair of pants that is decomposed into two ideal triangles so that all three shears between these two ideal triangles are equal to $1$. Then all three cuffs have the length equal to $2$. Let $\Su_t$, $t \in [0,1]$ be the genus two Riemann surface that is obtained by gluing one copy of $P$ onto another copy of $P$ (along the three cuffs) and twisting by $+2t$ along each cuff. The ``circle'' of surfaces  $\Su_t$ is the accumulation set of $\Su(R)$, when $R \to \infty$. Note that the edges of the ideal triangles that appear in the pants $P$ are the limits of the ($R$ long) cuffs from the pairs of pants $P(R)$. 
\vskip .1cm
Then we have the induced circle of orbifolds $\Orb_t$. Let $G_t$ be a circle of Fuchsian groups such that $\Orb_t=\Ha / G_t$. By $\Col_{0,t}$ we denote the lamination in $\Ha$ that is the lift of the corresponding ideal triangulation on $\Su_t$.  Then up to a conjugation by elements of $\PSLR$, the circle of groups $G_t$ is the accumulation set of the groups $G(R)$, when $R \to \infty$, and the circle of laminations $\Col_{0,t}$ is the accumulation set of the laminations $\Col_0(R)$. We observe that the group $G_t$ acts transitively on $\Col_{0,t}$.

\subsection{Quasisymmetric maps and hyperbolic geometry}

In this subsection we state and prove a few preparatory statements about quasisymmetric maps and the complex  distances between geodesics in $\Ho$, culminating in Theorem \ref{theorem-new-1}.

\begin{definition} We say that a geodesic lamination $\lambda$ on $\Ha$ is nonelementary if neither of the following holds:
\begin{enumerate}
\item There exists $z \in \partial{\Ha}$ that is an endpoint of every leaf of $\lambda$. 
\item There exists a geodesic $O \subset \Ha$ that is orthogonal to every leaf of $\lambda$.
\end{enumerate}
\end{definition}

Of course,  $\lambda$ has at least three elements if $\lambda$ is nonelementary. Moreover if $\lambda$ is nonelementary then there is a sublamination $\lambda' \subset \lambda$ such that $\lambda'$ contains exactly  three geodesics and that such that $\lambda'$ is nonelementary.
\vskip .1cm
Let $\lambda$ be a geodesic lamination, all of whose leaves have disjoint closures. By $\partial{\lambda}$ we denote the union of the endpoints of leaves from $\lambda$. We let $\iota_{\lambda}:\partial{\lambda} \to \partial{\lambda}$ be the involution such that $\iota_{\lambda}$ exchanges the two points of $\partial{\alpha}$, for every leaf $\alpha \in \lambda$. 
\vskip .1cm
We say that a quasisymmetric map $g:\partial{\Ha} \to \partial{\Ho}$ is $K$-quasisymmetric if for every 4 points on $\partial{\Ha}$ with cross ratio
equal to $-1$, the cross ratio of the image four points is within $\log K$ hyperbolic distance of $-1$ for the hyperbolic metric on 
$\C \setminus \{0,1, \infty\}$ (observe that a map is $K$-quasisymmetric if and only if it has a $K'$-quasiconformal extension to $\partial{\Ho}$ for some $K'>1$).

\vskip .1cm

If $\alpha$ and $\beta$ are oriented geodesics in $\Ho$ by $\dis(\alpha,\beta)$ we denote their unsigned complex distance.

\begin{lemma}\label{lemma-G} Suppose that $\lambda$ is nonelementary, and $f:\partial{\lambda} \to \partial{\Ho}$ is such that 
$$
\dis(f(\alpha),f(\beta))=\dis(\alpha,\beta),
$$
\noindent
for all $\alpha, \beta \in \lambda$. Then there is a unique M\"obius transformation $T$ such that either
\begin{enumerate}
\item $T=f$ on $\partial{\lambda}$, or
\item $T=f \circ \iota_{\lambda}$ on $\partial{\lambda}$.
\end{enumerate}

\end{lemma}
The second case can only occur when all the leaves of $\lambda$ have disjoint closures. We will prove two special cases of Lemma \ref{lemma-G} before we prove the lemma. 
\vskip .1cm
If the endpoints of $\alpha$ are $x$ and $y$ and $\alpha$ is oriented from $x$ to $y$ then we write $\partial{\alpha}=(x,y)$.
The following lemma is elementary.

\begin{lemma}\label{lemma-K} For every $d \in \C  /  2 \pi i \Z / \Z_2$, with $d \ne 0$, there exists a unique 
$s \in \C / 2\pi i \Z$ such that for two oriented geodesics $\alpha$ and $\beta$  we have $\dis(\alpha,\beta)=d$ if and only if $\partial{\beta}=(x,y)$ and $y=T_{s,\alpha}(x)$, where $T_{s,\alpha}$ is the translation by $s$ along $\alpha$. 
\end{lemma}

\begin{proposition}\label{proposition-A} Suppose that $\alpha_0,\alpha_1,\alpha_2$ are oriented geodesics in $\Ho$ for which $\dis(\alpha_i,\alpha_j)\ne 0$ for $i\ne j$, and $\alpha_0,\alpha_1,\alpha_2$ do not have a common orthogonal. 
Suppose $\alpha'_0,\alpha'_1,\alpha'_2$ are such that  $\dis(\alpha_i,\alpha_j)=\dis(\alpha'_i,\alpha'_j)$. Then we can find a unique 
$T \in \PSLC$  that satisfies one of the two conditions 
\begin{enumerate}
\item $T(\alpha_i)=\alpha'_i$, $i=0,1,2$,
\item $T(\alpha_i)=-\alpha'_i$, $i=0,1,2$, where $-\alpha'_i$ is $\alpha'_i$ with the orientation reversed,
\end{enumerate}

\end{proposition}

\begin{proof} 
Given $\alpha_i$ and $\alpha'_i$ satisfying the hypotheses of the proposition we can assume that $\alpha_i=\alpha'_i$ for $i=0,1$. Let $d_i=\dis(\alpha_i,\alpha_2)$, and let $T_i=T_{d_{i},\alpha_{i}}$ as in Lemma \ref{lemma-K}.
Then by Lemma \ref{lemma-K} for any $\beta$ for which $\dis(\alpha_i,\beta)=d_i$ we have $T_i(x)=y$ where $\partial{\beta}=(x,y)$.
Thus $(T_1^{-1} \circ T_{0})(x)=x$. Since $T_1 \ne T_0$ (because $\alpha_0 \ne \alpha_1$), 
we see that the equation $\dis(\alpha_i,\beta)=d_i$ (in $\beta$) has at most as many solutions as the equation $(T_1^{-1} \circ T_{0})(x)=x$, $x \in \partial{\Ha}$. Therefore $\dis(\alpha_i,\beta)=d_i$ has at most two solutions and it has at most one solution if $T_1^{-1} \circ T_{0}$ has a unique fixed point on $\partial{\Ha}$.
\vskip .1cm
On the other hand if we let $Q$ be the M\"obius transformation such that $Q(\alpha_i)=-\alpha_i$, for $i=0,1$ (such $Q$ exists since $\dis(\alpha_i,\alpha_j)\ne 0$ for $i\ne j$). Let $\wh{\alpha}_2=-Q(\alpha_2)$. Then $\dis(\alpha_i,\wh{\alpha}_2)=\dis(\alpha_i,\alpha_2)$ for $i=0,1$. Therefore $\wh{\alpha}_2 \ne \alpha_2$ since $\alpha_0,\alpha_1$ and $\alpha_2$ do not have a common orthogonal. We conclude that $\alpha'_2=\alpha_2$ or $\alpha'_2=\wh{\alpha}_2$.

\end{proof}

\begin{proposition}\label{proposition-P} Suppose that distinct geodesics $\alpha_0$ and $\alpha_1$ in $\Ha$ have a common endpoint $x \in \partial{\Ha}$, and let $\beta$ be another geodesic in $\Ha$ such that $x$ is not an endpoint of $\beta$. Set $E=\partial{\alpha_0} \cup \partial{\alpha_1} \cup \partial{\beta}$. Let $f:E \to \partial{\Ho}$ be such that $\dis(f(\alpha_i),f(\beta))=\dis(\alpha_i,\beta)$, $i=0,1$. Then there exists a unique M\"obius transformation $T$ such that $f=T$ on $E$.
\end{proposition}

\begin{proof} We can assume that the restriction of $f$ to $\partial{\alpha_0} \cup \partial{\alpha_1}$ is the identity. If $\partial{\beta} \subset 
\partial{\alpha_0} \cup \partial{\alpha_1}$, then $|E|=3$ and we are finished. If $\partial{\beta} \cap \partial{\alpha_0}=\{y\}$ for some $y$, then we can write $\partial{\beta}=(y,z)$ (or  $(z,y)$),  and then $\partial{f(\beta)}=(y,z')$ (or respectively $(z',y)$). But then $z=z'$  because  $\dis(f(\alpha_1),f(\beta))=\dis(\alpha_1,\beta)$
(here we use Lemma \ref{lemma-K}). Likewise if $\partial{\beta} \cap \partial{\alpha_1} \ne \emptyset$.
\vskip .1cm
If $\partial{\beta} \cap (\partial{\alpha_0} \cup \partial{\alpha_1}) =\emptyset$, then by Lemma \ref {lemma-K}
$\partial{\beta}=(y,z)$ and $\partial{f(\beta)}=(y',z')$ and
$z=T_0(y)=T_1(y)$, $z'=T_0(y')=T_1(y')$, where $T_i$ translates along $\alpha_i$, and then $T^{-1}_0 \circ T_1$ has $x$ as one of its fixed points, so the other must be $y$, so $y'=y$, so $f(\beta)=\beta$.
\end{proof}

Now we are ready to prove Lemma \ref{lemma-G}.

\begin{proof} First suppose that $\lambda$ has two distinct leaves $\alpha,\beta$ with a common endpoint $x$. Then there is a unique $T \in \PSLC$ for which $T=f$ on $\partial{\alpha}\cup \partial{\beta}$. By Proposition \ref{proposition-P} we have 
$T(\gamma)=f(\gamma)$, whenever $\gamma \in \lambda$ and $x$ does not belong to $\partial{\gamma}$. Because $\lambda$ is nonelementary we can find at least one such $\gamma$. 
\vskip .1cm
Now suppose $\delta \in \lambda$ and $x \in \partial{\delta}$. We want to show $T(\delta)=f(\delta)$. We can find $T' \in \PSLC$ such that $T'=f$ on $\partial{\alpha} \cup \partial{\delta}$. By Proposition \ref{proposition-P}, $T'(\gamma)=f(\gamma)$, so $T$ and $T'$ agree on $\partial{\alpha} \cup \partial{\delta}$, so $T=T'$, so $f(\delta)=T(\delta)$, and we are done.
\vskip .1cm
Now suppose that any two distinct leaves of $\lambda$ have disjoint closures. Then we can find three leaves $\alpha_i$, $i=0,1,2$, with no common orthogonal (because $\lambda$ is nonelementary). By Proposition \ref{proposition-A} we can find a unique $T \in \PSLC$ such that $T=f$ on $E=\bigcup_{i=0}^{2} \partial{\alpha_i}$, or $T=f \circ \iota_{\lambda}$ on $E$. In the latter case  we can replace $f$ with $f \circ \iota_{\lambda}$. In either case we can assume that $T$ is not the identity.
\vskip .1cm
Now given any $\beta \in \lambda$, we want to show that $f(\beta)=\beta$. For $i=1,2$ let $Q_i$ be the 180 degree rotation around $O_i$, the common orthogonal to $\alpha_0$ and $\alpha_i$. If $f(\beta)\ne \beta$, then $f(\beta)=-Q_i(\beta)$ for $i=1,2$, so $Q^{-1}_0 \circ Q_1$ fixes the endpoints of $\beta$. But $Q^{-1}_0 \circ Q_1$ fixes the endpoints of $\alpha_0$, and $\beta \ne \alpha_0$, so this is impossible. So $f(\beta)=\beta$ for every $\beta$, and we are finished.

\end{proof}

We observe that Lemma \ref{lemma-G} holds even if we do not require the lamination to be closed.

\begin{definition} Let $\lambda$ be a geodesic lamination on $\Ha$. An effective radius for $\lambda$ is a number $M>0$ 
such that every open hyperbolic disc of radius $M$ in $\Ha$ intersects $\lambda$ in a (not necessarily closed) nonelementary sublamination.
\end{definition}

We observe that the condition that the intersection of $\lambda$ and the open disc centred at $z$   of radius $M$ is nonelementary is open in both $z$ and  $\lambda$. The following proposition follows easily from this observation.

\begin{proposition}\label{prop-new-1} Let $\Lambda$ be a family of geodesic laminations on $\Ha$ such that 
\begin{enumerate}

\item \label{group invariant} if $\lambda \in \Lambda$, and $g \in \PSLR$, then $g(\lambda) \in \Lambda$,
\item \label{Hausdorff closed} $\Lambda$ is closed (and hence compact) in the Hausdorff topology on the space of geodesic  laminations modulo $\PSLR$,   
\item if $\lambda \in \Lambda$ then  $\lambda$ is nonelementary.
\end{enumerate}

Then we can find $M>0$ such that $M$ is an effective radius for every $\lambda \in \Lambda$.

\end{proposition}

We call such a family a closed invariant family of non-elementary laminations. 
For any $R_1>0$ we let $\Lambda(R_1)$ be the closure of $\bigcup_{R \ge R_1} \Col_0(R)$ under properties \ref{group invariant} and \ref{Hausdorff closed} in Proposition \ref{prop-new-1}. We observe that taking the Hausdorff closure just adds the translates of all the $\Col_{0,t}$ under $\PSLR$, where $\Col_{0,t}$ was defined in the previous subsection. Hence $\Lambda(R_1)$ is a closed invariant family of nonelementary laminations. 

We say that a lamination $\lambda$ is \emph{unflippable} if it has two distinct leaves with a common endpoint, or if the involution $\iota_\lambda$ is not continuous. The latter occurs if and only if there is a point of $\partial \lambda$ that is the limit of a sequence leaves of $\lambda$ whose diameter go to zero (or $\lambda$ has two distinct leaves with a common endpoint). This will always occur when $\lambda$ is invariant by a nonelementary Fuchsian group $G$, and $\lambda$ has a recurrent (or closed) leaf in $\Ha/G$. In particular, a nonempty lamination $\lambda$ that is invariant under a cocompact group is unflippable (and nonelementary). We conclude that all of the laminations in $\Lambda(R_1)$ are unflippable. 

%

We can now prove that a quasisymmetric map that locally preserves complex distances on an unflippable lamination is M\"obius.

\begin{proposition}\label{proposition-C} Suppose that $\lambda$ is an unflippable nonelementary lamination. Suppose that $M$ is an effective radius for $\lambda$, and $f:\partial{\Ha} \to \partial{\Ho}$ is a continuous embedding such that $\dis(f(\alpha),f(\beta)) =\dis(\alpha,\beta)$, for all $\alpha,\beta \in \lambda$ such that $d(\alpha,\beta) \le 3M$. Then $f$ is the restriction of a M\"obius transformation.
\end{proposition}

\begin{proof} For $z \in \Ha$ let $D_z$ be the open disc of radius $M$ centred at $z$, and $\lambda_z$ be the leaves of $\lambda$ that meet $D_z$.
Because $M$ is an effective radius, $\lambda_z$ is nonelementary. Therefore there is a unique $T_z \in \PSLC$ such that either $T_z=f$ on $\partial{\lambda_z}$ or $T_z=f \circ \iota_{\lambda}$ on  $\partial{\lambda_z}$. Now if $d(z,z') \le M$, then $\dis(f(\alpha),f(\beta)) =\dis(\alpha,\beta)$ for all $\alpha,\beta \in \lambda_z \cup \lambda_{z'}$, and $\lambda_z \cup \lambda_{z'}$ is nonelementary, so $T_z=T_{z'}$. We conclude that there is one $T \in \PSLC$ such that $T=f$ or $T=f \circ \iota_{\lambda}$ on all of $\partial{\lambda}$.
But in the latter case, $\iota_\lambda$ would be continuous, which is impossible since $\lambda$ is unflippable. 
\end{proof}

We now characterize the sequences of $K$-quasiconformal maps whose dilatations do not go to 1.
\begin{lemma}\label{lemma-new-1} Let $K_1>K>1$. Suppose that for $m \in \N$, $f_m:\partial{\Ha} \to \partial{\Ho}$ is $K_1$-quasisymmetric but not 
$K$-quasisymmetric. Then, after passing to a subsequence if necessary, we have that there exist  $h_m, q_m \in \PSLC$ such that 
$q_m \circ f_m \circ h_m \to f_{\infty} : \partial{\Ha} \to \partial{\Ho}$ is a $K_1$-quasisymmetric map and
$f_{\infty}$ is not a restriction of a M\"obius transformation on $\partial{\Ha}$.
\end{lemma}
 
\begin{proof} Fix $a,b,c,d \in \partial{\Ha}$ such that the cross ratio of these four points is equal to 1.  Since $f_m$ is not K-quasisymmetric 
there exist points $a_m,b_m,c_m,d_m \in \partial{\Ha}$  whose cross ratio is equal to one and such that the cross ratio of the points
$f_m(a_m),f_m(b_m),f_m(c_m),f_m(d_m) \in \partial{\Ho}$ stays outside some closed disc $U$ centred at the point $1 \in \C$ for every $m$. 
We let $h_m$ be the M\"obius transformation that maps $a,b,c,d$ to  $a_m,b_m,c_m,d_m$. We then choose $q_m \in \PSLC$ 
such that $q_m \circ f_m \circ h_m$ fixes the points $a,b,c$. Then for each $m$ the map $q_m \circ f_m \circ h_m$ is $K_1$-quasisymmetric and it 
fixes the points $a,b,c$. 
\vskip .1cm
The standard normal family argument states that given $L>1$, a sequence of $L$-quasisymmetric maps that all fix the same three distinct points, converges uniformly to a $L$-quasisymmetric map (after passing onto a subsequence if necessary). Therefore, we have $q_m \circ f_m \circ h_m \to f_{\infty}$. 
Moreover the cross ratio of the points $f_{\infty}(a),f_{\infty}(b), f_{\infty}(c),  f_{\infty}(d)$ lies outside the disc $U$ so we conclude that $f_{\infty}$ is not a M\"obius transformation on $\partial{\Ha}$.

\end{proof}

We can now conclude the constant of quasisymmetry for $f$ is  close to 1 when $f$ changes the complex distance of neighbouring 
geodesics  a sufficiently small amount.

\begin{theorem}\label{theorem-distance-qc}
Let $\Lambda$ be a closed invariant family of unflippable nonelementary laminations, and let $K_1 \ge K > 1$. Then there exists $\delta=\delta(K_1,K,\Lambda) > 0$ and $T=T(\Lambda)$  such that the following holds. If $\lambda \in \Lambda$ and $f:\partial{\Ha} \to \partial{\Ho}$ is a $K_1$-quasisymmetric map, and
$$
|\dis(f(\alpha),f(\beta))-\dis(\alpha,\beta)| \le \delta,
$$
\noindent
for all $\alpha,\beta \in \lambda$ such that $d(\alpha,\beta) \le T$. Then $f$ is $K$-quasisymmetric.
\end{theorem}

\begin{proof}
By Proposition \ref{prop-new-1}, we can find $M=M(\Lambda)>0$ such that $M$ is an effective radius for every $\lambda \in \Lambda$. We let $T = 3M$. 
Suppose that there is no good $\delta$. Then we can find $\lambda_m \in \Lambda$, $f_m$ (for $m \in \N$) such that 
$$
|\dis(f(\alpha),f(\beta))-\dis(\alpha,\beta)| \to 0, m \to \infty,
$$
\noindent
uniformly for all $\alpha,\beta \in \lambda_m$ for which $d(\alpha,\beta) \le T$, but for which $f$ is not $K$-quasisymmetric. Passing to a subsequence and applying Lemma \ref{lemma-new-1}, we obtain $\lambda_m \to \lambda_{\infty} \in \Lambda$, and 
$f_m \to f_{\infty}:\partial{\Ha} \to \partial{\Ho}$ such that $f_{\infty}$ is a $K_1$-quasisymmetric map that  is not a M\"obius transformation on $\partial{\Ha}$. Moreover $\dis(f_{\infty}(\alpha),f_{\infty}(\beta))=\dis(\alpha,\beta)$ for all 
$\alpha,\beta \in \lambda_{\infty}$ with $d(\alpha,\beta) \le T=3M$. Then by Proposition \ref{proposition-C}, $f_{\infty}$ is a M\"obius transformation, a contradiction.
\end{proof}

We can now derive a corollary, which is our object for this section:
\begin{theorem}\label{theorem-new-1} Let $K_1 \ge K>1$, and let $R_1 = 10$. There exists $\delta_1=\delta_1(K,K_1)>0$ and a universal constant $T_1$  such that the following holds. Suppose that $R \ge R_1$, and $f:\partial{\Ha} \to \partial{\Ho}$ is a $K_1$-quasisymmetric map, and
$$
|\dis(f(\alpha),f(\beta))-\dis(\alpha,\beta)| \le \delta_1,
$$
\noindent
for all $\alpha,\beta \in \Col_0(R)$ such that $d(\alpha,\beta) \le T_1$. Then $f$ is $K$-quasisymmetric.
\end{theorem}
This follows immediately from Theorem \ref{theorem-distance-qc}, because $\Lambda(R_1)$ is a closed invariant family of unflippable noninvariant laminations. Observe that $T_1=3M_1$ is a universal constant, where $M_1$ is the effective radius of every lamination in $\Lambda(R_1)$.

\subsection{Proof of Theorem \ref{geometry-1} }  

In this section we will verify Assumption \ref{ass} holds when the quasisymmetry constant for $f_{\tau}$ is close to 1. This will permit us, thanks to Lemma \ref{derivative-3}, to verify the hypotheses of Theorem \ref{theorem-new-1} and thereby improve the quasisymmetry constant for $f_{\tau}$.
We thus obtain an inductive argument for Theorem \ref{geometry-1}.
\vskip .1cm
This lemma is an abstraction of its corollary, Corollary \ref{cor-1} where $A,B,C$ will be $C_0(0),C_i(0),C_{n+1}(0)$.

\begin{lemma}\label{ll-1} For all $\delta_2,T_1>0$ we can find $K>1$ such that if

\begin{enumerate}
\item A,B,C are oriented geodesics in $\Ha$, $d(A,C)>0$, and $B$ separates $A$ and $C$, 
\item $d(A,C) \le T_1$,
\item O is the common orthogonal for $A$ and $C$,
\item $x=A \cap O$, $y=B \cap O$,
\item $f:\partial{\Ha} \to \partial{\Ho}$ is $K$-quasisymmetric,
\item $\partial{A'}=f(\partial{A})$,  $\partial{B'}=f(\partial{B})$, and  $\partial{C'}=f(\partial{C})$ 
(taking into account the order of the endpoints), 
\item $O'$ is the common orthogonal to $A'$ and $C'$, and $x'=A' \cap O'$,
\item $N$ is the common orthogonal to $O'$ and $B'$, and $y'=N \cap O'$,
\end{enumerate}
then $d(O',B') \le \delta_2$, and $|\dis_{O'}(x',y')-\dis(x,y)| \le \delta_2$.

\end{lemma}

\begin{proof} First suppose that $d(A,C)$ is small, say $d(A,C) \le T_2$ for some $T_2>0$, and $f$ is $2$-quasisymmetric. Then by applying a M\"obius transformation to the range and domain of $f$ we can assume that $\partial{A}=\partial{A'}=\{0,\infty\}$, and $1 \in \partial{O}$,  $1 \in \partial{O'}$ (and hence $\partial{O}=\partial{O'}=\{-1,1\}$). Note that while $f(0)=0$ and $f(\infty)=\infty$, $f(1)$ is not necessarily equal to 1. It follows that $\partial{C}=\{c,{{1}\over{c}} \}$, for $c$ real and small (we can assume $c>0$), and $\partial{C'}=\{c',{{1}\over{c'}} \}$ where $c'$ is small and 
$c'=f(c)$, ${{1}\over{c'}}=f( {{1}\over{c}})$.
\vskip .1cm
We let $\partial{B}=\{b_0,b_1\}$, where $b_0,{{1}\over{b_{1}}} \in (0,c)$. Then $|f(b_0)|<10|c'|$, and $|f(b_1)|>{{1}\over{10}}|{{1}\over{c'}}|$
because $f$ is $2$-quasisymmetric and $f$ fixes $0,\infty$. Therefore, by choosing $T_2$ to be  small enough we can arrange that $d(O',B')$, $d(x,y)$ and $d(x',y')$ are as small as we want, so 
we conclude that for every $\delta_2>0$ there exists $T_2>0$ such that if $d(A,C) \le T_2$, and $f$ is $2$-quasisymmetric then 
$$
d(O',B'), |d(x,y)-d(x',y')| <\delta_2.
$$
\noindent
So we need only show that for every $\delta_2$ and $T_1$ there exists $K>1$ such that if
$d(A,C) \in [T_2,T_1]$, where $T_2=T_2(\delta_2)$,  and all other hypotheses hold, then 
\begin{equation}\label{eka-1}
d(O',B'), |d(x,y)-d(x',y')| <\delta_2.
\end{equation}

\vskip .1cm
Suppose that this statement is false. Then we can find a sequence of $A_n,B_n,C_n$ and $f_n$ for which $f_n$ is $K_n$-quasisymmetric, $K_n \to 1$, but for which (\ref{eka-1}) does not hold. Then normalizing and passing to a subsequence we obtain $A,B,C$ in the limit, and $f_n \to \id$. So
$A'_n \to A'=A$, $B'_n \to B'=B$, and $C'_n \to C'=C$. Moreover, because the common orthogonal to two geodesics varies continuously when the complex distance is non-zero, $O_n \to O$ and $O'_n \to O'$, so $d(O',B'_n) \to 0$, and $\dis(B'_n,O) \ne 0$. Also 
$N'_n \to N$, and $(x_n,y_n,x'_n,y'_n) \to (x,y,x',y')$, and $x'=x$, $y'=y$ so
$$
|d(x'_n,y'_n)-d(x_n,y_n)| \to 0.
$$
\noindent
We conclude that (\ref{eka-1}) holds for large enough $n$, a contradiction.

\end{proof}

Assume that for some $\tau \in \D$ the representation $\rho_{\tau}:\pi_1(S^{0}) \to \PSLC$ is quasifuchsian and let  $f_{\tau}:\partial{\Ha} \to \partial{\Ho}$ be the normalised  equivariant quasisymmetric  map  (that conjugates $\rho_{0}(\pi_1(S^{0}))$ to  $\rho_{\tau}(\pi_1(S^{0}))$).
\vskip .1cm
Here we show that Assumption \ref{ass} holds if $f_{\tau}$ is sufficiently close to being conformal.
 
\begin{corollary}\label{cor-1}  Given $T_1$ we can find $K_1>1$ such that if $f_{\tau}$ is $K_1$-quasisymmetric then the following holds.
Let $C_0(0),C_{n+1}(0)$ be geodesics in $\Col_0(R)$ such that $d(C_0(0),C_{n+1}(0)) \le T_1$, and let $C_i(0) \in \Col_0(R)$, $i=1,...,n$, denote the intermediate geodesics. Also, $O(0), O(\tau)$, $z_i(0),z_i$ and $C_i(\tau)$ are defined as usual. Then 
$$
|d(z_i,z_{i+1})-d(z_i(0),z_{i+1}(0))| <{{e^{-5} }\over{4}},
$$
\noindent
and $d(O,C_i) \le  {{e^{-5} }\over{4}}$.
\end{corollary}

\begin{proof} We apply the previous lemma with $\delta_2={{ e^{-5} }\over{16}}$. Then  $d(O,C_i)<{{e^{-5}}\over{16}}$.
Furthermore
$$
|d(z'_0,z'_i)-d(z_0(0),z_i(0))| <{{e^{-5} }\over{16}},
$$
and
$$
|d(z'_0,z'_{i+1})-d(z_0(0),z_{i+1}(0))| <{{e^{-5} }\over{16}},
$$
so
$$
|d(z'_i,z'_{i+1})-d(z_i(0),z_{i+1}(0))| <{{e^{-5} }\over{8}}.
$$
Moreover 
$$
d(z_i, z_{i+1}) \le d(z'_i, z'_{i+1}) + d(O, C_i) + d(O, C_{i+1})
$$
and therefore
$$
|d(z_i,z_{i+1})-d(z_i(0),z_{i+1}(0))| <{{e^{-5} }\over{4}}.
$$

\end{proof}

We are now ready to complete the proof of Theorem \ref{geometry-1}. 
Let $R>R_1=10$. Since the space of quasifuchsian representations of the group $\pi_1(S^{0})$ is open (in the space of all representations),  
there exists $0<\epsilon_1<1$ so that the disc $\D(0,\epsilon_1)$ (of radius $\epsilon_1$ and centred at $0$) is the maximal disc such that $f_{\tau}$ is $K_1$-quasisymmetric on all of $\D(0,\epsilon_1)$,  where $K_1$ is the constant from Corollary \ref{cor-1}. We can choose such $\epsilon_1$ to be positive because the map $f_{0}$ is $1$-quasisymmetric and given any $K>1$ we can find an open neighbourhood of $0$ in the $\tau$ plane such that in that neighbourhood we have that every $f_{\tau}$ is $K$-quasisymmetric.

By that corollary, Assumption \ref{ass} holds for $f_{\tau}$, for all $\tau \in \D(0,\epsilon_1)$.  
Let $C_0(0), C_{n+1}(0) \in \Col_0(R)$ be such that  $d(C_0(0),C_0(n+1)) \le T_1$, where $T_1$ is the constant from Theorem \ref{theorem-new-1}. From Lemma \ref{derivative-3}, 
for $R$ large enough and for every $\tau \in \D(0,\epsilon_1)$, we have  
$$
|\dis'_{O} (C_0,C_{n+1})| \le  10^{9} T_1e^{T_{1}}.
$$
\noindent
This yields
\begin{equation}\label{ajde}
|\dis_{O} (C_0,C_{n+1})-\dis_{O(0)} (C_0(0),C_{n+1}(0))| \le 10^{9} \epsilon_1  T_1 e^{T_{1}},
\end{equation}
\noindent
for every $\tau \in \D(0,\epsilon_1)$. 
\vskip .1cm
Let  $0<\delta_1=\delta_1(\sqrt{K_1},K_1)$, 
be the corresponding constant from  Theorem \ref{theorem-new-1}.
We show 
$$
\epsilon_1 \ge {{\delta_1}\over{10^{9} T_1 e^{T_{1}}}}.
$$
\noindent
Assume that this is not the case. Then from (\ref{ajde})  we have that for every $\tau \in \D(0,\epsilon_1)$ the map $f_{\tau}$ is $\sqrt{K_1}$-quasisymmetric (and hence for   $\tau \in \overline{ \D(0,\epsilon_1)}$). This implies that $f_{\tau}$ is $K_1$-quasisymmetric for every $\tau \in \D(0,\epsilon)$ 
for some $\epsilon>\epsilon_1$. 
But this contradicts the assumption that $\D(0,\epsilon_1) $ is the maximal disc so that every $f_{\tau}$ is 
$K_1$-quasisymmetric. 
\vskip .1cm
Set
$$
\wh{\epsilon}= {{\delta_1}\over{10^{9} T_1 e^{T_{1}}}}.
$$
\noindent
Then for every  $\tau \in \D(0,\wh{\epsilon})$, and for $R$ large enough the map $f_{\tau}$ is $K_1$-quasisymmetric. 

We prove the other estimate in Theorem \ref{geometry-1} as follows. First of all, by the Slodkowski extension theorem (for the statement and proof of this theorem see  \cite{fletcher-markovic}), 
we can extend the maps $f_{\tau}$ to quasiconformal maps of the sphere $\partial{\Ho}$ such that the Beltrami dilatation
$$
\mu_{\tau}(z)=\frac{\overline{\partial} f_{\tau}}{\partial f_{\tau} }(z)
$$
varies holomorphically in $\tau$ for every fixed $z \in \partial{\Ho}$. Observe that $\mu_{0}(z)=0$, and  $|\mu_{\tau}(z)|<1$ for every $\tau$ and $z$ (recall that the absolute value of the Beltrami dilatation of any quasiconformal map is less than 1). For a fixed $z$ we then apply the Schwartz lemma to the function $\mu_{\tau}(z)$, and this yields the desired estimate from Theorem \ref{geometry-1}.

\section{Surface group representations in $\pi_1(\M)$}

\subsection{Labelled collection of oriented skew pants}  From now on $\M=\Ho / \KG$ is a  fixed closed hyperbolic three manifold and $\KG$ a suitable Kleinian group. By $\Gamma^{*}$  and $\Gamma$ we denote respectively the collection of oriented and unoriented closed geodesics in $\M$. By $-\gamma^{*}$ we denote the opposite orientation of an oriented geodesic $\gamma^{*} \in \Gamma^{*}$.
\vskip .1cm

Let $\Pi^{0}$ be a topological pair of pants.  Recall (from the beginning of Section 2) that a pair of pants in a closed hyperbolic 3-manifold $\M$  is an injective homomorphism 
$\rho:\pi_1(\Pi^{0}) \to \pi_1(\M)$, up to conjugacy. A pair of pants in $\M$ is determined by (and determines) a continuous map $f:\Pi^{0} \to \M$, up to homotopy. Moreover, the representation  $\rho$ induces a representation
$$
\rho:\pi_1(\Pi^{0}) \to \PSLC, 
$$
up to conjugacy.

Fix an orientation and a base point on $\Pi^{0}$.  We equip  $\Pi^0$ with an orientation preserving   homeomorphism $\omega:\Pi^{0} \to \Pi^{0}$, of order three that permutes the cuffs and let $\omega^{i}(C)$, $i=0,1,2$, denote the oriented cuffs of $\Pi^0$. We may assume that the base point of  $\Pi^{0}$ is fixed under $\omega$. 
By $\omega:\pi_1(\Pi^0) \to \pi_1(\Pi^0)$ we also denote the induced isomorphism of the fundamental group (observe that the homeomorphism $\omega:\Pi^{0} \to \Pi^{0}$ has a fixed point that is the base point for $\Pi^0$ so the isomorphism of the fundamental group is well defined). Choose  $c \in \pi_1(\Pi^0)$ to be an element in the conjugacy class that corresponds to the cuff $C$, such that $\omega^{-1}(c) c \omega(c)=\id$.

\begin{definition}\label{def-adm} Let  $\rho:\pi_1(\Pi^{0}) \to \PSLC$ be a faithful representation. We say that $\rho$ 
is an admissible representation if   $\rho(\omega^i(c))$ is a loxodromic M\"obius transformation, and
$$
\hl(\omega^{i}(C) )={{\len(\omega^{i}(C) )}\over{2}},
$$
\noindent
where $\len(\omega^{i}(C))$ is chosen so that $ -\pi <\IM(\len(\omega^{i}(C))) \le \pi$.

\end{definition}

\begin{definition}\label{def-skew} Let $\rho:\pi_1(\Pi^{0}) \to \KG$, be an admissible representation. A skew pants $\Pi$ is the conjugacy class   $\Pi=[\rho]$. The set of all skew pants is denoted by $\Pant$.
\end{definition}

For $\Pi \in \Pant$ we define $\refl(\Pi) \in \Pant$ as follows. Let $\rho:\pi_1(\Pi^{0}) \to \KG$
be a representation such that $[\rho]=\Pi$, and set $\rho(\omega^i(c))=A_i \in \KG$. Define the representation $\rho_1:\pi_1(\Pi^{0}) \to \KG$ by  
$\rho_1(\omega^{-i}(c) )=A^{-1}_i$. One verifies that $\rho_1$ is well defined and we let  $\refl(\Pi)=[\rho_1]$. The mapping 
$\refl:\Pant \to \Pant$ is a fixed point free involution.
\vskip .1cm
For $\Pi \in \Pant$ such that $\Pi=[\rho]$ we let $ \gamma^{*}(\Pi,\omega^{i}(c)) \in \Gamma^{*}$  denote the oriented geodesic that represents the conjugacy class of $\rho(\omega^i(c))$. Observe the identity $ \gamma^{*}(\refl(\Pi),\omega^{i}(c))= -\gamma^{*}(\Pi,\omega^{-i}(c))$.
The set of  pairs $(\Pi,\gamma^{*})$, where $\gamma^{*}=\gamma^{*}(\Pi,\omega^{i}(c))$, for some $i=0,1,2$, is called the set of marked skew pants and denoted by  $\Pant^{*}$.
\vskip .1cm
There is the induced  (fixed point free) involution  $\refl:\Pant^{*} \to \Pant^{*}$,   given by 
$\refl(\Pi,\gamma^{*}(\Pi,\omega^{i}(c)))=(\refl(\Pi),\gamma^{*}(\refl(\Pi),\omega^{-i}(c) ))$. 
Another obvious mapping  $\rot:\Pant^{*} \to \Pant^{*}$ is given  by 
$\rot(\Pi,\gamma^{*}(\Pi,\omega^{i}(c)))=(\Pi,\gamma^{*}(\Pi,\omega^{i+1}(c) ) )$.

\begin{definition} Let $\Lab$ be a finite set of labels. We say that a map  $\lab:\Lab \to \Pant^{*}$ is a legal labeling map if the following holds
\begin{enumerate}
\item  there exists an involution $\refl_{\Lab}:\Lab \to \Lab$,  such that $\refl(\lab(a))=\lab(\refl_{\Lab}(a))$,
\item there is a bijection $\rot_{\Lab}:\Lab \to \Lab$ such that $\rot(\lab(a))=\lab(\rot_{\Lab}(a))$.
\end{enumerate}
\end{definition}

\begin{example} Let $\N \Pant$ denote the collection of all formal sums of oriented skew pants from $\Pant$ over non-negative integers.
We say that $W \in \N \Pant $ is symmetric if $W=n_1(\Pi_1+\refl(\Pi_1))+n_2(\Pi_2+\refl(\Pi_2))+...+n_m(\Pi_m+\refl(\Pi_m))$, where 
$n_i$ are positive integers, and $\Pi_i \in \Pant$. Every symmetric $W$ induces a   canonical legal labeling defined as follows.
The corresponding set of labels is $\Lab=\{(j,k): j=1,2,...,2(n_1+n_2+...+n_m); k=0,1,2 \}$ (observe that the set $\Lab$ has $6(n_1+...+n_m)$ elements). 
Set  $\lab(j,k)=(\Pi_s, \gamma^{*}(\Pi_s, \omega^{k}(c) ) )$, if $j$ is odd and  if $2(n_1+\cdot \cdot \cdot +n_{s-1}) < j \le 2(n_1+ \cdot \cdot \cdot+n_{s})$. Set $\lab(j,k)=(\refl(\Pi_s), \gamma^{*}(\refl(\Pi_{s}),\omega^{-k}(c) ) )$, if $j$ is even and $2(n_1+ \cdot \cdot \cdot +n_{s-1}) < j \le 2(n_1 + \cdot \cdot \cdot +n_{s})$.
The bijection  $\refl_{\Lab}$ is given by  $\refl_{\Lab}(j,k)=(j+\delta(j),k)$, where $\delta(j)=+1$ if $j$ is even and  $\delta(j)=-1$ if $j$ is odd. 
The bijection $\rot_{\Lab}$ is defined accordingly.

\end{example}

\begin{definition} Let $\sigma:\Lab \to \Lab$ be an involution. We say that $\sigma$ is admissible with respect to a legal labeling $\lab$ if the following holds. Let $a \in \Lab$ and let $\lab(a)=(\Pi_1,\gamma^{*})$, for some $\Pi_1 \in \Pant$  where $\gamma^{*}=\gamma(\Pi_{1},\omega^{i}(c) )$, for some $i \in \{0,1,2\}$. Then $\lab(\sigma(a))=(\Pi_2,-\gamma^{*})$, where $\Pi_2 \in \Pant$ is some other skew pants. 
\end{definition}

Observe that every legal labeling has an admissible involution $\sigma:\Lab \to \Lab$, given by $\sigma(a)=\refl_{\Lab}(a)$. 

\vskip .1cm
Suppose that we are given a legal labeling $\lab:\Lab \to \Pant^{*}$ and an admissible involution $\sigma:\Lab \to \Lab$. We construct a closed topological surface $S^{0}$ (not necessarily connected) with a pants decomposition $\Col^{0}$, and a representation $\rho_{\lab,\sigma}:\pi(S^{0}) \to \KG$ as follows. Each element of $\Lab$ determines an oriented cuff in $\Col^{0}$. Each element in the orbit space $\Lab/ \rot_{\Lab}$ gives a copy of the oriented topological pair of pants $\Pi^{0}$. The pairs of pants are glued according to the instructions given by $\sigma$, and this defines the representation $\rho_{\lab,\sigma}$. One can check that after we glue the corresponding pairs of pants we construct a closed surface $S^{0}$. Moreover $S^{0}$ is connected if and only if the action of  the group of bijections $\left< \refl_{\Lab}, \rot_{\Lab},\sigma \right>$ is minimal on $\Lab$ (that is  $\Lab$ is the smallest invariant set under the action of this group). 

\vskip .1cm
Let $a \in \Lab$. Then   $(\Pi,\gamma^{*})=\lab(a)$ and $(\Pi_1,-\gamma^{*})=\lab(\sigma(a))$ for some skew pants $\Pi,\Pi_1 \in \Pant$. Also
$\gamma^{*}=\gamma^{*}(\Pi,\omega^i(c))$ and $-\gamma^{*}=\gamma^{*}(\Pi_1,\omega^j(c))$. 
Set  
$$
\hl(a)=\hl(\omega^i(C)), 
$$
\noindent
where the half length  $\hl(\omega^i(C))$ is computed for the representation that corresponds to the skew pants $\Pi$.

It follows from our definition of admissible representations that 
$\hl(a)=\hl(\sigma(a))$. Set $\len(a)= \len(\omega^i(C))$. Then $\len(a)=\len(\sigma(a))$ and
$$
\hl(a)={{\len(a)}\over{2}}.
$$

\subsection{The unit normal bundle of a closed geodesic} Next, we discuss in more details the structure of the unit normal bundle $\NB(\gamma)$ of a closed geodesic  $\gamma \subset \M$ (for the readers convenience we will repeat several definitions given at the beginning of Section 2).  The bundle  $\NB(\gamma)$ has an induced differentiable structure and it is diffeomorphic to a torus. Elements of $\NB(\gamma)$ are pairs $(p,v)$, where $p \in \gamma$ and $v$ is a unit  vector at $p$ that is orthogonal to $\gamma$. The disjoint union of all the bundles is denoted by $\NB(\Gamma)$.
\vskip .1cm
Fix an orientation $\gamma^{*}$ on $\gamma$. Consider $\C$ as an additive group and for $\zeta \in \C$, let $\A_{\zeta}:\NB(\gamma) \to \NB(\gamma)$ be the mapping given by $\A_{\zeta}(p,v)=(p_1,v_1)$ where $p_1$ and $v_1$ are defined as follows. Let $\wt{\gamma^{*}}$ be a lift of $\gamma^{*}$ to $\Ho$ and let $(\wt{p},\wt{v}) \in  \NB(\wt{\gamma})$ be a lift of $(p,v)$. 
We may assume that $\wt{\gamma^{*}}$ is the geodesic between $0,\infty \in \partial{\Ho}$. Let $A_{\zeta} \in \PSLC$ be given by $A_{\zeta}(w)=e^{\zeta} w$, for $w \in \partial{\Ho}$. Set $(\wt{p}_1,\wt{v}_1)=A_{\zeta}(\wt{p},\wt{v})$. Then $(\wt{p}_1,\wt{v}_1)$ is a lift of $(p_1,v_1)$.
\vskip .1cm
If $\A^{1}_{\zeta}:\NB(\gamma)\to \NB(\gamma) $ and $\A^{2}_{\zeta}:\NB(\gamma)\to \NB(\gamma)$ are the actions that correspond to different orientations on $\gamma$  then on $\NB(\gamma)$ we have $\A^{1}_{\zeta}=\A^{2}_{-\zeta }=(\A^{2}_{\zeta})^{-1}$, $\zeta \in \C$. Unless we specify otherwise, by $\A_{\zeta}$ we denote either of the two actions.
\vskip .1cm
The group $\C$ acts transitively on $\NB(\gamma)$. Let $\len(\gamma)$ be the complex translation length of $\gamma$, such that $-\pi<\IM(\len(\gamma)) \le \pi$ (by definition  $\RE(\len(\gamma)) >0$). 
Then $A_{\len(\gamma)}=\id$ on $\NB(\gamma)$. This implies that the map $A_{ {{\len(\gamma)}\over{ 2}} }$ is an involution which enables us to define the bundle $\NB(\sqrt{\gamma})= \NB(\gamma)/ A_{ {{\len(\gamma)}\over{2}} }$. The disjoint union of all the bundles is denoted by $\NB(\sqrt{\Gamma})$.   
\vskip .1cm
The additive group $\C$ acts on $\NB(\sqrt{\gamma})$ as well. There is a unique complex structure on  $\NB(\sqrt{\gamma})$ so that the action 
$\A_{\zeta}$ is by biholomorphic maps. With this complex structure we have
$$
\NB(\sqrt{\gamma}) \equiv \C/ \big( {{\len(\gamma)}\over{2}} \Z+ 2 \pi i \Z \big ).
$$
\noindent 
The corresponding  Euclidean distance on  $\NB(\sqrt{\gamma})$ is denoted by $\dist$. Then for $|\zeta|$ small we have    $\dist((p,v),(\A_{\zeta}(p,v)))=|\zeta|$.  There is also the induced map $\A_{\zeta}: \NB(\sqrt{\Gamma}) \to \NB(\sqrt{\Gamma})$, $\zeta \in \C$,  where 
the restriction of $\A_{\zeta}$ on each torus $\NB(\sqrt{\gamma})$ is defined above. 
\vskip .1cm
Let $(\Pi,\gamma^{*}) \in \Pant^{*}$  and let $\gamma^{*}_k$ be such that  $(\Pi,\gamma^{*}_k)=\rot^{k}(\Pi,\gamma^{*} )$, $k=1,2$.
Let $\delta^{*}_k$ be an oriented  geodesic  (not necessarily closed) in $\M$ such that $\delta^{*}_k$ is the common orthogonal of   $\gamma^{*}$ and $\gamma^{*}_k$, 
and so that a lift of $\delta^{*}_k$ is a side in the corresponding  skew right-angled hexagon that determines $\Pi$ (see Section 2). The orientation on $\delta^{*}_k$ is determined so that the point $\delta^{*}_k \cap \gamma^{*}_k$ comes after the point $\delta^{*}_k \cap \gamma^{*}$.  
Let $p_k=\delta^{*}_k \cap \gamma^{*}$, and let $v_k$ be the unit  vector $v_k$ at $p_k$ that has the same direction as $\delta^{*}_k$. 
Since the pants $\Pi$ is the conjugacy class of an admissible representation in sense of Definition \ref{def-adm}, we observe that 
$\A_{ {{\len(\gamma)}\over{2}} }$ exchanges $(p_1,v_1)$ and $(p_2,v_2)$, so the class $[(p_k,v_k)] \in \NB(\sqrt{\gamma})$ does not depend on $k \in \{1,2\}$.  Define the map
$$
\foot:\Pant^{*} \to \NB(\sqrt{\Gamma}),
$$
\noindent
by
$$
\foot_{\gamma}(\Pi)=\foot(\Pi,\gamma^{*})=[(p_k,v_k)] \in \NB(\gamma), 
$$
Observe that $\foot(\Pi,\gamma^{*})=\foot(\refl(\Pi,\gamma^{*}))$.
\vskip .1cm
Let   $S^{0}$ be a topological surface with a pants decomposition $\Col^{0}$, and let  $\rho:\pi_1(S^{0}) \to \KG$ be a representation, such that the restriction of $\rho$ on the fundamental group of each pair of pants satisfies the assumptions of Definition \ref{def-adm} (recall that $\KG$ is the Kleinian group such that $\M \equiv \Ho / \KG$). Let  $\Pi^{0}_i$, $i=1,2$ be two pairs of pants from the pants decomposition of  $S^{0}$ that both have a given cuff $C \in \Col^{0}$ in its boundary.  By $(\Pi_1,\gamma^{*})$ and $(\Pi_2,-\gamma^{*})$ we denote the corresponding marked pants in $\Pant^{*}$. 
Let $s(C) $ denote the corresponding reduced  complex Fenchel-Nielsen coordinate for $\rho$. 
Let $\A^{1}_{\zeta}$ be the action on $\NB(\sqrt{\gamma} )$ that corresponds to the orientation $\gamma^{*}$. Fix $\zeta_0 \in \C$ to be such that 
$$
\A^{1}_{\zeta_{0}}( \foot(\Pi_1,\gamma^{*} ))= \foot(\Pi_2,-\gamma^{*}).
$$
\noindent
Such $\zeta_0$ is uniquely determined up to a translation from the lattice  ${{\len(\gamma)}\over{2}} \Z+ 2i\pi \Z$. If $\A^{2}_{\zeta}$ is the other action then we have
$$
\A^{2}_{\zeta_{0}}( \foot(\Pi_2,-\gamma^{*} ))= (\Pi_1,\gamma^{*}),
$$
\noindent
since $\A^{1}_{\zeta} \circ \A^{2}_{\zeta}=\id$. That is, the choice of $\zeta_0$ does not depend on the choice of the action $\A_{\zeta}$.
Then $s(C) \in  \C / ({{\len(\gamma)}\over{2}} \Z  + 2\pi i \Z)$ and 
\begin{equation}\label{FN-0}
s(C)=(\zeta_0-i\pi), \,\, \pmod {  {{\len(\gamma)}\over{2}} \Z + 2\pi i \Z }.
\end{equation}
\vskip .1cm
The rest of the paper is devoted to proving the following theorem.

\begin{theorem}\label{thm-lab}  There exist  constants ${\bf q} >0$ and $K>0$ such that for every $\epsilon>0$ and for every $R>0$ large enough  the following holds. There exist a finite set of labels $\Lab$, a legal labeling  $\lab:\Lab \to \Pant$ and an admissible involution $\sigma:\Lab \to \Lab$, such that for every $a \in \Lab$ we have
$$
|\hl(a)-{{R}\over{2}}| < \epsilon,
$$
\noindent
and
$$
\dist( \A_{1+i\pi}( \foot(\lab(a))),\foot(\lab(\sigma(a)))) \le K R e^{- {\bf q}  R}, 
$$
\noindent
where $\dist$ is the Euclidean distance on  $\NB(\sqrt{\gamma})$. 
\end{theorem}

\begin{remark} The constant $\bf {q}$ depends on the manifold $\M$. In fact it only depends on the first eigenvalue for the Laplacian on $\M$.
\end{remark}

Given this theorem we can prove Theorem \ref{main} as follows. We saw that every legal labeling together with an admissible involution yields a representation $\rho(\lab,\sigma):\pi_1(S^{0}) \to \KG$, where $\KG$ is  the corresponding Kleinian group and $S^{0}$ is a closed topological surface (if $S^{0}$ is not connected we pass onto a connected component). By the above discussion the reduced complex Fenchel-Nielsen coordinates $(\hl(C),s(C))$ satisfy the assumptions of Theorem \ref{geometry} (observe that  $K R e^{-{\bf q} R}=o({{1}\over{R}})$, when $R \to \infty$). Then   Theorem \ref{main} follows from Theorem \ref{geometry} .

\subsection{Transport of measure} Let $(X,d)$ be a metric space. By $\Mes(X)$ we denote  the space of  positive, finite Borel measures on $X$ with compact support. For $A \subset X$ and $\delta>0$ let 
$$
\Ne_{\delta}(A)=\{x \in X: \quad \text{there exists} \quad  a \in A \quad \text{such that} \quad  d(x,a) \le \delta \},
$$
\noindent
be the $\delta$-neighbourhood of $A$.
 
\begin{definition} Let $\mu, \nu \in \Mes(X)$ be two measures such that $\mu(X)=\nu(X)$, and let $\delta>0$. 
Suppose that for every Borel set $A \subset X$ we have  $\mu(A) \le \nu(\Ne_{\delta}(A))$. Then we say that $\mu$ and $\nu$ are $\delta$-equivalent measures.
\end{definition}

It appears that the definition is asymmetric in $\mu$ and $\nu$. But this is not the case. For any Borel set $A \subset X$ the above definition yields
$\nu(A) \le \nu(X)-\nu \big(\Ne_{\delta}(X \setminus \Ne_{\delta}(A))\big) \le \mu(X)-\mu(X \setminus \Ne_{\delta}(A))=\mu(\Ne_{\delta}(A)$. This shows that the definition is in fact symmetric in $\mu$ and $\nu$.

\vskip .1cm
The following propositions follow from the definition of equivalent measures.

\begin{proposition}\label{prop-basic-1} Suppose that  $\mu$ and $\nu$ are $\delta$-equivalent. Then for any $K>0$ the measures  $K\mu$ and $K\nu$ are $\delta$-equivalent. 
If in addition we assume that measures $\nu$ and $\eta$ are  $\delta_1$-equivalent, then $\mu$ and $\eta$ are $(\delta+\delta_1)$-equivalent. 
\end{proposition}

\begin{proposition}\label{prop-basic-2}  Let $(T,\Lambda)$ be a measure space and let $f_i:T \to X$, $i=1,2$, be two maps such that $d(f_1(t),f_2(t)) \le \delta$, for almost every $t \in T$. Then the measures $(f_1)_{*}\Lambda$ and $(f_2)_{*} \Lambda$ are $\delta$-equivalent.
\end{proposition}

In the remainder of this subsection  we prove two theorems, each representing  a converse of the previous proposition in a special case. The following theorem is a converse of Proposition \ref{prop-basic-2} in the special case of discrete measures.

\begin{theorem}\label{thm-Hall}  Suppose that $A$ and $B$ are finite sets with the same number of elements, and equipped with the standard counting measures $\Lambda_A$ and $\Lambda_B$  respectively. Suppose that there are maps $f:A \to X$ and $g:B \to X$ such that the measures $f_{*} \Lambda_A$ and $g_{*} \Lambda_B$ are $\delta$-equivalent for some $\delta>0$. Then one can find a bijection $h:A \to B$ such that $d(g(h(a)),f(a)) \le \delta$, for every $a \in A$.
\end{theorem}

\begin{proof} We use the Hall's marriage theorem which states the following. Suppose that $\Rel \subset A \times B$ is a relation. 
For every $Q \subset A$ we let
$$
\Rel(Q)=\{b \in B: \, \text{there exists} \, a \in Q, \, \text{such that} \, (a,b) \in \Rel \}.
$$
\noindent
If $|\Rel(Q)| \ge |Q|$ for every $Q \subset A$ then there exists an injection $h:A \to B$ such that $(a,h(a)) \in \Rel$ for every $a \in A$. This is Hall's marriage theorem. In the general case of this theorem the sets $A$ and $B$ need not have the same number of elements. However, in our case they do so the map $h$ is a bijection.
\vskip .1cm
Define $\Rel \subset A \times B$ by saying that $(a,b) \in \Rel$ if $d(f(a),g(b)) \le \delta$. Then 
$$
\Rel(Q)=\{b \in B: \, \text{there exists} \, a \in Q, \, \text{such that} \, d(f(a),g(b))\le \delta \},
$$
\noindent
for every $Q \subset A$. Therefore $|\Rel(Q)|=(g_{*}\Lambda_B)(\Ne_{\delta}(f(Q))) \ge  (f_{*}\Lambda_A)(f(Q)) = |Q|$, since  $f_{*} \Lambda_A$ and $g_{*} \Lambda_B$ are $\delta$-equivalent. This means that the hypothesis of the Hall's marriage theorem is satisfied, and one can find the bijection $h:A \to B$ such that $d(g(h(a)),f(a)) \le \delta$.

\end{proof}

Let $a,b \in \C$ be two complex numbers such that $T(a,b)=\C/(a\Z + ib\Z)$ is a  torus. We let $z=x+iy$ denote a point in $\C$ (sometimes we use $(x,y)$ to denote a point in $\R^2 \equiv \C$).

Let $\phi$ be a positive $C^0$ function on $T(a,b)$. As usual, by $\phi(x,y) \, dx \, dy $ we denote the corresponding two form on the torus $T(a,b)$. By $\lambda_{\phi}$ we denote the measure on $T(a,b)$ given by
$$
\lambda_{\phi}(A)=\int_{A} \phi(x,y) \, dx \, dy,
$$
for any measurable set $A$. We abbreviate this equation $d\lambda_{\phi}=\phi\, dx \, dy$. By $\lambda$ we denote the standard Lebesgue measure on $T(a,b)$,  that is $\lambda=\lambda_{\phi}$ for $\phi \equiv 1$. 

In  the following  lemma we show that any $C^0$ measure that is close to the Lebesgue measure is obtained by transporting the Lebesgue measure by a diffeomorphism that is $C^0$ close to the identity.

\begin{lemma}\label{lemma-Radon} Let $g:\R^2 \to \R$ be a $C^{0}$ function on $\C$ that is well defined on the quotient $T(1,1)=\C/(\Z + i\Z)$, and such that 
\begin{enumerate} 
\item For some $0< \delta<\frac{1}{3}$ we have 
$$
1-\delta \le g(x,y) \le 1+\delta
$$
for all $(x,y) \in \R^2$,
\item The following equality holds
$$
\int_{0}^{1} \int_{0}^{1} g(x,y) \,dx \, dy=1.
$$
\end{enumerate}
Then we can find a $C^1$ diffeomorphism $h:T(1,1) \to T(1,1)$ such that 
\begin{enumerate} 
\item $g(x,y)=\Jac(h)(x,y)$, that is $g(x,y) \, dx \, dy= h^{*}(dx \, dy)$, where  $h^{*}(dx \, dy)$ is the pull-back of the two form $dx \, dy$ by the diffeomorphism $h$ and  $\Jac(h)$ is the Jacobian of $h$,
\item The inequality 
$$
|h(z)-z| \le 4 \delta,
$$
holds for every $z=x+iy \in \C$.
\end{enumerate}

\end{lemma}

\begin{proof} We define the map $h:\R^2 \to \R^2$ by  $h(x,y)=(h_1(x,y),h_2(x,y))$, where
$$
h_1(x,y)=\int_{0}^{x} \left( \int_{0}^{1}g(s,t)\, dt \right)\, ds,
$$
and 
$$
h_2(x,y)=\frac{ \int_{0}^{y} g(x,t)\, dt } { \int_{0}^{1} g(x,t) \, dt }.
$$
Since $g(x+1,y)=g(x,y+1)=g(x,y)$, we find that  $h(x+1,y)-h(x,y)=(1,0)$ and $h(x,y+1)-h(x,y)=(0,1)$, so $h$ descends to a map from $T(1,1)$ to itself.

Furthermore, we find that
$$
\frac{\partial{h_1} } {\partial{x} }=\int_{0}^{1}g(x,t) \, dt ; \,\, \, \frac{\partial{h_1} } {\partial{y} }=0,
$$
and
$$
\frac{\partial{ h_2} }{\partial{y}}=\frac{g(x,y)}{\int_{0}^{1}g(x,t) \, dt },
$$
which is sufficient to conclude that 
$$
\Jac(h)(x,y)=g(x,y). 
$$
Therefore, the map $h:T(1,1) \to T(1,1)$ is a local diffeomorphism, and thus a covering map of degree $n$ where 
$$
n=\int_{T(1,1)} \Jac(x,y) \, dx \, dy.
$$
Since $\Jac(h)(x,y)=g(x,y)$, and
$$
\int_{T(1,1)} g(x,y) \, dx \, dy =1,
$$
it follows that $n=1$, that is $h$ is a diffeomorphism.

On the other hand, for $x,y \in [0,1]$,
$$
|h_1(x,y)-x| \le \delta x \le \delta,
$$
and 
$$
h_2(x,y)-y \le \frac{y(1+\delta)}{1-\delta} -y \le 3\delta y \le 3 \delta,
$$
since $\delta<\frac{1}{3}$, and 
$$
y-h_2(x,y) \le y-\frac{y(1-\delta)}{1+\delta} \le 2\delta y \le 2 \delta.
$$
Therefore, $|h_2(x,y)-y| \le 3 \delta$. Combining the estimates for $|h_1(x,y)-x|$ and $|h_2(x,y)-y|$ we find that
$$
|h(z)-z| \le |h_1(x,y)-x|+|h_2(x,y)-y| \le 4\delta.
$$
The completes the proof.

\end{proof}

The following theorem is a corollary of the of the previous lemma.

\begin{theorem}\label{thm-Radon} Let $\mu \in \Mes(T(a,b))$ be a measure whose Radon-Nikodym derivative  ${{d\mu}\over{d \lambda}}(z)$ is a $C^{0}$ function on the torus $T(a,b)$,  such that 
for some $K>0$ and $\frac{1}{3}>\delta>0$ we have  $\mu(T(a,b))=K\lambda(T(a,b))$ and
$$
K \le \left| {{d\mu}\over{d \lambda}} \right| \le K(1+\delta), \, \, \, \text{everywhere on} \,\,  T(a,b),
$$
Then $\mu$ is $4\delta (|a|+|b|)$-equivalent to the measure $K\lambda$.
\end{theorem}

\begin{proof} By $\mu$  we also denote the lift of the corresponding measure to the universal cover $\C$. Then $d\mu=g_1(x,y)dx \, dy$, where $g_1(x,y)={{d\mu}\over{d \lambda}}(x,y)$ is the Radon-Nikodym derivative. The function $g_1$ is  $C^0$ on $\C$, and $g_1$ is well defined on the quotient  $\C/(a\Z +b i\Z)=T(a,b)$.

Let $L:T(1,1) \to T(a,b)$ be the standard affine map. Let 
$$
g(x,y)=\frac{1}{K}  (g_1 \circ L)(x,y).
$$
Then $g(x,y)$ satisfies the assumptions of the previous lemma. Let $h$ be the corresponding diffeomorphism from Lemma \ref{lemma-Radon}, and let $h_1=L \circ h \circ L^{-1}$. 
Then $(h_1)^{*} \mu=K\lambda$ on $T(a,b)$. Since the affine map $L$ is $(|a|+|b|)$ bi-Lipschitz we conclude that 
$$
|h_1(z)-z| \le 4\delta(|a|+|b|)
$$
for every $z \in \C$, so $\mu$ is $4\delta (|a|+|b|)$-equivalent to $K\lambda$.

\end{proof}

\subsection{Measures on skew pants and the $\ph$ operator} We have

\begin{definition} By $\Mes^{\refl}_0(\Pant)$ we denote the space of positive  Borel measures with finite support  on the set of oriented skew pants $\Pant$ such that  the involution $\refl:\Pant \to \Pant$ preserves each measure in $\Mes^{\refl}_0(\Pant)$.  By $\Mes_0( \NB(\sqrt{\Gamma}) ) $ we denote the space of positive  Borel measures with compact support  on  the manifold $\NB(\sqrt{\Gamma})$ (a measure from  $\Mes_0(\NB(\sqrt{\Gamma}))$ has a compact support if and only if its support is contained in at most finitely many tori  $\NB(\sqrt{\gamma}) \subset \NB(\sqrt{\Gamma})$). 
\end{definition}

We define  the  operator 
$$
\ph:\Mes^{\refl}_0(\Pant) \to \Mes_0( \NB(\sqrt{\Gamma}) ),
$$
\noindent
as follows. The set $\Pant$ is a countable set, so every measure from $\mu \in \Mes^{\refl}_0(\Pant)$ is determined by its value $\mu(\Pi)$ 
on every $ \Pi \in \Pant$. Let $ \Pi \in \Pant$ and let $\gamma^{*}_i \in \Gamma^{*}$, $i=0,1,2$, denote the corresponding oriented geodesics so that $(\Pi,\gamma^{*}_i) \in \Pant^{*}$. Let $\alpha^{\Pi}_i \in \Mes_0( \NB(\sqrt{\Gamma}) ) $ be the atomic measure supported at the point $\foot(\Pi,\gamma^{*}_i) \in  \NB(\sqrt{\gamma_i})$, where the mass of the atom is equal to $1$. Let 
$$
\alpha^{\Pi}=\sum_{i=0}^{2} \alpha^{\Pi}_i,
$$
\noindent
and define
$$
\ph \mu= \sum_{\Pi \in \Pant} \mu(\Pi) \alpha^{\Pi}.
$$
\noindent
We call this the $\ph$ operator on measures. The total measure of $\ph \mu$ is three times the total measure of $\mu$.
\vskip .1cm
Let $\alpha \in \Mes_0(\NB(\sqrt{\Gamma})) $. Choose $\gamma^{*} \in \Gamma^{*}$, and recall the action $\A_{\zeta}: \NB(\sqrt{\Gamma}) \to \NB(\sqrt{\Gamma})$, $\zeta \in \C$. Let $(\A_{\zeta})_{*} \alpha$ denote the push-forward of the measure $\alpha$.  We say that $\alpha$ is $\delta$-symmetric if the measures  $\alpha$ and $(\A_{\zeta})_{*} \alpha$ are $\delta$-equivalent for every $\zeta \in \C$.

\begin{theorem}\label{theorem-mes} 
There exists ${\bf q} >0$ and  $D_1,D_2>0$, so that for every $1 \ge \epsilon>0$ and every $R>0$ large enough, there exists a measure  
$\mu \in \Mes^{\refl}_0(\Pant)$ with the following properties. If $\mu(\Pi)>0$ for some $\Pi \in \Pant$, then the half lengths $\hl(\omega^{i}(C))$
that correspond to the skew pants $\Pi$ satisfy the inequality 
$$
\left| \hl(\omega^i(C))-{{R}\over{2}} \right|  \le \epsilon .
$$
\noindent 
There exists a measure $\beta \in \Mes_0(\NB(\sqrt{\Gamma}))$, such that  the measure $\ph \mu$ and $\beta$ are  
$D_1e^{-{{R}\over{8}} }$-equivalent, and such that for each torus $\NB(\sqrt{\gamma})$, there exists a constant $K_{\gamma} \ge 0$ so that
$$
K_{\gamma} \le \left| {{d\beta}\over{d \lambda}} \right| \le K_{\gamma}(1+D_2e^{-  {\bf q}  R} ), \text{almost everywhere on}\, \NB(\sqrt{\gamma}) ,
$$
\noindent
where $\lambda$ is the standard Lebesgue measure on the torus $\NB(\sqrt{\gamma})=\C/ \big( {{\len(\gamma)}\over{2}} \Z+ 2i\pi \Z \big)$.
\end{theorem}

\begin{remark} This theorem holds in two dimensions as well.  That is, in the statement of the above theorem  we can replace a closed hyperbolic three manifold $\M$ with any hyperbolic closed surface. 
\end{remark}

We prove this theorem in the next section.  But first we prove Theorem \ref{thm-lab} assuming Theorem \ref{theorem-mes}. We have

\begin{proposition}\label{proposition-mes} 
There exist ${\bf q} >0$, $D>0$ so that for every $1 \ge \epsilon>0$ and every $R>0$ large enough, there exists a measure  
$\mu \in \Mes^{\refl}_0(\Pant)$ with the following properties
\begin{enumerate}
\item $\mu(\Pi)$ is a rational number for  every $\Pi \in \Pant$,
\item  if $\mu(\Pi)>0$ for some $\Pi \in \Pant$, then the half lengths $\hl(\omega^{i}(C))$
that correspond to the skew pants $\Pi$ satisfy the inequality 
$ |\hl(\omega^i(C))-{{R}\over{2}}| \le \epsilon $,

\item the measures $\ph \mu$  and $(\A_{1+i\pi})_{*} \ph \mu$   are $D R e^{ - {\bf q}  R}$-equivalent.

\end{enumerate}
\end{proposition}

\begin{proof} Assume the notation and the conclusions of Theorem \ref{theorem-mes}. First we show that the measures $\ph \mu$  and 
$(\A_{1+i\pi})_{*} \ph \mu$   are $D R e^{-{\bf q} R}$-equivalent. 
Let $\gamma \in \Gamma$ be a closed geodesic such that  $\beta( \NB(\sqrt{\gamma})) >0$, that is the support of $\beta$ has a non-empty intersection with the torus 
$$
\NB(\sqrt{\gamma}) \equiv \C / \big( {{\len(\gamma)}\over{2}} \Z+2\pi i \Z \big).
$$
\noindent
The Lebesgue measure $\lambda$ on $\NB(\sqrt{\gamma})$  is invariant under the action $\A_{\zeta}$. This, together with Theorem \ref{thm-Radon}, 
implies that for any $\zeta \in \C$ the measure $(\A_{\zeta})_{*} \beta$ is $(2\pi+ \big| {{\len(\gamma^{*})}\over{2}} \big| ) D_2 e^{-{\bf q} R}$-equivalent with the measure $K'\lambda$, for some $K'>0$, where $D_2$ is from the previous theorem. Since $\big| {{\len(\gamma^{*})}\over{2}} \big| 
\le {{R}\over{2}}+1$,  we have that the measures
$(\A_{\zeta})_{*} \beta$ and  $K'\lambda$ are $C_1 R e^{- {\bf q} R}$-equivalent, for some $C_1>0$.
\vskip .1cm
On the other hand, the measures $(\A_{\zeta})_{*} \beta$ and $(\A_{\zeta})_{*} \ph \mu$ are $D_1e^{-{{R}\over{8}} }$-equivalent.  From Proposition \ref{prop-basic-1} we conclude that the measures  $(\A_{\zeta})_{*} \ph \mu$ and $K'\lambda$ are 
$D_2  (R e^{- {\bf q} R}+e^{ -{{R}\over{8}} })$-equivalent, for every $\zeta \in \C$, and for some constant $D_2>0$. 
Again, since $\lambda$ is  invariant under $\A_{\zeta}$ and from Proposition \ref{prop-basic-1} we conclude that $\ph \mu$ is $D R e^{-{\bf q}  R}$-symmetric, for some constant $D>0$ and assuming that ${\bf q} \le \frac{1}{8}$ (this assumption can be made without loss of generality). 
In particular, we have that the measures $\ph \mu$  and  $(\A_{1+i\pi})_{*} \ph \mu$   are $D R e^{- {\bf q} R}$-equivalent. 
\vskip .1cm
Both measures $\ph \mu$  and $ (\A_{1+i\pi})_{*} \ph \mu$ are atomic (with finitely many atoms), so it follows from the definition that  the measures $\ph \mu$  and $(\A_{1+i\pi})_{*} \ph \mu$   are $D R e^{- {\bf q} R}$-equivalent if and only if  a finite system of linear inequalities  with integer coefficients has a real valued solution. Then the standard rationalisation procedure (see Proposition 4.2 in \cite{kahn-markovic} and \cite{calegari}) implies that this system of equations has a rational solution, so we may assume that that the measure $\mu$ from Theorem \ref{theorem-mes} has rational weights. This proves the proposition.
\end{proof}

\subsection{Proof of Theorem \ref{thm-lab}} First we make several observations about an arbitrary measure  $\nu \in \Mes^{\refl}_0(\Pant)$. The measure 
$\nu$ is supported on finitely  many skew pants $\Pi \in \Pant$. Moreover, $\nu(\Pi)=\nu(\refl(\Pi))$, for every $\Pi \in \Pant$. Let $\Pant^{+}$ and $\Pant^{-}$ be disjoint subsets of $\Pant$, such that $\Pant^{+} \cup \Pant^{-}=\Pant$, and $\refl(\Pant^{+})=\Pant^{-}$ (there are many such decompositions of $\Pant$). Let $\nu^{+}$ and $\nu^{-}$ denote the restrictions of $\nu$ on the  sets $\Pant^{+}$ and $\Pant^{-}$ respectively. 
Then $\ph \nu^{+}=\ph \nu^{-}$ and $\ph \nu = 2 \ph \nu^{-}$ (this follows from the fact that $\foot(\Pi,\gamma^{*})=\foot(\refl(\Pi),-\gamma^{*})$). Therefore if the measure $\ph \nu$ is $\delta$-symmetric then so are the measures $\ph \nu^{+}$ and $\ph \nu^{-}$.
\vskip .1cm
Let $\mu$ be the measure from Proposition \ref{proposition-mes}. Then $\mu$ has rational weights. We multiply $\mu$ by a large enough integer and obtain the measure $\mu'$, such  that the weights $\mu'(\Pi)$ are even numbers, $\Pi \in \Pant$. Then $\ph \mu'$ and $ (\A_{1+i\pi})_{*} \ph \mu'$ are  
$D Re^{- {\bf q} R}$-equivalent. For simplicity we set $\mu=\mu'$. 
\vskip .1cm
Since $\mu$ is invariant under reflection and the weights are even integers, we see that $\mu \in \N \Pant$ is a $\refl$-symmetric formal sum. Let $\lab:\Lab \to \Pant^{*}$ denote the corresponding legal labeling (see the example at the beginning of this section). It remains to define an admissible involution $\sigma:\Lab \to \Lab$. 
\vskip .1cm
Fix $\gamma^{*} \in \Gamma^{*}$. Let $X^{+} \subset \Lab$ such that $ a \in X^{+}$ if  $\lab(a)=(\Pi,\gamma^{*})$, where $\Pi \in \Pant^{+}$. 
Define $X^{-}$ similarly, and let $f^{+/-}:X^{+/-} \to \Pant^{*}$ denote the corresponding restriction of the labeling map $\lab$ on the set $X^{+/-}$
(observe that $f^{+}=\refl \circ f^{-} \circ \refl_{\Lab}$). 
\vskip .1cm
Denote by $\alpha^{+}$ the restriction of $\ph \mu^{+}$ on $\NB(\sqrt{\gamma})$ (define $\alpha^{-}$ similarly). Observe that $\alpha^{+}=\alpha^{-}$. Then by the definition of $\Lab$,  the measure $\alpha^{+/-}$ is the $\ph$ of  
the push-forward of the counting measure on $X^{+/-}$ by the map $f^{+/-}$.   
\vskip .1cm
Define $g:X^{-} \to \NB(\sqrt{\gamma} )$  by $g=\A_{1+i\pi} \circ f^{-}$. Then the measure $(\A_{1+i \pi } )_{*} \alpha^{-}$ is the push-forward of the counting measure on $X^{-}$ by the map $g$. Since $\alpha^{+}$ and $(\A_{1+i \pi})_{*} \alpha^{-}$ are $2 D R e^{- {\bf q} R}$-equivalent, by Theorem \ref{thm-Hall}  there is a bijection $h:X^{+} \to X^{-}$, such that $\dist(g(h(b)),f^{+}(b)) \le 2D R e^{- {\bf q} R}$, for any $b \in X^{+}$ (recall that $\dist$ denotes the Euclidean distance on  $\NB(\sqrt{\gamma^{*}})$). 
\vskip .1cm
We define $\sigma:X^{+} \cup X^{-} \to X^{+} \cup X^{-}$ by $\sigma(x)=h(x)$ for $x \in X^{+}$, $\sigma(x)=h^{-1}(x)$ for $x \in X^{-}$.
The  map $\sigma:(X^{+}\cup X^{-}) \to (X^{+}\cup X^{-})$ is an involution. By varying $\gamma^{*}$ we construct the involution 
$\sigma:\Lab \to \Lab$. It follows from the definitions that $\sigma$ is admissible and that the pair $(\lab,\sigma)$ satisfies the assumptions of  Theorem \ref{thm-lab}.

\section{Measures on skew pants and the frame flow} We start by outlining the construction of the measures from Theorem \ref{theorem-mes}.
Fix a sufficiently small number $\epsilon>0$   and let $r>>0$ denote any large enough real number. Set $R=2(r-\log{{4}\over{3}})$. We let $\Pant_{\epsilon,R}$ be the set of skew pants $\Pi$ in $\M$ for which  $|\hl(\delta)-\frac{R}{2}| < \epsilon$ for all $\delta \in \partial{\Pi}$. In this section we will construct a measure $\mu$ on  $\Pant_{D\epsilon,R}$ (for some universal constant $D>0$)
and a measure $\beta_{\delta}$ on each $\NB(\sqrt{\delta})$ such that for $r$ large enough we have
 
$$
\left| K(\delta){{d\beta_{\delta} }\over{d \Eucl _{\delta} }}-1 \right| \le  e^{-{\bf q} r},
$$
\noindent
and the measures $\ph\mu|_{\NB(\sqrt{\delta})}$ and $\beta_{\delta}$ are $Ce^{-{{r}\over{4}} }$ equivalent,  
where $\Eucl_{\delta}$ is the Euclidean measure on $\NB(\sqrt{\delta})$, the unique probability measure invariant under $\C/(2\pi i \Z+l(\delta)\Z)$ action.
\vskip .1cm
Let $\FR(\Ho)$ denote the set of (unit) 2-frames $F_p=(p,u,n)$ where $p \in \Ho$ and  the unit tangent vectors $u$ and $n$ are orthogonal at $p$. 
By $\flow_t$, $t \in \R$,  we denote the frame flow that acts on $\FR(\Ho)$ and by 
$\Lambda$ the invariant Liouville measure on $\FR(\Ho)$.  
We then define a bounded non-negative affinity function $\abs=\abs_{\epsilon,r}:\FR(\Ho) \times \FR(\Ho) \to \R$ with the following properties 
(for $r$ large enough):

\begin{enumerate}  
\item $\abs(F_p,F_q)=\abs(F_q,F_p)$, for every $F_p,F_q, \in \FR(\Ho)$.
\item $\abs(A(F_p),A(F_q))=\abs(F_p,F_q)$, for every $A \in \PSLC$.
\item If $\abs(F_p,F_q)>0$, and $F_p=(p,u,n)$ and $F_q=(q,v,m)$ then 
$$
|d(p,q)-r|<\epsilon,
$$
$$
\Theta(n \at q,m)<\epsilon,
$$
$$
\Theta(u,v(p,q))<Ce^{-{{r}\over{4}}}, \,\, \Theta(v,v(q,p))<Ce^{-{{r}\over{4}}}.
$$
\noindent
where $\Theta(x,y)$ denotes the unoriented angle between vectors $x$ and $y$, and $v(p,q)$ denotes the unit vector at $p$ that is tangent 
to the geodesic segment from $p$ and $q$. Here $n@q$ denotes the parallel transport of $n$ along the geodesic segment from $p$ (where $n$ is based) to $q$.
\item For every co-compact group $G < \PSLC$ we have
$$
\left|\sum\limits_{A \in G} \abs(F_p,A(F_q))-{{1}\over{\Lambda(\FR(\Ho) / G)}} \right|<e^{-{\bf q}_{G} r}.
$$
\end{enumerate}

The last property will follow from the exponential mixing of the frame flow on $\FR(\Ho)/ G$.

\vskip .1cm
Now let $F_p=(p,u,n)$ and $F_q=(q,v,m)$ be two 2-frames in $\FR(\M)=\FR(\Ho)/\KG$, where $\M$ is a closed hyperbolic 3-manifold and $\KG$ is the corresponding Kleinian group. 
Let $\gamma$ be a geodesic segment in $\M$ between $p$ and $q$.
We let $\wt{F}_p$ be an arbitrary lift of $F_p$ to $\FR(\Ho)$, and let $\wt{F}_q$ be the lift of $F_q$ along $\gamma$. We let 
$\abs_{\gamma}(F_p,F_q)=\abs(\wt{F}_p,\wt{F}_q)$. By the properties $(1)$ and $(2)$ this is well defined. Moreover for any $F_p,F_q \in \FR(\M)$

\begin{equation}\label{*}
\left|\sum\limits_{\gamma} \abs_{\gamma}(F_p,F_q)-{{1}\over{\Lambda(\FR(\M) ) }} \right|<e^{-{\bf q} r},
\end{equation}
\noindent
by property $(4)$.
\vskip .1cm
We define $\omega:\FR(\Ho) \to \FR(\Ho)$ by $\omega(p,u,n)=(p,\omega(u),n)$ where $\omega(u)$ is equal to $u$ rotated around $n$ for ${{2\pi}\over{3}}$, using
the right-hand rule. Observe that $\omega^{3}$ is the identity and we let $\omega^{-1}=\overline{\omega}$. 
To any frame $F$ we associate the tripod $T=(F,\omega(F),\omega^{2}(F))$ and likewise to any frame $F$ we associate  the ``anti-tripod'' $\overline{T}=(F,\overline{\omega}(F),\overline{\omega} ^{2}(F) )$. We have the similar definitions for frames in $\FR(\M)$.
\vskip .1cm
Let $\theta$-graph be the $1$-complex comprising three one cells (called $h_0,h_1,h_2$) each connecting two $0$-cells (called $\underline{p}$ and $\underline{q}$). A connected pair of tripods is a pair of frames $F_p=(p,u,n)$, $F_q=(q,v,m)$ from $\FR(\M)$, and three geodesic segments $\gamma_i$, $i=0,1,2$, that connect $p$ and $q$ in $\M$. We abbreviate $\gamma=(\gamma_0,\gamma_1,\gamma_2)$ and we let
$$
\taso_{\gamma}(T_p,T_q)=\prod\limits_{i=0}^{2} \abs_{\gamma_{i}}(\omega^{i}(F_p),\overline{\omega} ^{i}(F_q)).
$$
\noindent
We say $(T_p,T_q,\gamma)$ is a well connected pair of tripods along the triple of segments $\gamma$ if $\taso_{\gamma}(T_p,T_q)>0$. 
\vskip .1cm 
For any connected pair of tripods  $(T_p,T_q,\gamma)$ there is a continuous map from the $\theta$-graph to $\M$ that is obvious up to homotopy (map $\underline{p}$ to $p$ and $\underline{q}$ to $q$, and $h_i$ to $\gamma_i$). If $(T_p,T_q,\gamma)$ is a well connected pair of tripods then this map will be injective on the fundamental group $\pi_1(\theta-\text{graph})$. Moreover, the resulting pair of skew pants $\Pi$ has the half-lengths $D\epsilon$ close to $\frac{R}{2}$ where  $R=2(r-\log{{4}\over{3}})$ (then the cuff lengths of the skew pants $\Pi$ are close to $R$) and $D$ is a universal constant. Recall that  the collection of skew pants whose half-lengths are $D\epsilon$ close to $\frac{R}{2}$ (for some large $R$ and fixed $\epsilon$) is called  $\Pant_{D\epsilon,R}$.

\vskip .1cm
We write $\Pi=\pi(T_p,T_q,\gamma)$, so $\pi$ maps well connected pairs of tripods to pairs of  skew pants in $\Pant_{D\epsilon,R}$. 
We define the measure $\wt{\mu}$ on well connected tripods by 
$$
d\wt{\mu}(T_p,T_q,\gamma)=\taso_{\gamma}(T_p,T_q)\, d\lambda_{T}(T_p,T_q,\gamma),
$$
\noindent
where $\lambda_{T}(T_p,T_q,\gamma)$ is the product of the Liouville measure $\Lambda$ (for $\FR(\M)$) on the first two terms, and the counting measure on the third term. The measure $\lambda_{T}$ is infinite but $\abs_{\gamma}(T_p,T_q)$ has compact support, so $\wt{\mu}$ is finite. We define the measure $\mu$ on $\Pant_{D\epsilon,R}$ by $\mu=\pi_{*} \wt{\mu}$. This is the measure from Theorem \ref{theorem-mes}.

\vskip .1cm
It remains to construct the measure $\beta_{\delta}$ and show the $Ce^{-{{r}\over{4}}}$-equivalence of $\beta_{\delta}$ and $\ph\mu|_{\NB(\sqrt{\delta})}$. 
To any frame $F$ we associate the bipod $B=(F,\omega(F))$ and likewise to any frame $F$ we associate the ``anti-bipod'' 
$\overline{B}=(F,\overline{\omega}(F))$. We have the similar definitions for frames in $\FR(\M)$.

We say that  $(B_p,B_q,\gamma_0,\gamma_1)$ is a well connected pair of bipods along the pair of segments $\gamma_0$ and $\gamma_1$ if
$$
\abs_{\gamma_{0}}(F_p,F_q) \abs_{\gamma_{1}}(\omega(F_p),\overline{\omega}(F_q))>0.
$$ 
Then the closed curve $\gamma_0 \cup \gamma_1$ is homotopic to a closed geodesic in $\M$.   
Given a closed geodesic  $\delta \in \Gamma$ we let $S_{\delta}$ be the set of well connected bipods $(B_p,B_q,\gamma_0,\gamma_1)$  such that  $\gamma_0 \cup \gamma_1$ is homotopic to $\delta$.
(Note that $S_{\delta}$ is an open subset of the space of connected bipods which is the space of quadruples $(B_p,B_q,\gamma_0,\gamma_1)$, where $B_p$ and $B_q$ are tripods and $\gamma_0$ and $\gamma_1$  are geodesic  segments in $\M$ connecting the points $p$ and $q$). The  set $S_{\delta}$  of connected bipods carries the natural measure $\lambda_{B}$  which is the product of the Liouville measures on the first two terms and the counting measure on  the third and fourth.

\begin{remark} One can show that if $\epsilon$ is small in terms of the injectivity radius of $\M$ then for two bipods $B_p$ and $B_q$ in $\FR(\M)$ there exists at most one pair of segments $(\gamma_0,\gamma_1)$ such that $(B_p,B_q,\gamma_0,\gamma_1)$ is a well connected pair of bipods and that $\gamma_0 \cup \gamma_1$ is homotopic to $\delta$. However, we do not use this.
\end{remark}

Next, we define the action of the torus $\C/(2\pi i \Z+l(\delta)\Z)$  on $S_{\delta}$ that leaves the measure $\lambda_{B}$ invariant.
\vskip .1cm
Let $\TT_{\delta}$ be the open solid torus cover associated to $\delta$ (so $\delta$ lifts to a closed geodesic $\wt{\delta}$ in $\TT_{\delta}$).
Given a pair of well connected bipods in $S_{\delta}$, each bipod lifts  in a unique way to a bipod in  $\FR(\TT_{\delta})$ such that the pair of the lifted bipods is well connected in $\TT_{\delta}$.
We denote by $\wt{S}_{\delta}$ the set of such lifts, so  $\wt{S}_{\delta}$ is in one-to-one correspondence with $S_{\delta}$. There is a natural action of the torus $\C/(2\pi i \Z+l(\delta)\Z)$ on both $\NB(\delta)$ ($=\NB(\wt{\delta})$) and on $\FR(\TT_{\delta})$, and hence on $\wt{S}_{\delta}$ as well. Since $\wt{S}_{\delta}$ and $S_{\delta}$ are in one-to-one correspondence we have the induced action of $\C/(2\pi i \Z+l(\delta)\Z)$  on $S_{\delta}$. This action leaves invariant the measure $\lambda_{B}$ on $S_{\delta}$.

For either choice of $\hl(\delta)$ there is a natural action of $\C/(2\pi i \Z+l(\delta)\Z)$ on $\NB(\sqrt{\delta})$ via  $\C/(2\pi i \Z+\hl(\delta)\Z)$. We define in Section 4.7 a map 
$\ft_{\delta}:S_{\delta} \to \NB(\sqrt{\delta})$ with two important properties. The first one is that $\ft_{\delta}$   is equivariant with respect to the action of  $\C/(2\pi i \Z+l(\delta)\Z)$. 
The second  property is as follows.

Let $C_{\delta}$ be the set of well connected tripods $(T_p,T_q,\gamma)$ for which $\gamma_0 \cup \gamma_1$ is homotopic to $\delta$, and let $\chi:C_{\delta} \to S_{\delta}$ 
be the forgetting map, so $\chi(T_p,T_q,\gamma_0,\gamma_1,\gamma_2)=(B_p,B_q,\gamma_0,\gamma_1)$.  Then for any pair of well connected tripods  $T=(T_p,T_q,\gamma) \in C_{\delta}$
\begin{equation}\label{f-foot}
|\ft_{\delta}(\chi (T) )-\foot_{\delta}(\pi(T))|<Ce^{-{{r}\over{4}}},
\end{equation}
\noindent
where $\pi(T)$ is the skew pants defined above (recall that the map $\foot_{\delta}(\Pi)$ that associates the foot to a pair of marked skew pants $(\Pi,\delta)$, $\delta \in \partial{\Pi}$,
was defined in Section 3).  In other words, the map $\ft_{\delta}$ predicts feet of the skew pants $\pi(T)$ (just by knowing the pair of well connected bipods  $\chi(T)$) up to an error of $Ce^{-{{r}\over{4}}}$. This $Ce^{-{{r}\over{4}}}$ comes from the property $(3)$ of the affinity function $\abs$ defined above.
\vskip .1cm
There are two more natural measures on $S_{\delta}$. The first is $\chi_{*}(\wt{\mu}|_{C_{\delta}})$. The second is $\nu_{\delta}$ defined on $S_{\delta}$ by
$$
d\nu_{\delta}(B_p,B_q,\gamma_0,\gamma_1)=\abs_{\gamma_{0}}(F_p,F_q)\abs_{\gamma_{1}}(\omega(F_p),\overline{\omega}(F_q))\, d\lambda_{B}(F_p,F_q,\gamma_0,\gamma_1),
$$
\noindent
where we recall that $\lambda_{B}(F_p,F_q,\gamma_0,\gamma_1)$ is the product of the Liouville measure on the first two terms and the counting measure on the last two.
The two measures satisfy the fundamental inequality
\begin{equation}\label{**}
\left| {{d \chi_{*}(\wt{\mu}|_{C_{\delta}}) }\over{ d\nu_{\delta}(B_p,B_q,\gamma_0,\gamma_1)}}-{{1}\over{\Lambda(\FR(\M))}} \right|<Ce^{-{\bf q} r},
\end{equation}
\noindent
because the total affinity between $\omega^{2}(F_p)$ and $\overline{\omega}^{2}(F_q)$ (summing over all positive connections $\gamma_2$) is exponentially close to ${{1}\over{\Lambda(\FR(\M))}}$ by the inequality (\ref{*}) above. 

Moreover, since $\lambda_{B}$ and the product $\abs_{\gamma_{0}}(F_p,F_q)\abs_{\gamma_{1}}(\omega(F_p),\overline{\omega}(F_q))$  are both invariant under the action of  $\C/(2\pi i \Z+l(\delta)\Z)$, 
we see that $\nu_{\delta}$ is also invariant under the action of $\C/(2\pi i \Z+l(\delta)\Z)$. 
Therefore $(\ft_{\delta})_{*}\nu_{\delta}$ is as well, because $\ft_{\delta}$ is  $\C/(2\pi i \Z+l(\delta)\Z)$ equivariant. It follows that $(\ft_{\delta})_{*}(\nu_{\delta})=K_{\delta}\Eucl_{\delta}$, for some constant $K_{\delta}$. Therefore, by (\ref{**}) 
\begin{equation}\label{***}
\left|{{ d(\ft_{\delta})_{*}(\wt{\mu}|_{C_{\delta}} ) }\over{ d K'_{\delta}\Eucl_{\delta} }}-1 \right|<Ce^{-{\bf q}r},
\end{equation}
\noindent
where $K'_{\delta}=K_{\delta}/\Lambda(\FR(\M))$.
\vskip .1cm
This measure $(\ft_{\delta} )_{*} (\wt{\mu} |_{C_{\delta}} )$ is our desired measure $\beta_{\delta}$; it is $Ce^{-{{r}\over{4}}}$-equivalent to the measure $\ph \mu|_{\NB(\sqrt{\delta} ) }$ because the later measure is just $(\foot_{\delta} )_{*}  \pi_{*}(\wt{\mu}|_{C_{\delta}} )$, and as we already said
$$
|\ft_{\delta}(\chi(T))-\foot_{\delta}(\pi(T))|<Ce^{-{{r}\over{4}}},
$$
\noindent 
for every tripod $T$ in $C_{\delta}$.

\smallskip

We define $\abs(F_p, F_q)$ in Section 4.4,
and we prove that the skew pants $\pi(T_p, T_q, \gamma) \in \Pant_{D\epsilon,R}$  (for some universal constant $D>0$)
when $\taso_\gamma(T_p, T_q) > 0$ in Sections 4.5 and 4.6, using preliminaries developed in Section 4.1 and 4.2.
We define $\ft_\delta$ and prove \eqref{f-foot} in Section 4.7,
using preliminaries developed in Section 4.3. 
Finally we prove equation \eqref{***} in Section 4.8.

\subsection{The Chain Lemma} Let $T^{1}(\Ho)$ denote the unit tangent bundle. Elements of  $T^{1}(\Ho)$ are pairs $(p,u)$, where $p \in \Ho$ and $u \in T^{1}_p(\Ho)$.  For $u,v \in T^{1}_p(\Ho)$ we let $\Ang(u,v)$ denote the unoriented angle between $u$ and $v$. The function $\Ang$ takes values in the interval $[0,\pi]$. 
For $a,b \in \Ho$ we let $v(a,b) \in T^{1}_a(\Ho)$ denote the unit vector at $a$ that points toward $b$. If $v \in T^{1}_a(\Ho)$ then $v \at b \in T^{1}_b(\Ho)$ denotes the vector parallel transported to $b$ along the geodesic segment connecting $a$ and $b$. By $(a,b,c)$ we denote the hyperbolic triangle with vertices  $a,b,c \in \Ho$. For two points $a,b \in \Ho$, we let $|ab|=d(a,b)$.

\begin{proposition}\label{prop-tra-1} Let $a,b,c \in \Ho$ and $v \in T^{1}_a(\Ho)$. Then the inequalities  
$$
\Ang(v \at b \at c \at a,v)\le \text{Area}(abc) \le  |bc|,
$$
\noindent
hold,  where  $\text{Area}(abc)$ denotes the hyperbolic area of the triangle $(a,b,c)$. 
\end{proposition}

\begin{proof} It follows from the Gauss-Bonnet  theorem that the inequality $\Ang(v \at b \at c \at a,v)\le \text{Area}(abc)$ holds for every $v \in T^{1}_a(\Ho)$. Moreover, if $v$ is in the plane of the triangle $(a,b,c)$, then the equality  $\Ang(v \at b \at c \at a,v)= \text{Area}(abc)$ holds. 
\vskip .1cm

We now prove that in every hyperbolic triangle the length of a side is greater than the area of the triangle, that is we prove $|bc| \ge \text{Area}(abc)$. Consider the geodesic ray that starts at $b$ and that contains $a$, and let  $a' \in \partial{\Ha}$ be the point where this ray hits the ideal boundary. Then the triangle $(a,b,c)$ is contained in the triangle $(a',b,c)$, so it suffices to show that $\text{Area}(a',b,c) \le |bc|$. Thus we may assume that the vertex $a$ is a point on $\partial{\Ha}$.

Considering the standard model of the upper half plane $\Ha=\{z \in \C: \IM(z)>0 \}$, we can assume that $a=\infty$ and that the geodesic segment $(bc)$ lies on the unit circle $\{z \in \C:|z|=1 \}$.
By the first part of the proposition we know that $\text{Area}(abc)$ is  equal to $\alpha$, where $\alpha$ is the unoriented angle between the Euclidean lines $l_b$ and $l_c$, where $l_b$ contains $0$ and $b$ and $l_c$ contains $0$ and $c$ ($0 \in \C$ denotes the origin). Since $b$ and $c$ lie on the unit circle we have that $\alpha$ is also equal to the Euclidean length of the arc of the unit circle between $b$ and $c$. On the other hand, the hyperbolic length of this arc (which is the geodesic segment $(bc)$ between $b$ and $c$ in the hyperbolic metric) is strictly larger than $\alpha$ because the density of the hyperbolic metric is $y^{-1}|dz|$, which is greater  than $1$ on the unit circle. We have
$$
\text{Area}(abc) \le \alpha \le |bc|,
$$
which proves the proposition.

\end{proof}

The following claim will be used in the proof of Theorem \ref{thm-chain} below.

\begin{claim}\label{claim-1} Let $a,b,c \in \Ho$. Then the inequality 
$$
\Ang( v(c,a), v(b,a) \at c ) \le \Ang( v(a,b),v(a,c) ) + \text{Area}(abc) \le \Ang(v(a,b),v(a,c)) + |bc|
$$
holds.

\end{claim}

\begin{proof} We have

\begin{align*} 
\Ang( v(c,a), v(b,a) \at c ) &= \Ang( v(c,a) \at a, v(b,a) \at c \at a ) \\
&= \Ang( -v(a,c), v(b,a) \at c \at a ) \\
&\le  \Ang( -v(a,c), -v(a,b))+ \Ang(-v(a,b),v(b,a) \at c \at a) \\
&= \Ang(v(a,c), v(a,b))+ \Ang(-v(a,b) \at b, v(b,a) \at c \at a \at b) \\
&= \Ang(v(a,c), v(a,b))+ \Ang(v(b,a), v(b,a) \at c \at a \at b).
\end{align*}
By the previous proposition we have $\Ang(v(b,a), v(b,a) \at c \at a \at b) \le \text{Area}(abc) \le |bc|$, and we are finished.

\end{proof}

The following two propositions are elementary and follow from the $\cosh$ rule for hyperbolic triangles.

\begin{proposition}\label{prop-tra-2} Let $(a,b,c)$ be a hyperbolic triangle such that $|ab|=l_1$ and $|bc|=\eta$. 
Then for $l_1$ large and $\eta$ small enough, the inequality
$\Ang(v(a,b),v(a,c)) \le D \eta e^{-l_{1}}$ holds, for some constant $D>0$.
\end{proposition}

\begin{proposition}\label{prop-tra-3} Let $(a,b,c)$ be a hyperbolic triangle and set $|ab|=l_1$, $|cb|=l_2$ and $|ac|=l$. Let $\eta=\pi-\Ang(v(b,a),v(b,c))$.  Then  for $l_1$ and $l_2$ large we have
\begin{enumerate}
\item $|(l-(l_1+l_2))+\log 2 -\log(1+\cos\eta) | \le D e^{- 2\min\{l_{1},l_{2} \} }$, for any  $0 \le \eta \le {{\pi}\over{2}}$,
\item $|l-(l_1+l_2)| \le D \eta$, for   $\eta$ small, 
\item $\Ang(v(a,c),v(a,b)) \le D \eta e^{-l_{1}}$,  for   $\eta$ small, 
\item $\Ang(v(c,a),v(c,b)) \le D \eta e^{-l_{2}}$,  for   $\eta$ small, 
\end{enumerate}
for some constant $D>0$. 
\end{proposition}

The following Theorem (the ``Chain Lemma'') allows us to estimate the geometry of a segment that is formed from a chain of long segments that nearly meet at their endpoints. It will be used in Section 4.2 to estimate the complex length of a closed geodesic formed from a closed chain of such segments. 
\begin{theorem}\label{thm-chain} Suppose $a_i,b_i \in \Ho$, $i=1,...,k$, and
\begin{enumerate} 
\item $|a_i b_i| \ge Q$,
\item $|b_i a_{i+1}| \le \epsilon$,
\item $\Ang(v(b_i,a_i) \at a_{i+1},-v(a_{i+1},b_{i+1})) \le \epsilon$.
\end{enumerate}
Suppose also that $n_i \in T^{1}_{a_{i}}(\Ho)$ is a vector at $a_i$ normal to $v(a_i,b_i)$ and
$$
\Ang(n_i  \at b_i,n_{i+1} \at b_i) \le \epsilon.
$$
\noindent
Then for $\epsilon$ small and $Q$ large, and some constant $D>0$
\begin{equation}\label{s-1} 
\left| ||a_1 b_k|-\sum\limits_{i=1}^{k} |a_ib_i| \right| \le k D \epsilon,
\end{equation}

\begin{equation}\label{s-2} 
\Ang(v(a_1,b_k),v(a_1,b_1))<kD\epsilon e^{-Q},\,\, \text{and} \,\,\,  \Ang(v(b_k,a_1),v(b_k,a_k))<kD\epsilon e^{-Q},
\end{equation}

\begin{equation}\label{s-2-1} 
\Ang(v(a_1,a_k),v(a_1,b_1))<2kD\epsilon e^{-Q},\,\, \text{if} \,\, k>1 ,
\end{equation}

\begin{equation}\label{s-3}
\Ang(n_k,n_1 \at a_k) \le 5k \epsilon.
\end{equation}
\end{theorem}

We can think of the sequence of geodesic segments from $a_i$ to $b_i$ as forming an ``$\epsilon$-chain'', and we can think of the broken segment connecting $a_1, b_1, a_2, b_2, \ldots, a_k, b_k$ as the concatenation of the $\epsilon$-chain, and the geodesic segment from $a_1$ to $b_k$ (or $a_k$) as the geodesic representative of the concatenation. Then the Chain Lemma is describing the relationship between the concatenation of an $\epsilon$-chain and its geodesic representative, and also estimating the discrepancy between parallel transport along the concatenation, and transport along its geodesic representative. 

\begin{proof} By induction. Suppose that the statement is true for some $k \ge 1$. We need to prove the above inequalities for $k+1$.
\vskip .1cm
We first prove the inequalities (\ref{s-2}) and (\ref{s-2-1}). By the triangle inequality we have  
$$
\Ang(v(a_1,b_k),v(a_1,b_{k+1})) \le \Ang(v(a_1,b_k),v(a_1,a_{k+1}))+\Ang(v(a_1,a_{k+1}),v(a_1,b_{k+1})).
$$
\noindent

By Proposition \ref{prop-tra-2} we have $\Ang(v(a_1,b_k),v(a_1,a_{k+1}))\le D_1\epsilon e^{-Q}$, where $D_1$ is the constant from   Proposition \ref{prop-tra-2}. By  (\ref{pom}) and Proposition \ref{prop-tra-3} we have $\Ang(v(a_1,a_{k+1}),v(a_1,b_{k+1}))\le 2D_2\epsilon e^{-Q}$, where $D_2$ is from Proposition \ref{prop-tra-3}. Together this shows
$$
\Ang(v(a_1,b_k),v(a_1,b_{k+1}))\le D\epsilon e^{-Q}.
$$
\noindent
Then by the triangle inequality and the induction hypothesis we have
\begin{align}
\Ang(v(a_1,b_{k+1}),v(a_1,b_1 ) ) & \le \Ang(v(a_1,b_k),v(a_1,b_1 ) ) +   \Ang(v(a_1,b_k),v(a_1,b_{k+1})) \notag \\
& \le  kD\epsilon e^{-Q}+D\epsilon e^{-Q}= (k+1)D\epsilon e^{-Q} \notag ,
\end{align}
\noindent
which proves the first inequality in (\ref{s-2}). The second one follows by symmetry. The inequality (\ref{s-2-1}) follows from (\ref{s-2}) and Proposition \ref{prop-tra-2}.
\vskip .1cm
Next, we prove the inequality (\ref{s-1}). By the triangle inequality we have
$$
\Ang(v(a_1,a_{k+1}),v(a_1,b_k)) \le \Ang(v(a_1,a_{k+1}),v(a_1,b_1) ) + \Ang(v(a_1,b_1),v(a_1,b_k)),
$$
and then applying (\ref{s-2}) and (\ref{s-2-1})  we get
$$
\Ang(v(a_1,a_{k+1}),v(a_1,b_k)) \le 2(k+1)D \epsilon e^{-Q} +kD \epsilon e^{-Q}< \epsilon
$$
for $Q$ large enough. Then by Claim \ref{claim-1} we have
$$
\Ang(v(a_{k+1},a_1),v(b_k,a_1) \at a_{k+1}) \le \Ang(v(a_1,a_{k+1}),v(a_1,b_k))+|b_k a_{k+1}| \le 2 \epsilon.
$$
Combining this inequality with the assumption (3) of the theorem, and by the triangle inequality we obtain
$\Ang(v(a_{k+1},a_1),-v(a_{k+1},b_{k+1}) ) \le 3\epsilon$. Therefore the inequality

\begin{equation}\label{pom}
\pi-\Ang(v(a_{k+1},a_1),v(a_{k+1},b_{k+1}) ) \le 3\epsilon
\end{equation}
\noindent
holds (observe that the same inequality holds for all $1 \le i \le k$).
\vskip .1cm

It follows from  Proposition \ref{prop-tra-3} and (\ref{pom}) that $\big| |a_1a_{k+1}|+|a_{k+1} b_{k+1}|- |a_1b_{k+1}| \big| \le 3D_1\epsilon$, where $D_1$ is the constant from Proposition \ref{prop-tra-3}. Since by the triangle inequality
$$
\big| |a_1b_k|-|a_1a_{k+1}| \big|\le \epsilon,
$$ 
\noindent
we obtain
$$
\big| |a_1b_k|+|a_{k+1} b_{k+1}|- |a_1b_{k+1}| \big| \le D\epsilon.
$$
\noindent
This proves the induction step for the inequality (\ref{s-1}).
\vskip .1cm

It remains to prove (\ref{s-3}). Using the induction hypothesis and the assumptions in the statement of this theorem, we obtain the following string of inequalities

\begin{align*}
\Ang(n_{k+1} , n_1 \at a_{k+1} ) &= \Ang (n_{k+1} \at b_k , n_1 \at a_{k+1} \at b_k ) \\
&\le \Ang(n_{k+1} \at b_k , n_k \at b_k ) + \Ang(n_k \at b_k , n_1 \at a_{k+1} \at b_k) \\
&\le \epsilon + \Ang(n_k \at b_k , n_1 \at a_{k+1} \at b_k) \\
&\le \epsilon  + \Ang(n_k \at b_k , n_1 \at a_k \at b_k ) + \Ang(n_1 \at a_k \at b_k, n_1 \at a_{k+1}  \at b_k ) \\
&\le  (5k + 1)\epsilon + \Ang(n_1 \at a_k \at b_k , n_1 \at a_{k+1} \at b_k) \\
&\le  (5k + 1)\epsilon + \Ang(n_1 \at a_k  \at b_k , n_1 \at b_k ) + \Ang(n_1 \at b_k , n_1 \at a_{k+1} \at b_k). 
\end{align*}
By (\ref{pom}) we have 
$$
\Ang(n_1 \at a_k  \at b_k , n_1 \at b_k ) \le 3\epsilon,
$$
and by Claim  \ref{claim-1} we have 
$$
\Ang(n_1 \at b_k , n_1 \at a_{k+1} \at b_k) \le  \epsilon.
$$
Combining these estimates gives
$$
\Ang(n_{k+1} , n_1 \at a_{k+1} ) \le  (5k + 5)\epsilon,
$$
which proves the induction step for (\ref{s-3}).

\end{proof}

\subsection{Corollaries of the Chain Lemma} For $X \in \PSLC$ we write $X(z)={{az+b}\over{cz+d}}$, where $ad-bc=1$.

The following proposition will provide the bridge between the Chain Lemma and Lemma \ref{lemma-chain}.

\begin{proposition}\label{prop-M} Let $p,q \in \Ho$ and  $A \in \PSLC$ be such that $A(p)=q$. Suppose that for every $u \in T^{1}_{p}(\Ho)$ we have $\Ang(A(u),u \at q)\le \epsilon$. Then for $\epsilon$ small enough and  $d(p,q)$  large enough, and for some constant $D>0$  we have  
\begin{enumerate}
\item the transformation $A$  is loxodromic,
\item $|\len(A)-d(p,q)|\le D\epsilon$, 
\item if $\text{axis}(A)$ denotes the axis of $A$ then $d(p,\text{axis}(A)),d(q,\text{axis}(A) ) \le D\epsilon$. 
\end{enumerate}
\end{proposition}

\begin{proof} We may assume that the points $p$ and $q$ lie on the  geodesic that connects $0$ and $\infty$, such that  $q$ is the point with coordinates $(0,0,1)$ in $\Ho$, and  
$p$ is $(0,0,x)$ for some $0<x<1$. Let $B \in \PSLC$ be given by $B(z)=Kz$, where $\log K=d(p,q)$. Since $K$ is a positive number it follows that for every $u \in T^{1}_p(\Ho)$ the identity $B(u)=u \at q$.
\vskip .1cm
Let  $A=C \circ B$, where $C \in \PSLC$ fixes the point $(0,0,1)\in \Ho$. It follows that for every $u \in T^{1}_{(0,0,1)}(\Ho)$ we have 
$\Ang(u,C(u))\le \epsilon$. This implies that for some $a,b,c,d \in \C$, $ad-bc=1$, we have 
$$
C(z)={{az+b}\over{cz+d}}, 
$$
\noindent
and $|a-1|,|b|,|c|,|d-1| \le D_1\epsilon$, for some constant $D_1>0$. Then
$$
A(z)={{a\sqrt{K}z+{{b}\over{\sqrt{K}}} }\over{\sqrt{K}cz+ {{d}\over{\sqrt{K}}} }},
$$
\noindent
and we find  
$$
\tr(A)=a\sqrt{K}+{{d}\over{\sqrt{K}}},
$$ 
\noindent
where $\tr(A)$ denotes the trace of $A$. Since $|a-1|, |d-1| \le D_1\epsilon$ we see that for $K$ large enough the real part of the trace $\tr(A)$ is a positive number $>2$,  which shows that $A$ is loxodromic. 
On the other hand, $\tr(A)=2\cosh({{\len(A)}\over{2}})$. This shows that $|\len(A)-\log K| \le D_2\epsilon$, for some constant $D_2>0$. 
\vskip .1cm
Let $z_1,z_2 \in \overline{\C}$ denote the fixed points of $A$. We find
$$
z_{1,2}={{ (a-{{d}\over{K}}) \pm  \sqrt{ ( a-{{d}\over{K}} )^{2}+{{4bc}\over{K}} } }\over{2c }}.
$$
\noindent
Then for $K$ large enough we have  
$$
|z_1| \le \epsilon ,\,\, |z_2| \ge {{3}\over{\epsilon}}.
$$
\noindent
This shows that $d(q,\text{axis}(A))=d((0,0,1),\text{axis}(A)) \le D_3 \epsilon$, for some constant $D_3>0$. The inequality 
$d(p,\text{axis}(A) ) \le D_3 \epsilon$ follows by symmetry.

\end{proof}

The following lemma is a corollary of Theorem \ref{thm-chain} and the previous proposition.  It provides an estimate for the complex length of the closed geodesic that is freely homotopic 
to the concatenation of a closed chain of geodesic segments.
\begin{lemma}\label{lemma-chain} Let $a_i,b_i \in \Ho$, $i \in \Z$ such that 
\begin{enumerate} 
\item $|a_i b_i| \ge Q$,
\item $|b_i a_{i+1}| \le \epsilon$,
\item $\Ang(v(b_i,a_i) \at a_{i+1},-v(a_{i+1},b_{i+1})) \le \epsilon$.
\end{enumerate}
Suppose also that $n_i \in T^{1}_{a_{i}}(\Ho)$ is a vector at $a_i$ normal to $v(a_i,b_i)$ and
$$
\Ang(n_i  \at b_i,n_{i+1} \at b_i) \le \epsilon.
$$
\noindent
Suppose there exists $A \in \PSLC$ and $k>0$ be such that  $A(a_i)=a_{i+k}$, $A(b_i)=b_{i+k}$, and $A(n_i)=n_{i+k}$, $i \in \Z$.  Then for $\epsilon$ small and $Q$ large $A$ is a  loxodromic transformation and 
\begin{equation}\label{A-1} 
\left| \len(A)-\sum\limits_{i=0}^{k} |a_ib_i| \right| \le kD \epsilon,
\end{equation}
\noindent
for some constant $D>0$. Moreover $a_i,b_i \in \Ne_{D k \epsilon}(\text{axis}(A))$, where $\Ne_{D k \epsilon}(\text{axis}(A)) \subset \Ho$ is the 
$D k \epsilon$ neighbourhood of $\text{axis}(A)$.
\end{lemma}

We can think of taking $a_i,b_i \in \Ho / A$ (or even in some hyperbolic $3$-manifold  $N$), and $i \in \Z / k\Z$. We must then describe the geodesic segments from $a_i$ to $b_i$ which we will use  to determine $v(b_i,a_i)$ and $n_i \at b_i$, and so forth. (As long as the injectivity radius of $N$  is greater than $\epsilon$,  there are unique choices of geodesic segments from $b_i$ to $a_{i+1}$ with length less than $\epsilon$.) We then 
think of this sequence of segments as a ``closed $\epsilon$-chain'' and $\text{axis}(A)/A$ as its geodesic representative. 

\begin{proof}  Let $v_0=v(a_0,b_0)$. Observe that $A(v_0)=v(a_{k},b_{k})$. First we show  that the inequality $\Ang(A(v_0),v_0 \at a_k) \le 4\epsilon$ holds for $Q$ large enough.

Recall that for $Q$ large enough the inequality (\ref{pom}) holds (see the proof of Theorem \ref {thm-chain}), that is we have
$$
\pi-\Ang(v(a_{k},a_0),v(a_{k},b_{k}) ) \le 3\epsilon.
$$
Since $\Ang(v(a_{k},a_0),-v(a_{k},b_{k}) )= \Ang(v(a_0,a_k),v(a_{k},b_{k}) \at a_0 )$, we have
$$
\Ang(v(a_0,a_k),v(a_{k},b_{k}) \at a_0 ) \le 3\epsilon.
$$

On the other hand, it follows from (\ref{s-2-1}) that for $Q$ large enough we have 
$$
\Ang(v(a_0,a_k),v(a_{0},b_{0})  ) \le \epsilon, 
$$
so by the triangle inequality we obtain

\begin{align*}
\Ang(v(a_0,b_0),v(a_{k},b_{k}) \at a_0 ) &\le   
\Ang(v(a_0,a_k),v(a_{0},b_{0}))+    \Ang(v(a_0,a_k),v(a_{k},b_{k}) \at a_0 ) \\   
&\le 4\epsilon.
\end{align*}

Since $v(a_k,b_k) \at a_0 \at a_k=v(a_k,b_k)$ we find 
$\Ang(v(a_0,b_0),v(a_{k},b_{k}) \at a_0 )=\Ang(v_0 \at a_k, A(v_0) )$ so we have proved the inequality  $\Ang(v_0 \at a_k, A(v_0) ) \le 4 \epsilon$.

Next, from (\ref{s-3}) we find $\Ang(n_{k},n_0 \at a_k) \le 4k \epsilon$. Since $v_0$ is normal to $n_0$, and the parallel transport preserves   angles, it follows that 
\begin{equation}\label{esta-1}
\Ang(u \at a_k,A(u)) \le 4kD \epsilon,\,\,\,\text{for every vector}\,\, u \in T^{1}_{a_{0}}(\Ho).
\end{equation}
\noindent
On the other hand, the inequality
\begin{equation}\label{esta-2} 
\left| d(a_0, A(a_0)) - \sum\limits_{i=0}^{k} |a_ib_i| \right| \le kD\epsilon,
\end{equation}
\noindent
follows from (\ref{s-1}). The lemma now follows from Proposition \ref{prop-M}.

\end{proof}

\subsection{Preliminary propositions}    In this subsection we will prove two results in hyperbolic geometry (Lemmas  \ref{lemma-ph-1} and \ref{lemma-ph-2} ).
that we will use in Section \ref{section-4.7}.  The following proposition is elementary

\begin{proposition}\label{elem-0} Let $\alpha$ be an geodesic in $\Ho$ and let $p_1,p_2 \in \Ho$ be two points such that $d(p_1,p_2) \le C$ and $d(p_i,\alpha)\ge s$, $i=1,2$, for some constants $C,s>0$. Let $\eta_i$ be the oriented geodesic that contains $p_i$ and is normal to $\alpha$, and that is oriented from $\alpha$ to $p_i$. Then there exists a constant $D>0$, that depends only on $C$, such that 
$|\dis_{\alpha}(\eta_1,\eta_2)| \le De^{-s}$.
\end{proposition}
\vskip .1cm
Let $\alpha, \beta$ be two oriented geodesics in $\Ho$  such that $d(\alpha,\beta)>0$ and let $\gamma$ be their common orthogonal that is oriented from $\alpha$ to $\beta$.  
We observe that both $\alpha$ and $\beta$ are mapped to $-\alpha$ and $-\beta$ respectively,  by a $180$ degree rotation around $\gamma$.  Let $t \in \R$ and let $q:\R \to \beta$ be parametrisation by arc length such that $q_0(0)=\beta \cap \gamma$.  Let $\delta(t)$ be the geodesic that contains $q_0(t)$ and is orthogonal to $\alpha$, and is oriented from $\alpha$ to $q_0(t)$. The following proposition follows from the symmetry of $\alpha$ and $\beta$ around $\gamma$. Recall that the complex distance is well defined $\pmod {2\pi i}$, so we can always choose a complex distance such that its imaginary part is in the interval $(-\pi,\pi]$.

\begin{proposition}\label{elem-1} Assume that $\alpha \ne \beta$.  Then 
$\dis(q_0(t_1),\alpha)=\dis(q_0(t_2),\alpha)$ if and only if $|t_2|=|t_1|$. Moreover, if for some  $t \in \R$ 
we can choose the complex distance $\dis_{\alpha}(\delta(-t),\delta(t))$ such that 
$$
-\pi < \IM\big( \dis_{\alpha}(\delta(-t),\delta(t)) \big)< \pi,
$$
\noindent
then 
\begin{equation}
\dis_{\alpha}(\delta(-t),\gamma)={{1}\over{2}}\dis_{\alpha}(\delta(-t),\delta(t)).
\end{equation}
\end{proposition}
\vskip .1cm
\begin{remark} Observe that if $\alpha$ and $\beta$ do not intersect we can always  choose the complex distance $\dis_{\alpha}(\delta(-t),\delta(t))$ such that 
$$
-\pi < \IM\big( \dis_{\alpha}(\delta(-t),\delta(t)) \big)< \pi.
$$
\end{remark}
\vskip .1cm
Assuming the above notation we have the following.

\begin{proposition} Let $s(t)=d(q_0(t),\alpha)$. Suppose  that $d(\alpha,\beta) \le 1$.
Then for $s(t)$ large enough we have 
\begin{align*}
&s(t+h)=s(t)+h+o(1), \,\, \text{as} \,\, t \to \infty  \\ 
&s(t+h)=s(t)-h+o(1), \,\, \text{as}\,\, t \to \infty \\
\end{align*}
\noindent
for any $|h|\le {{s(t)}\over{2}}$.
\end{proposition}

\begin{proof} By the triangle inequality we have

\begin{align*}
s(t) &= d(q_0(t),\alpha) \\ 
&\le  d(q_0(t), q_0(0)) + d(q_0(0),\alpha) \\ 
&\le |t| + 1,
\end{align*}
since $ d(q_0(0),\alpha)=d(\alpha,\beta) \le 1$. That is, we have
\begin{equation}\label{in-1}
s(t) \le |t|+1.
\end{equation}
\noindent
It follows from (\ref{in-1}) that $s(t)$ large implies that $|t|$ is large. 

Recall the following formula (\ref{epstein}) from Section 1.
$$
\sinh^{2}(d(q_0(t),\alpha))=\sinh^{2}(d(\alpha,\beta) ) \cosh^{2}(t)+
\sin^{2}(\IM[\dis_{\gamma}(\alpha,\beta)]) \sinh^{2}(t).
$$
\noindent
Combining this with (\ref{in-1}) we get
$$
e^{2s(t)}=e^{2t}\big(\sinh^{2}(d(\alpha,\beta) )+ \sin^{2}(\IM[\dis_{\gamma}(\alpha,\beta)]) \big )+O(1),
$$
\noindent
which proves the proposition.

\end{proof}

\vskip .1cm

We can define the foot of the geodesic $\beta$ on $\alpha$ as the normal to $\alpha$ pointing along $\gamma$.  The lemma below estimates how the foot of $\beta$ on $\alpha$ moves when $\beta$ is moved (and $\beta$ is very close to $\alpha$).

Let $\epsilon \in \D$ be a complex number and let $r>0$. Assume that 
\begin{equation}\label{assumption}
\dis_{\gamma}(\alpha,\beta)=e^{-{{r}\over{2}}+\epsilon}.
\end{equation}
\noindent
Then there exists $\epsilon_0>0$ such that for every for $|\epsilon|<\epsilon_0$, for every $r>1$, and for every $t \in \R$ we can choose the complex distance $\dis_{\alpha}(\delta(-t),\delta(t))$ such that 
\begin{equation}\label{d-1}
-{{\pi}\over{4}} < \IM  \dis_{\alpha}(\delta(-t),\delta(t)) < {{\pi}\over{4}}.
\end{equation}

\vskip .1cm
Let $\beta_1$ be another geodesic with   a parametrisation by arc length $q_1:\R \to \beta_1$.  We let $\gamma_1$,  denote the common orthogonal between $\alpha$ and $\beta_1$, that is oriented from $\alpha$ to $\beta_1$. We have

\begin{lemma}\label{lemma-ph-1} Assume  that $\alpha$ and $\beta$ satisfy (\ref{assumption}). Let $C>0$
and suppose that for some $t_1,t'_1,t_2,t'_2 \in \R$, where $t_1<0<t_2$, we have  
\begin{enumerate}
\item $d(q_1(t'_1),q_0(t_1)),d(q_1(t'_2),q_0(t_2)) \le C$, 
\item  $|d(q_0(t_1),\alpha)-d(q_0(t_2),\alpha)| \le C$,
\item $d(q_0(t_1),\alpha)>{{r}\over{4}}-C$.
\end{enumerate}
Then for $|\epsilon|<\epsilon_0$ and for $r$ large, we have 
$$
\dis_{\alpha}(\gamma,\gamma_1) \le D e^{-{{r}\over{4}}},
$$
\noindent
for some constant $D>0$, where $D$ only depends on $C$.
\end{lemma}

\begin{proof}  The constants $D_i$ defined below all depend only on $C$. 
From (\ref{assumption}) we have  $d(\alpha,\beta)<e^{-{{r}\over{2}}+1}$.
Since  
$$
d(q_1(t'_1),q_0(t_1)),d(q_1(t'_2),q_0(t_2)) \le C, 
$$
\noindent
it follows that for $r$ large we have  $d(\alpha,\beta_1)=o(1)$, and in  particular we have $d(\alpha,\beta_1)<1$.
By the triangle inequality we obtain $|d(q_1(t'_1),\alpha)-d(q_1(t'_2),\alpha)| \le D_1$.
Then it  follows from the previous proposition that the  inequalities $|t_2+t_1|,|t'_2+t'_1| \le D_2$ hold. 
This implies that $d(q_0(-t_1),q_1(-t'_1)) \le D_3$. 
\vskip .1cm
Let $\delta_1(t)$ be the geodesic that contains $q_1(t)$ and is orthogonal to $\alpha$, and is oriented from $\alpha$ to $q_1(t)$.
Now we apply Proposition \ref{elem-0} and find that
$$
|\dis_{\alpha}(\delta(-t_1),\delta_1(-t'_1))| \le D_4e^{-{{r}\over{4}}}.
$$
\noindent
Similarly
$$
|\dis_{\alpha}(\delta(t_1),\delta_1(t'_1))| \le D_4e^{-{{r}\over{4}}}.
$$
\noindent
It follows from (\ref{d-1}) and  the above two inequalities that for $r$ large we can choose 
the complex distance $\dis_{\alpha}(\delta_1(-t'_1),\delta_1(t'_1))$ such that 
$$
-{{\pi}\over{3}} < \IM  \dis_{\alpha}(\delta_1(-t'_1),\delta_1(t'_1)) < {{\pi}\over{3}}.
$$
\noindent
In particular, we can choose the complex distances $\dis_{\alpha}(\delta(-t_1),\delta(t_1))$ and 
$\dis_{\alpha}(\delta_1(-t'_1),\delta_1(t'_1))$ such that the corresponding imaginary parts belong to the interval $(-\pi,\pi)$, and such that
$$
\big| \dis_{\alpha}(\delta(-t_1),\delta(t_1)) - \dis_{\alpha}(\delta_1(-t'_1),\delta_1(t'_1)) \big| \le 2D_4e^{-{{r}\over{4}}}.
$$
\noindent
The proof now follows from Proposition \ref{elem-1} and the triangle inequality.

\end{proof}

\begin{lemma}\label{lemma-ph-2} Let $A \in \PSLC$ be a loxodromic transformation  with the axis $\gamma$.  Let $p,q \in (\partial{\Ho} \setminus \text{endpoints}(A))$,  and denote by $\alpha_1$ the oriented geodesic from $p$ to $q$, and by $\alpha_2$ the oriented geodesic from $q$ to $A(p)$. We let $\delta_j$ be the common orthogonal between $\gamma$ and $\alpha_j$, oriented from $\gamma$ to $\alpha_j$. Then 
$$
\dis_{\gamma}(\delta_1,\delta_2)=(-1)^{j} {{\len(A)}\over{2}}+k\pi i, 
$$
\noindent
for some  $k \in \{0,1\}$ and some $j \in \{1,2\}$. 
\end{lemma}

Alternatively, we can think of $p$ and $q$ as points on the ideal boundary of $\Ho /A$, and $\alpha_1$ and $\alpha_2$ as two geodesics from $p$ to $q$, such that $\alpha_1 \cdot
(\alpha_2)^{-1}$ is freely homotopic to the core curve of the solid torus $\Ho / A$.

\begin{proof} Let $\alpha_3$ be the oriented geodesic from $A(p)$ to $A(q)$, and let $\delta_3$ be the common orthogonal between $\gamma$ and $\alpha_3$ (oriented from $\gamma$ to $\alpha_3$). Consider the right-angled hexagon $H_1$ with the sides $L_0=\gamma$, $L_1=\delta_1$, $L_2=\alpha_1$, $L_3=q$, $L_4=\alpha_2$ and $L_5=\delta_2$. Let  $H_2$ be the right-angled hexagon with the sides 
$L'_0=\gamma$, $L'_1=\delta_3$, $L'_2=\alpha_3$, $L'_3=A(p)$, $L'_4=\alpha_2$ and $L'_5=\delta_2$.
Note that  $H_1$ is a degenerate hexagon since the common orthogonal between $\alpha_1$ and $\alpha_2$ has shrunk to a point on $\partial{\Ho}$. The same holds for $H_2$. We note that the $\cosh$ formula is valid in degenerate right-angled hexagons and every such hexagon is uniquely determined by the complex lengths of its three alternating sides.
\vskip .1cm
Denote by $\sigma_k$ and $\sigma'_k$ the complex lengths of the sides $L_k$ and $L'_k$ respectively. By changing the orientations of the sides
$L_{k}$ and $L'_k$ if necessary we can arrange that $\sigma_1=\sigma'_1$, $\sigma_5=\sigma'_5$ and $\sigma_3=\sigma'_3=0$ (see Section 2.2 in \cite{series}). This shows that the hexagons $H_1$ and $H_2$ are isometric modulo the orientations of the sides, and this implies the equality $\sigma_0=\sigma'_0$. On the other hand, changing orientations of the sides can change the complex length of a side by changing its sign and/or adding $\pi i$. This proves the lemma.

\end{proof}

\subsection{The two-frame bundle and  the well connected frames}  Let $\FR(\Ho)$ denote the two frame bundle over $\Ho$. Elements of $\FR(\Ho)$ are frames $F=(p,u,n)$, where $p \in \Ho$ and $u,n \in \TB_p(\Ho)$ are two orthogonal vectors at $p$  (here $\TB(\Ho)$ denotes the unit tangent bundle). The group $\PSLC$ acts naturally on $\FR(\Ho)$. 
For $(p_i,u_i,n_i)$, $i=1,2$, we define the distance function $\dista$ on $\FR(\Ho)$ by
$$
\dista((p_1,u_1,n_1),(p_2,u_2,n_2))=d(p_1,p_2)+\Ang(u'_1,u_2)+\Ang(n'_1,n_2),
$$
\noindent
where $u'_1, n'_1 \in \TB_{p_{2}}(\Ho)$, are the parallel transports of $u_1$ and $v_1$ along the geodesic that connects $p_1$ and $p_2$. One can check that $\dista$ is invariant under the action of $\PSLC$ (we do not claim that $\dista$ is a metric on $\FR(\Ho)$). 
By $\Ne_{\epsilon}(F) \subset \FR(\Ho)$ we denote the $\epsilon$ ball around a frame $F \in \FR(\Ho)$.
\vskip .1cm
Recall the standard geodesic flow  $\flow_r:\TB(\Ho) \to \TB(\Ho)$, $r \in \R$. The flow action extends naturally on $\FR(\Ho)$, that is the map $\flow_r:\FR(\Ho) \to \FR(\Ho)$, is given by $\flow_r(p,u,n)=(p_1,u_1,n_1)$, where $(p_1,u_1)=\flow_r(p,u)$ and $n_1$ is the parallel transport of the vector $n$ along the geodesic that connects $p$ and $p_1$. The flow $\flow_r$ on $\FR(\Ho)$ is called the frame flow. The space $\FR(\Ho)$ is equipped with  the Liouville measure $\Lambda$ which is invariant under the frame flow, and under the $\PSLC$ action. Locally on $\FR(\Ho)$, the measure $\Lambda$ is the product of the standard Liouville measure for the geodesic flow and the Lebesgue measure on the unit circle. 
\vskip .1cm
Recall that $\M=\Ho/ \KG $ denotes a closed hyperbolic three manifold, and $\KG$ from now on denotes an appropriate Kleinian group. We identify the frame bundle $\FR(\M)$ with the quotient $\FR(\Ho)/ \KG$. The frame flow acts on $\FR(\M)$ by the projection. 
\vskip .1cm
It is well known \cite{brin-gromov} that the frame flow is mixing on closed 3-manifolds of variable negative curvature. In the case of constant negative curvature the frame flow is known to be exponentially mixing. This was proved by Moore in \cite{moore} using representation theory (see also \cite{pollicott}). The proof of the following theorem follows from the spectral gap theorem for  the Laplacian on closed hyperbolic manifold $\M$ and  Proposition 3.6 in \cite{moore} (we thank Livio Flaminio and Mark Pollicott for explaining this to us). 

\begin{theorem}\label{exp-mixing} There exists a ${\bf q}>0$ that depends only on $\M$ such that the following holds. 
Let $\psi,\phi:\FR(\M) \to \R$ be two $C^{1}$ functions. Then for every $r \in \R$ the inequality
$$
\left| \Lambda(\FR(\M)) \int\limits_{\FR(\M)} (\flow^{*}_r \psi)(x) \phi(x) \, d\Lambda(x)- \int\limits_{\FR(\M)} \psi(x)\, d\Lambda(x)
\int\limits_{\FR(\M)} \phi(x) \, d\Lambda(x)\right| \le Ce^{- {\bf q} |r| },
$$
\noindent
holds, where $C>0$ only depends on the $C^{1}$ norm of $\psi$ and $\phi$.
\end{theorem}

\begin{remark} In fact, one can replace the $C^{1}$ norm in the above theorem by the (weaker)  H\"older norm (see \cite{moore}).
\end{remark}

For two functions $\psi,\phi:\FR(\M) \to \R$ we set
$$
(\psi,\phi)=\int\limits_{\FR(\M)} \psi(x) \phi(x) \, d\Lambda(x).
$$
\vskip .1cm
From now on $r>>0$ denotes a large positive number that stands for the flow time of the frame flow. Also let $\epsilon>0$ denote a positive number that is smaller than the injectivity radius of $\M$. Then the projection map $\FR(\Ho) \to \FR(\M)$ is injective on every $\epsilon$ ball 
$\Ne_{\epsilon}(F) \subset \FR(\Ho)$. 
\vskip .1cm
Fix $F_0 \in \FR(\Ho)$ and let  $\Ne_{\epsilon}(F_0) \subset \FR(\Ho)$  denote the $\epsilon$ ball around the frame $F_0$.
Choose a  $C^{1}$-function $f_{\epsilon}(F_0):\FR(\Ho) \to \R$, that is positive on $\Ne_{\epsilon}(F_0)$, supported on  $\Ne_{\epsilon}(F_0)$, and such that
\begin{equation}\label{nor}
\int\limits_{\FR(\Ho)}f_{\epsilon}(F_0)(X) \, d\Lambda(X)=1.
\end{equation}
\noindent
For every $F \in \FR(\Ho)$ we define $f_{\epsilon}(F)$ by pulling back $f_{\epsilon}(F_0)$ by the corresponding element of $\PSLC$. For $F \in \FR(\M)$ the function  $f_{\epsilon}(F):\FR(\M) \to \R$ is defined accordingly (it is well defined since every ball $\Ne_{\epsilon}(F) \subset \FR(\Ho)$ embeds in $\FR(\M)$). Moreover the equality (\ref{nor}) holds for every $f_{\epsilon}(F)$.
\vskip .1cm

The following definition tells us when two frames in $\FR(\Ho)$ are well connected.

\begin{definition}\label{def-well-0} Let $F_j=(p_j,u_j,n_j) \in \FR(\Ho)$, $j=1,2$, be two frames, and set
$\flow_{ {{r}\over{4}} }  (p_j,u_j,n_j)=(\wh{p}_j,\wh{u}_j,\wh{n}_j)$.
Define
$$
\abs_{\Ho}(F_1,F_2)= 
\big( \flow^{*}_{ {{r}\over{2}}}f_{\epsilon}( \wh{p}_1,\wh{u}_1,\wh{n}_1 ), f_{\epsilon}(\wh{p}_2,-\wh{u}_2,\wh{n}_2) \big) .
$$
\noindent
We say that the frames $F_1$ and $F_2$ are $(\epsilon,r)$ well connected (or just well connected if $\epsilon$ and $r$ are understood) if $\abs_{\Ho}(F_1,F_2) >0$.
\end{definition}

The preliminary flow by time ${{r}\over{4}}$ to get $(\wh{p}_j,\wh{u}_j,\wh{n}_j)$  is used to get the estimates needed for 
Proposition \ref{prop-tripo-1} and  Proposition \ref{prop-dis}.

\begin{definition}\label{def-well} Let $F_j=(p_j,u_j,n_j) \in \FR(\M)$, $j=1,2$, be two frames and let $\gamma$ be a geodesic segment in $\M$ that connects $p_1$ and $p_2$.  Let $\wt{p}_1 \in \Ho$ be a lift of $p_1$, and let $\wt{p}_2$ denotes the lift of $p_2$ along $\gamma$. By  $\wt{F}_j=(\wt{p}_j,\wt{u}_j,\wt{n}_j) \in \FR(\Ho)$ we denote the corresponding lifts. Set $\abs_{\gamma}(F_1,F_2)=\abs_{\Ho}(\wt{F}_1,\wt{F}_2)$.
We say that the frames $F_1$ and $F_2$ are $(\epsilon,r)$ well connected (or just well connected if $\epsilon$ and $r$ are understood) along the segment $\gamma$, if $\abs_{\gamma}(F_1,F_2) >0$.
\end{definition}

The function  $\abs_{\gamma}(F_1,F_2)$ is the affinity function from the outline above. 
\vskip .1cm
Let $F_j=(p_j,u_j,n_j) \in \FR(\M)$, $j=1,2$ and let 
$\flow_{ {{r}\over{4}} }  (p_j,u_j,n_j)=(p'_j,u'_j,n'_j)$. Define
$$
\abs(F_1,F_2)=
\big( \flow^{*}_{ -{{r}\over{2}}} f_{\epsilon}( p'_1,u'_1,n'_1), 
f_{\epsilon}(p'_2,-u'_2,n'_2 ) \big).
$$
\noindent
Then
$$
\abs(F_1,F_2)=\sum_{\gamma} \abs_{\gamma}(F_1,F_2),
$$
\noindent
where $\gamma$ varies over all geodesic segments in $\M$ that connect $p_1$ and $p_2$ (only finitely many numbers $\abs_{\gamma}(F_1,F_2)$ are non-zero).
One can think of $\abs(F_1,F_2)$ as the total probability that the frames $F_1$ and $F_2$ are well connected, and $\abs_{\gamma}(F_1,F_2)$ represents the probability that they are well connected along the segment $\gamma$. The following lemma follows from Theorem \ref{exp-mixing}.

\begin{lemma}\label{mixing} Fix $\epsilon>0$.  Then for $r$ large and any $F_1,F_2 \in \FR(\M)$ we have
$$
\abs(F_1,F_2)= {{1}\over{\Lambda(\FR(\M))}} (1+O(e^{-{\bf q} {{r}\over{2}} })),
$$
\noindent
where ${\bf q}>0$ is a constant that depends only on the manifold $\M$.
\end{lemma}

\subsection{The geometry of well connected bipods}   There is natural order three homeomorphism $\omega:\FR(\Ho) \to \FR(\Ho)$ given by $\omega(p,u,n)=(p,\omega(u),n)$, where $\omega(u)$ is the vector in $\TB_p(\Ho)$ that is orthogonal to $n$ and  such that the oriented angle (measured anticlockwise) between $u$ and $\omega(u)$ is ${{2\pi}\over{3}}$ (the plane containing the vectors $u$ and $\omega(u)$ is oriented by the normal vector $n$). An equivalent way of defining $\omega$ is by the Right hand rule. The homeomorphism $\omega$ commutes with the $\PSLC$ action and it is well defined on $\FR(\M)$ by the projection. The distance function $\dista$ on $\FR(\M)$ is  invariant under $\omega$. 

To every  $F_p=(p,u,n) \in \FR(\M)$  we associate the bipod $B_p=(F_p,\omega(F_p)$ and the anti-bipod  $\overline{B}_p=(F_p,\overline{\omega}(F_p)$ 
(we recall that $\overline{\omega}=\omega^{-1}$).  We have the following definition.

\begin{definition}\label{def-bipo} Given two frames $F_p=(p,u,n) \in \FR(\M)$, and $F_q=(q,v,m) \in \FR(\M)$, let $B_p$ and $B_q$ denote  the corresponding bipods. 
Let $\gamma=(\gamma_0,\gamma_1)$,  be a pair of  geodesic segments in $\M$, each connecting the points $p$ and $q$. 
We say that the bipods $B_p$ and $B_q$ are   $(\epsilon,r)$ well connected along the pair of segments $\gamma$,  if the pairs of frames $F_p$ and $F_q$, and  $\omega(F_p)$ and $\overline{\omega}(F_q)$, are  $(\epsilon,r)$ well connected along the segments $\gamma_0$ and $\gamma_1$ respectively.
\end{definition}

\begin{lemma}\label{lemma-bipo} Let  $F_p=(p,u,n)$ and $F_q=(q,v,m)$ be two frames in $\M$. Suppose that  the corresponding bipods  $B_p$ and  $B_q$ are $(\epsilon,r)$-well connected along a pair of geodesic segments 
$\gamma_0$ and $\gamma_1$ that connect $p$ and $q$ in $\M$,  that is we assume $\abs_{\gamma_{0}}(F_p,F_q)>0$ and likewise  $\abs_{\gamma_{1}}(\omega(F_p),\overline{\omega}(F_q) )>0$. Then for $r$ large, the closed curve $\gamma_0 \cup \gamma_1$ is homotopic to a closed geodesic $\delta \in \Gamma$, and the following inequality holds 
$$
\big| \len(\delta)- 2r + 2\log {{4}\over{3}} \big| \le D\epsilon,
$$
\noindent
for some constant $D>0$.  Moreover,
$$
d(p,\delta), \, d(q,\delta) \le \log \sqrt{3} +D\epsilon.
$$
\end{lemma}
\begin{figure}
	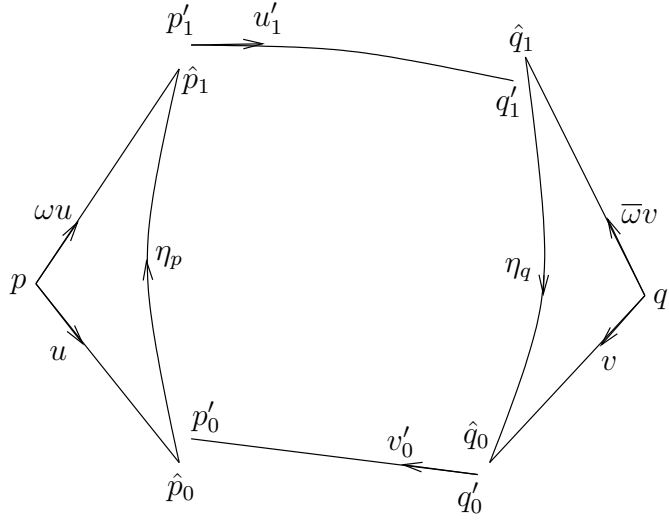
	\caption{The closed $\epsilon$-chain for two well-connected bipods}
	\label{fig:bipods}
\end{figure}
\begin{proof}  We define, for $i=0,1$, $F_{\wh{p}_{i}} =(\wh{p}_i,\wh{u}_i,\wh{n}_i)$ by $\flow_{\frac{r}{4}}(\omega^{i}(F_p))$. Likewise, we let 
$F_{\wh{q}_{i}} =(\wh{q}_i,\wh{v}_i,\wh{m}_i)$ by $\flow_{\frac{r}{4}}(\overline{\omega}^{i}(F_q))$.  Because $\omega^{i}(F_p)$ and $\overline{\omega}^{i}(F_q)$ are well connected,
we can find $F_{p'_{i}} \in \Ne_{\epsilon}(F_{\wh{p}_{i}})$, and  $F_{q'_{i}} \in \Ne_{\epsilon}(F_{\wh{q}_{i}})$ such that $\flow_{\frac{r}{2}}(p'_i,u'_i,n'_i)=(q'_i,-v'_i,m'_i)$.
Moreover, there is a homotopy condition that is satisfies, namely that the concatenation of the $\epsilon$-chain 
$$
\left(  \flow_{[0,\frac{r}{4}]}(p_i,u_i), \flow_{[0,\frac{r}{2}]}(p'_i,u'_i),\flow_{[0,\frac{r}{4}]}(\wh{q}_i,-\wh{v}_i \right),  
$$
is homotopic rel endpoints to $\gamma_i$. 

We let $\eta_p$ be the geodesic segment from $\wh{p}_0$ to $\wh{p}_1$ that is homotopic  rel endpoints to 

$$
\flow_{[0,\frac{r}{4}]}(\wh{p}_0,-\wh{u}_0)  \cdot   \flow_{[0,\frac{r}{4}]}(p,u).
$$
Then $\wh{n}_0$ and $\wh{n}_1$ are parallel along $\eta_p$ because the are orthogonal to the plane of the immersed triangle we have formed), and the angle between $\eta_p$ and $-\wh{u}_i$ (at $\wh{p}_i$) is less than
$De^{-\frac{r}{4}}$. Moreover, 

$$
\left| \len(\eta_p)-\frac{r}{2}+\log \frac{4}{3} \right| \le De^{-\frac{r}{4}}.
$$
We likewise define $\eta_q$ and make the same observation.

We refer the reader to Figure \ref{fig:bipods} for an illustration of our construction. 

The segments $\eta_p$, $\flow_{[0,\frac{r}{2}]}(p'_1,u'_1)$,  $\eta^{-1}_q$,  $\flow_{[0,\frac{r}{2}]}(q'_0,v'_0)$,   form a closed $\epsilon$-chain, and we are therefore in a position to apply Lemma \ref{lemma-chain}.
We take 
$$
(a_0,b_0,a_1,b_1,a_2,b_2,a_3,b_3)=(\wh{p}_0,\wh{p}_1, p'_1,q'_1,\wh{q}_1, \wh{q}_0,q'_0,p'_0)
$$ 
(and connect $a_i$ to $b_i$ by the aforementioned segments), and we let 
$(n_0,n_1,n_2,n_3)=(\wh{n}_0,n'_1,m'_1,m'_0)$. We can easily verify that the hypotheses of Lemma \ref{lemma-chain} are satisfied, and we conclude that $\gamma_0 \cup \gamma_1$ is freely homotopic to a closed geodesic $\delta$, and the following inequalities
$$
\big| \len(\delta)- 2r + 2\log {{4}\over{3}} \big| \le D\epsilon,
$$
and 
$$
d(\wh{p}_i,\delta), \, d(\wh{q}_i,\delta) \le D\epsilon
$$
hold. It follows that the projection of $p$ onto $\eta_p$ is exponentially close to $\delta$, and therefore
$$
d(p,\delta), \, d(q,\delta) \le \log \sqrt{3} +D\epsilon.
$$

\end{proof}

\subsection{The geometry of well connected tripods} Let $P,P_1,P_2 \in \FR(\Ho)$. We call $P$ the  reference frame and $P_1,P_2$ the moving frames. 
Let  $F_1 \in \FR(\Ho)$, and $r$ large. 
Then the frame $F_2=L(F_1,P_1,P_2,r)$ is defined as follows: 
\vskip .1cm
Let $\wt{F}_1=\flow_{ {{r}\over{4}} }(F_1)$. Let $\wh{F}_1 \in \Ne_{\epsilon}(\wt{F}_1)$ denote the frame such that for some $M_1 \in \PSLC$ we have $M_1(P)=\wt{F}_1$ and $M_1(P_1)=\wh{F}_1$. Set $\flow_{ {{r}\over{2}} }(\wh{F}_1)=(\wh{q},-\wh{v},\wh{m})$, and $\wh{F}_2=(\wh{q},\wh{v},\wh{m})$. Let $\wt{F}_2$ denote the frame such that for some $M_2 \in \PSLC$ we have $M_2(P)=\wt{F}_2$ and $M_2(P_2)=\wh{F}_2$. Set $\flow_{ -{{r}\over{4}} }(\wt{F}_2)=F_2=(q,v,m)$.  Observe that the frame $F_2$ only depends on $F_1$, $P_1$, $P_2$, and $r$.
\vskip .1cm
Recall from Section 3 that  $\Pi^{0}$ denotes an oriented  topological pair of pants equipped with a homeomorphism $\omega_{0}:\Pi^{0} \to \Pi^{0}$, of order three that permutes the cuffs. By $\omega^{i}_0(C)$, $i=0,1,2$, we denote the oriented cuffs of $\Pi^0$.
For each $i=0,1,2$, we choose  $\omega^{i}_0(c) \in \pi_1(\Pi^0)$ to be an element in the conjugacy class that corresponds to the cuff $\omega^i_0(C)$, such that $\omega^{0}_0(c)  \omega^{1}_0(c) \omega^{2}_0(c)=\id$.
\vskip .1cm
Fix a frame $P \in \FR(\Ho)$, and fix six frames $P^{j}_i \in \Ne_{\epsilon}(P)$, $i=0,1,2$, $j=1,2$,  where $\Ne_{\epsilon}(P)$ is the $\epsilon$ neighbourhood of $P$. Denote by $(P^{j}_i)$ the corresponding six-tuple of frames. We define the representation 
$$
\rho(P^{j}_i): \pi_1(\Pi^0) \to \PSLC
$$
\noindent
as follows:
\vskip .1cm
Choose a frame  $F^{0}_1=(p,u,n) \in \Ho$, and let $F_2=L(F^{0}_1,P^{1}_0,P^{2}_0,r)$. Denote by $F^{j}_1$, $j=1,2$, given by
$\omega(F^{1}_1)=L(\omega^{-1}(F_2),P^{2}_1,P^{1}_1,r)$, and $\omega^{2}(F^{2}_1)=L(\omega^{-2}(F_2),P^{2}_2,P^{1}_2,r)$. Let $A_i \in \PSLC$ given by
$A_0(F^{0}_1)=F^{1}_1$, $A_1(F^{1}_1)=F^{2}_1$, and $A_2(F^{2}_1)=F^{0}_1$. Observe $A_2 A_1 A_0=\id$. We define $\rho(P^{j}_i)=\rho$ by
$\rho(\omega^{i}(c))=A_i$. Up to conjugation in $\PSLC$, the representation  $\rho(P^{j}_i)$ depends only on the six-tuple $(P^{j}_i)$ and $r$. Observe that if $P^{j}_i=P$, for all $i,j$,  then  
$\Ho / \rho(P^{j}_i)$ is a planar pair of pants whose all three cuffs have equal length, and the half lengths of the cuffs that  correspond to  this representation are positive real numbers.

We will use the following lemma to show that the skew pants that corresponds to a pair of well connected tripods (see the definition below) is indeed in $\Pant_{D\epsilon,R}$ for  some universal constant $D>0$.

\begin{lemma}\label{lemma-tripo-0} Fix a frame $P \in \FR(\Ho)$, and fix $P^{j}_i \in \Ne_{\epsilon}(P)$, $i=0,1,2$, $j=1,2$. Set $\rho(P^{j}_i)=\rho$. Then 
$$
\big| \hl(\omega^{i}_0(C))   - r + \log {{4}\over{3}} \big| \le D\epsilon,
$$
\noindent
for some constant $D>0$, where $\hl(\omega^{i}_0(C))$ denotes the half lengths that correspond to the representation $\rho$.
In particular, the transformation  $\rho(\omega^{i}_0(c))$ is loxodromic. 
\end{lemma}

\begin{proof}  It follows from Lemma \ref{lemma-bipo} that  
$$
\big| \len(\omega^{i}_0(C))   - 2r + 2\log {{4}\over{3}} \big| \le D\epsilon,
$$
\noindent
where  $\len(\omega^{i}_0(C))$ denotes the cuff length of $\rho(\omega^{i}_0(c))=A_i$. 
\vskip .1cm
We have $\hl(\omega^{i}(C))={{\len(\omega^{i}(C) )}\over{2}}+k \pi i$, for some $k \in \{0,1\}$. It remains to show that $k=0$. 
\vskip .1cm
Let $t \in [0,1]$, and let $P^{j}_i(t)$ be a continuous path in $\Ne_{\epsilon}(P)$, such that $P^{j}_i(1)=P^{j}_i$, and $P^{j}_i(0)=P$.
Set $\rho_t=\rho(P^{j}_i(t))$. Then for each $t$ we obtain the corresponding number $k(t) \in \{0,1\}$. Since $k(0)=0$ and since $k(t)$ is continuous we have $k(1)=k=0$.
\end{proof}

\vskip .3cm 

To every frame $F \in \FR(\M)$ we associate the tripod  $T=\omega^{i}(F)$, $i=0,1,2$ and the anti-tripod  $\overline{T}=\overline{\omega}^{i}(F)$, $i=0,1,2$, where $\overline{\omega}=\omega^{-1}$.

\begin{definition}\label{def-trip} Given two frames   $F_p=(p,u,n)$ and $F_q=(q,v,m)$ in $\FR(\M)$, let $T_p=\omega^{i}(F_p)$ and $T_q=\omega^{i}(F_q)$, $i=0,1,2$, be the corresponding tripods. Let $\gamma=(\gamma_0,\gamma_1,\gamma_2)$,  be a triple of  geodesic segments in $\M$, each connecting the points $p$ and $q$. 
We say that the pair of tripods $T_p$ and $T_q$ is well connected along  $\gamma$, if each pair of frames  $\omega^{i}(F_p)$ and  $\omega^{-i}(F_q)$ is well connected along the segment $\gamma_i$. 
\end{definition}

\vskip .1cm
Next we show that to every pair of well connected tripods we can naturally associate a skew pants in the sense of Definition \ref{def-skew}.
Recall from Section 3 that  $\Pi^{0}$ denotes an oriented  topological pair of pants equipped with a homeomorphism $\omega_{0}:\Pi^{0} \to \Pi^{0}$, of order three that permutes the cuffs. By $\omega^{i}_0(C)$, $i=0,1,2$, we denote the oriented cuffs of $\Pi^0$.
For each $i=0,1,2$, we choose  $\omega^{i}_0(c) \in \pi_1(\Pi^0)$ to be an element in the conjugacy class that corresponds to the cuff $\omega^i_0(C)$, such that $\omega^{0}_0(c) \omega^{1}_0(c) \omega^{2}_0(c)=\id$.
\vskip .1cm
Let $a,b \in \Pi^{0}$ be the fixed points of the homeomorphism $\omega_0$. Let $\alpha_0 \subset \Pi^{0}$ be a simple arc  that connects $a$ and $b$, and set $\omega^{i}_0(\alpha_0)=\alpha_i$.  The union of two different arcs $\alpha_i$ and $\alpha_j$ is a closed curve in $\Pi^0$ homotopic to a cuff. One can think of the union of these three segments as the spine of $\Pi^0$. Moreover, there is an obvious projection from $\Pi^0$ to the spine $\alpha_0\cup \alpha_1 \cup \alpha_2$, and this projection is a homotopy equivalence.
\vskip .1cm
Let  $T_p=(p,\omega^{i}(u),n)$ and $T_q=(q,\omega^{i}(v),m)$, $i=0,1,2$, be two tripods in $\FR(\M)$ and $\gamma=(\gamma_0,\gamma_1,\gamma_2)$ a triple of geodesic segments in $\M$ each connecting the points $p$ and $q$. One  constructs a map $\phi$ from the spine of $\Pi^0$ to $\M$ by letting $\phi(a)=p$, $\phi(b)=q$, and by letting $\phi:\alpha_i \to \gamma_i$ be any homeomorphism. By precomposing this map with the projection from $\Pi^0$ to its spine we get a well defined map  $\phi:\Pi^0 \to \M$. By $\rho(T_p,T_q,\gamma):\pi_1(\Pi^0) \to \KG$ we denote the induced  representation of the fundamental group of $\Pi^0$. 
\vskip .1cm
In principle, the representation $\rho(T_p,T_q,\gamma)$ can be trivial. However if the the tripods $T_p$ and $T_q$ are well connected along $\gamma$, we prove below that the representation $\rho(T_p,T_q,\gamma)$ is admissible (in sense of Definition \ref{def-adm}) and that the conjugacy class  
$[\rho(T_p,T_q,\gamma)]$ is a skew pants in terms of Definition \ref{def-skew}.

\begin{lemma}\label{lemma-tripo} Let  $T_p$ and $T_q$ be two tripods that are well connected along a triple of segments $\gamma$ and set 
$\rho= \rho(T_p,T_q,\gamma)$.  Then
$$
\big| \hl(\omega^{i}_0(C))   - r + \log {{4}\over{3}} \big| \le D\epsilon,
$$
\noindent
for some constant $D>0$. In particular, the conjugacy class of  transformations  $\rho(\omega^{i}_0(C))$ is loxodromic. 
\end{lemma}

\begin{proof} 

Observe that there exist $P^{j}_i \in \Ne_{\epsilon}(P)$, such that $\rho(P^{j}_i)=\rho(T_p,T_q,\gamma)$. The lemma follows from 
Lemma \ref{lemma-tripo-0}.

\end{proof}

Recall that $\Pant_{\epsilon,R}$ is the set of skew pants whose half-lengths are $\epsilon$ close to $\frac{R}{2}$, and that $R=2(r-\log \frac{4}{3} )$. 
If we  write $\pi(T_p,T_q,\gamma)=[\rho(T_p,T_q,\gamma)]$, then by Lemma \ref{lemma-tripo},  $\pi$ maps well connected pairs of tripods to pairs of skew pants in $\Pant_{D\epsilon,R}$. 

\vskip .1cm
We have
\begin{definition} Let $T_p$ and $T_q$ be two tripods that are well connected along a triple of segments $\gamma=(\gamma_0,\gamma_1,\gamma_2)$.
Set
$$ 
\taso_{\gamma}(T_p,T_q)=\prod\limits_{i=0}^{i=2} \abs_{\gamma_{i}}(\omega^{i}(F_p) ) , (\omega^{-i}(F_q) ) .
$$
\end{definition}

Observe that two tripods  $T_p$ and $T_q$  are   $(\epsilon,r)$ well connected along a triple of geodesic segments $\gamma$, if and only if  $\taso_{\gamma}(T_p,T_q)>0$. 
\vskip .1cm

We define the space of well connected tripods as the space of all triples $(T_p,T_q,\gamma)$, such that the tripods $T_p$ and $T_q$ are well connected along $\gamma$. It follows from the exponential  mixing statement that given any two tripods $T_p$ and $T_q$, and for $r$ large enough, there will exist at least one triple of segments $\gamma$ so that  $T_p$ and $T_q$ are well connected along $\gamma$ (in fact, it can be shown that there will be many such segments). 

We define the measure $\wt{\mu}$ on the set of well connected tripods by 
\begin{equation}\label{wt-mu}
d\wt{\mu}(T_p,T_q,\gamma)=\taso_{\gamma}(T_p,T_q)\, d\lambda_{T}(T_p,T_q,\gamma),
\end{equation}
\noindent
where $\lambda_{T}(T_p,T_q,\gamma)$ is the product of the Liouville measure $\Lambda$ (for $\FR(\M)$) on the first two terms, and the counting measure on the third term. The measure $\lambda_{T}$ is infinite (since there are infinitely many geodesic segments between any two points $p,q \in \M$) but $\taso_{\gamma}(T_p,T_q)$ has compact support (that is, only finitely many such triples of connections $\gamma$ are ``good"), so $\wt{\mu}$ is finite.

Recall that $R=2(r-\log \frac{4}{3} )$ (see the discussion after Lemma \ref{lemma-tripo} above).  
We define the measure $\mu$ on $\Pant_{D\epsilon,R}$ by $\mu=\pi_{*} \wt{\mu}$. This is the measure from Theorem \ref{theorem-mes}. It follows from the construction that this measure is invariant under the involution $\refl:\Pant \to \Pant$ (see Section 3 for the definition), that is  $\mu \in \Mes^{\refl}_0(\Pant)$.

\vskip .1cm

In order to prove Theorem \ref{theorem-mes} it remains to construct the corresponding measure $\beta \in \Mes_0(\NB(\sqrt{\Gamma}))$ and prove the stated 
properties.

\subsection{The  ``predicted foot" map $\ft_{\delta}$ }\label{section-4.7}   By $F_p=(p,u,n)$ and $F_q=(q,v,m)$ we continue to denote two frames in $\FR(\M)$. Suppose 
the frames $\omega^{i}(F_p)$ and  $\overline{\omega}^{i}(F_q)$ are well connected along the geodesic segments $\gamma_i$, $i=0,1$. In our terminology this means that the bipods $B_p$ and $B_q$ are well connected along the segments $\gamma_0$ and $\gamma_1$. Let  $\delta_{2} \in \Gamma$ denote the closed geodesic in $\M$ freely homotopic to  $\gamma_{0} \cup \gamma _{1}$.
We now associate the ``geometric feet" to $(B_p,B_q,\gamma_0,\gamma_1)$.

We first define the geodesic ray $\alpha_p:[0,\infty) \to \M$ by $\alpha_p(0)=p$, $\alpha'_p(0)=\overline{\omega}(u)$, and we likewise define the geodesic ray $\alpha_q:[0,\infty) \to \M$ by $\alpha_q(0)=q$, $\alpha'_q(0)=\omega(v)$. Then for $t \in [0,\infty)$, and $i=0,1$, we let $\beta^{t}_i$ be the geodesic segment homotopic relative endpoints to the piecewise geodesic arc $(\alpha_p[0,t])^{-1} \cdot \gamma_{i} \cdot \alpha_q[0,t])$. (The endpoints of both segments $\beta^{t}_{1}$ and $\beta^{t}_{2}$ are $\alpha_p(t)$ and $\alpha_q(t)$, and $\beta^{0}_{i}=\gamma_i$). We let $\beta^{\infty}_i$ be the limiting geodesic   of $\beta^{t}_i$, when $t \to \infty$. For each $t>0$ and $i=0,1$, there is an obvious choice of common orthogonal from $\delta_{2}$ to $\beta^{t}_i$, which varies continuously with $t \in [0,\infty]$. We let $f^{t}_{i} \in \NB(\delta_{2})$ be the foot of this common orthogonal at $\delta_{2}$, and we let $f_i=f^{\infty}_i$.

For a closed geodesic $\delta \in \Gamma$,  let $\TT_{\delta}$ denote the solid torus whose core curve is $\delta$. 
As an alternative point of view, we can lift $\gamma_0 \cup \gamma_1$ to a closed curve in the solid torus $\TT_{\delta_{2}}$ (there is a unique such lift to a closed curve in $\TT_{\delta_{2}}$). 
We can then lift $F_p$ and $F_q$, and also $\alpha_p[0,\infty]$ and $\alpha_q[0,\infty]$, where $\alpha_p(\infty), \alpha_q(\infty) \in \partial{\TT_{\delta_{2}}}$. Then we define $\beta^{t}_i$ (and 
$\beta^{\infty}_i$) as before and there will be unique common orthogonals from (the lift of) $\delta_{2}$ to $\beta^{t}_{i}$, $t \in [0,\infty]$.

By Lemma \ref{lemma-ph-2} we see that $\dis_{\delta_{2}}(f_0,f_1)=\hl(\delta_{2})$, so $f_0$ and $f_1$ represent the same point in $\NB(\sqrt{\delta_{2}})$. Therefore we have defined the  mapping
$$
(B_p,B_q,\gamma_0,\gamma_1) \mapsto \ft_{\delta_{2}}(B_p,B_q,\gamma_0,\gamma_1) \in \NB(\sqrt{\delta_{2}}),
$$ 
on the set of all well connected bipods such that the $\gamma_0 \cup \gamma_1$ is homotopic to $\delta_{2}$. We think of the vector 
$\ft_{\delta_{2}}(B_p,B_q,\gamma_0,\gamma_1) \in \NB(\sqrt{\delta_{2}})$ as the geometric foot of $(B_p,B_q,\gamma_0,\gamma_1)$.

Assume now that we are given a third geodesic segment $\gamma_2$ between $p$ and $q$ (also known as the third connection) such that $(T_p,T_q,\gamma)$ is a pair of well connected tripods along the triple of segments $\gamma=(\gamma_0,\gamma_1,\gamma_2)$.  Above, we have defined the skew pants $\Pi=\pi(T_p,T_q,\gamma)$, such that $\partial{\Pi}=\delta_0+\delta_1 +\delta_2$, where $\delta_i$ is homotopic to $\gamma_{i-1} \cup \gamma_{i+1}$ (using the convention $\gamma_i=\gamma_{i+3}$).

Let $h_i \in \NB(\delta_2)$, $i=0,1$, denote the foot of the common orthogonal from $\delta_2$ to $\delta_i$. Recall that since $\dis_{\delta_{2}}(h_0,h_1)=\hl(\delta_2)$ the projections of  $h_0$ and $h_1$ to $\NB(\sqrt{\delta_{2}})$ agree, and as before we let $\foot_{\delta_{2}}(\Pi) \in \NB(\sqrt{\delta_{2}})$ denote this projection. We say that $\foot_{\delta_{2}}(\Pi)$ is the foot of the skew pants $\Pi$ on the cuff $\delta_2$.

We will now verify that  on $\NB(\sqrt{\delta_{2}})$ we have $\dis_{\delta_{2}}(f_0,h_1)=\dis_{\delta_{2}}(f_1,h_0)=O(e^{-\frac{r}{4}})$. 
This will imply that the pairs $\{h_0,h_1\}$ and $\{f_0,f_1 \}$ project to  vectors in $\NB(\sqrt{\delta_{2}})$ that are $e^{-\frac{r}{4}}$ close.

\begin{proposition}\label{prop-tripo-1}  With the above notation we have that for $r$ large and $\epsilon$ small the inequalities 
$$
\dist(f_0,h_1), \dist(f_1,h_0) \le De^{-{{r}\over{4}} }
$$
holds for some universal constant $D>0$.

\end{proposition}

\begin{proof} Assume that we are given a skew pants $\Pi=\pi(T_p,T_q,\gamma)$, where $\gamma=(\gamma_0,\gamma_1,\gamma_2)$ is a triple of good connections. Recall that $\delta_i$ is a cuff of $\Pi$ that is homotopic to $\gamma_{i-1} \cup \gamma_{i+1}$.  Then for $i=0,1$, the geodesics $\delta_2$ and $\delta_i$ (or more precisely the appropriate lifts of $\delta_2$ and $\delta_i$ to the solid torus cover corresponding to $\delta_{2}$) satisfy (\ref{assumption}).

On the other hand, since $\gamma_2$ is a good connection, and from the definition of a good connection between two frames, it follows that for some universal constant $E>0$ the segment $\beta^{\frac{r}{4}}_{0}$ (considered in the solid torus cover $\TT_{\delta_{2}}$) has the endpoints $E$ close to $\delta_1$. Similarly the segment $\beta^{\frac{r}{4}}_{1}$ has the endpoints $E$ close to $\delta_0$.  
The inequality $\dist(f_0,h_1), \dist(f_1,h_0) \le De^{-{{r}\over{4}} }$ now follows from Lemma \ref{lemma-ph-1}.

\end{proof}

\vskip .1cm

For each skew pants $\Pi=\pi(T_p,T_q,\gamma)$ we let

$$
\ft_{\delta_{2}}(\Pi)=\ft_{\delta_{2}}(T_p,T_q,\gamma)= \ft_{\delta_{2}}(B_p,B_q,\gamma_0,\gamma_1).
$$
That is, we have defined the map $(\Pi,\delta^{*}) \mapsto \ft_{\delta}(\Pi,\delta) \in \NB(\sqrt{\delta})$ on the set of all marked skew pants $\Pant^{*}_{D\epsilon,R}$ that contain the geodesic $\delta$ in its boundary. Recall that we have already defined the mapping $(\Pi,\delta^{*}) \mapsto \foot_{\delta}(\Pi,\delta) \in \NB(\sqrt{\delta})$. We have

\begin{proposition}\label{prop-dis} Let $(\pi(T_p,T_q,\gamma),\delta^{*}) \in \Pant^{*}$.  Then for $r$ large and $\epsilon$ small  
we have
$$
\dis(\foot_{\delta}(\pi(T_p,T_q,\gamma),\ft_{\delta}(T_p,T_q,\gamma)) \le De^{-{{r}\over{4}} },
$$
\noindent
for some constant $D>0$.
\end{proposition}

\begin{proof} It follows from Proposition \ref{prop-tripo-1}.

\end{proof}

Given skew pants  $\Pi=\pi(T_p,T_q,\gamma)$, the new foot $\ft_{\delta_{2}}(T_p,T_q,\gamma)$ ``predicts" the location of the old foot $\foot_{\delta_{2}}(T_p,T_q,\gamma)$ (up to an exponentially small error in $r$) without knowing the third connection $\gamma_2$.

\subsection{The proof of Theorem \ref{theorem-mes} } Fix $\delta \in \Gamma$. For a given measure $\alpha$ on $\NB(\sqrt{\Gamma})$ we let  $\alpha_{\delta}$ denote the restriction of  $\alpha$ on $\NB(\sqrt{\delta})$.  It remains to construct the measure $\beta$ on $\NB(\sqrt{\Gamma})$ from Theorem \ref{theorem-mes} and  estimate the  Radon-Nikodym derivative of $\beta_{\delta}$  
with respect to the Euclidean measure on $\NB(\sqrt{\delta})$.

Recall that $(B_p,B_q,\gamma_0,\gamma_1)$ is a well connected pair of bipods along the pair of segments $\gamma_0$ and $\gamma_1$ if
$$
\abs_{\gamma_{0}}(F_1,F_2) \abs_{\gamma_{1}}(\omega(F_1),\overline{\omega}(F_2))>0.
$$ 
We define the set $S_{\delta}$ by saying that  $(F_p,F_q,\gamma_0,\gamma_1) \in S_{\delta}$ if  $(B_p,B_q,\gamma_0,\gamma_1)$ is a well connected pair of bipods along a pair of segments $\gamma_0$ and $\gamma_1$ such that  $\gamma_0 \cup \gamma_1$ is homotopic to $\delta$. In the previous subsection we have defined the map
$$
\ft_{\delta}:S_{\delta} \to \NB(\sqrt{\delta}).
$$

Recall that that the bundle $\NB(\sqrt{\delta})$ has the natural  $\C/(2\pi i \Z+l(\delta)\Z)$ action by isometries.  Now, we define the action of the torus $\C/(2\pi i \Z+l(\delta)\Z)$  on $S_{\delta}$ so that the map $\ft_{\delta}$ becomes equivariant with respect to the torus actions on $S_{\delta}$ and $\NB(\sqrt{\delta})$, that is for each $\tau \in \C/(2\pi i \Z+l(\delta)\Z)$ we have
\begin{equation}\label{vazno}
\ft_{\delta}(\tau+(B_p,B_q,\gamma_0,\gamma_1))=\tau+\ft_{\delta}(B_p,B_q,\gamma_0,\gamma_1),
\end{equation}
where $\tau+(B_p,B_q,\gamma_0,\gamma_1)$ denotes the new element of $S_{\delta}$ (obtained after applying the action by $\tau$ to $(B_p,B_q,\gamma_0,\gamma_1)$).

Let $\TT_{\delta}$ be the open solid torus cover associated to $\delta$ (so $\delta$ has a unique lift to a closed geodesic in $\TT_{\delta}$ which we denote by $\wh{\delta}(\delta)$). Given a pair of well connected bipods in $S_{\delta}$, each bipod lifts  in a unique way to a bipod in  $\FR(\TT_{\delta})$ such that the pair of the lifted bipods is well connected in $\TT_{\delta}$. We denote by $\wt{S}_{\delta}$ the set of such lifts, so  $\wt{S}_{\delta}$ is in one-to-one correspondence with $S_{\delta}$.

We observe that the group of automorphisms of the solid torus $\TT_{\delta}$ is isomorphic to the group of isomorphisms of the unit normal bundle $\NB(\delta)$, that is in turn isomorphic to 
$\C/(2\pi i \Z+l(\delta)\Z)$ which acts  on both $\NB(\delta)$ and on $\FR^{2}(\TT_{\delta})$ so as to map  $\wt{S}_{\delta}$ to itself. 
Since $\wt{S}_{\delta}$ and $S_{\delta}$ are in one-to-one correspondence we have the induced action of $\C/(2\pi i \Z+l(\delta)\Z)$  on $S_{\delta}$. The equivariance (\ref{vazno}) follows from the construction.

Let $C_{\delta}$ be the space of well connected tripods $(T_p,T_q,\gamma)$, where $\gamma=(\gamma_0,\gamma_1,\gamma_2)$, such that $\gamma_0 \cup \gamma_1$ is homotopic to $\delta$. 
Let $\chi:C_{\delta} \to S_{\delta}$ be the forgetting map (the term forgetting map refers to forgetting the third connection $\gamma_2$), so $\chi(T_p,T_q,\gamma_0,\gamma_1,\gamma_2)=(B_p,B_q,\gamma_0,\gamma_1)$.  

It follows from Proposition \ref{prop-dis}  that for any pair of well connected tripods   $T=(T_p,T_q,\gamma) \in C_{\delta}$ we have
\begin{equation}\label{f-foot-new}
|\ft_{\delta}(\chi (T))-\foot_{\delta}(\pi(T,\gamma))|<Ce^{-{{r}\over{4}}},
\end{equation}
where $\pi(T,\gamma)$ is the corresponding skew pants. 

\vskip .3cm

Next, we define the measure $\nu_{\delta}$ on $S_{\delta}$ by
$$
d\nu_{\delta}(B_p,B_q,\gamma_0,\gamma_1)=\abs_{\gamma_{0}}(F_p,F_q)\abs_{\gamma_{1}}(\omega(F_p),\overline{\omega}(F_q))\, d\lambda_{B}(B_p,B_q,\gamma_0,\gamma_1),
$$
where $\lambda_B$ is the measure on $S_{\delta}$ defined as the product of the Liouville measures on the first two terms and the counting measure on the other two terms. 

We make two observations. The first one is that  $\lambda_{B}$ is invariant under the  $\C/(2\pi i \Z+l(\delta)\Z)$ action on $S_{\delta}$. The second one is as follows. Let $\tau \in \C/(2\pi i \Z+l(\delta)\Z)$, and for $(B_p,B_q,\gamma_0,\gamma_1) \in S_{\delta}$ we let  
$$
(B_{p(\tau)},B_{q(\tau)},\gamma_0(\tau),\gamma_1(\tau))=\tau+(B_p,B_q,\gamma_0,\gamma_1).
$$ 
denote the corresponding element of $S_{\delta}$. It follows from the definition of the affinity functions that 
$$
\abs_{\gamma_{0}}(F_p,F_q)\abs_{\gamma_{1}}(\omega(F_p),\overline{\omega}(F_q))=\abs_{\gamma_{0}(\tau)}(F_{p(\tau)},F_{q(\tau)})\abs_{\gamma_{1}(\tau)}(\omega(F_{p(\tau)}),\overline{\omega}(F_{q(\tau)})),
$$
for any $\tau$. These two observations show that the measure $\nu_{\delta}$ is invariant under the $\C/(2\pi i \Z+l(\delta)\Z)$ action on $S_{\delta}$.

Since the map $\ft_{\delta}$  is invariant under the  $\C/(2\pi i \Z+l(\delta)\Z)$ actions  (see $(\ref{vazno})$), it follows from the above two observations that the measure $(\ft_{\delta})_{*} \nu_{\delta}$ is  invariant under the $\C/(2\pi i \Z+l(\delta)\Z)$ action on $\NB(\sqrt{\delta})$. Therefore, the measure $(\ft_{\delta})_{*} \nu_{\delta}$  
is equal to a multiple of the Euclidean measure $\Eucl_{\delta}$ on $\NB(\sqrt{\delta})$. We write

\begin{equation}\label{eqos-1}
(\ft_{\delta})_{*} \nu_{\delta}=E_{\delta} \Eucl_{\delta},
\end{equation}
for some constant $E_{\delta} \ge 0$.

The other natural measure on $S_{\delta}$ is defined as follows. Let   $\chi:C_{\delta} \to S_{\delta}$ be the forgetting map (defined above).  
Recall that  $\wt{\mu}$ is the measure (defined by $(\ref{wt-mu})$ above) on the space of well connected tripods given by
$$
d\wt{\mu}(T_p,T_q,\gamma)=\taso_{\gamma}(T_p,T_q)\, d\lambda_{T}(T_p,T_q,\gamma),
$$
where $\lambda_{T}(T_p,T_q,\gamma)$ is the product of the Liouville measure $\Lambda$ (for $\FR(\M)$) on the first two terms, and the counting measure on the third term.
Then we get a new measure on $S_{\delta}$ by  $\chi_{*} (\wt{\mu}|_{C_{\delta}})$, where $\wt{\mu}|_{C_{\delta}}$ is  the restriction of $\wt{\mu}$ to the set $C_{\delta}$.

The two measures satisfy
\begin{align*}
\left| \frac{d \chi_{*}(\wt{\mu}|_{C_{\delta}})} { d\nu_{\delta}} \right| &= \sum_{\gamma_{2}} \abs_{\gamma_{2}}(\omega^{2}(F_p),\overline{\omega}^{2}(F_q)) \\
&=\abs(\omega^{2}(F_p),\overline{\omega}^{2}(F_q)).
\end{align*}
But by the mixing we have  
$$
\abs(\omega^{2}(F_p),\overline{\omega}^{2}(F_q))= \frac{1}{\Lambda(\FR(\M))} (1+ O(e^{-{\bf q} r}) ),
$$
so we find that  for some constant $C=C(\epsilon,\M)>0$ we have
$$
\left| \frac{d \chi_{*}(\wt{\mu}|_{C_{\delta}})} { d\nu_{\delta}}- \frac{1}{\Lambda(\FR(\M))}  \right|<Ce^{-{\bf q} r},
$$
which implies
\begin{equation}\label{**}
\frac{1}{\Lambda(\FR(\M))} (1-Ce^{-{\bf q} r} )\nu_{\delta} \le \chi_{*}( \wt{\mu}|_{C_{\delta}} )  \le \frac{1}{\Lambda(\FR(\M))} (1+Ce^{-{\bf q} r} ) \nu_{\delta}.
\end{equation}

Applying the mapping $(\ft_{\delta})_*$, and from (\ref{eqos-1}) we obtain

$$
\frac{E_{\delta}}{\Lambda(\FR(\M))} (1-Ce^{-{\bf q} r} )\Eucl_{\delta} \le \ft_{*}(\chi_{*}( \wt{\mu}|_{C_{\delta}}))  \le \frac{E_{\delta}}{\Lambda(\FR(\M))} (1+Ce^{-{\bf q} r} ) \Eucl_{\delta}.
$$

\vskip .3cm
We  let 
$$
\beta_{\delta}= \ft_{*}(\chi_{*}( \wt{\mu}|_{C_{\delta}})).
$$ 
It follows that the Radon-Nikodym derivative of $\beta_{\delta}$ satisfies desired inequality from Theorem \ref{theorem-mes}. On the other hand, it follows from
(\ref{f-foot-new}) that $\beta_{\delta}$ and $\ph\mu|_{\NB(\sqrt{\delta}) }$ are $O(e^{-\frac{r}{4}})$ equivalent. This completes the proof.

\end{document}

%% file: strans.pdf_t
\begin{picture}(0,0)%
\includegraphics{strans.pdf}%
\end{picture}%
\setlength{\unitlength}{3947sp}%
\begingroup\makeatletter\ifx\SetFigFont\undefined%
\gdef\SetFigFont#1#2#3#4#5{%
  \reset@font\fontsize{#1}{#2pt}%
  \fontfamily{#3}\fontseries{#4}\fontshape{#5}%
  \selectfont}%
\fi\endgroup%
\begin{picture}(6024,5299)(664,-5573)
\put(5966,-4886){\makebox(0,0)[lb]{\smash{{\SetFigFont{12}{14.4}{\familydefault}{\mddefault}{\updefault}{\color[rgb]{0,0,0}$D_3$}%
}}}}
\put(1061,-1441){\makebox(0,0)[lb]{\smash{{\SetFigFont{12}{14.4}{\familydefault}{\mddefault}{\updefault}{\color[rgb]{0,0,0}$C_0$}%
}}}}
\put(2371,-1336){\makebox(0,0)[lb]{\smash{{\SetFigFont{12}{14.4}{\familydefault}{\mddefault}{\updefault}{\color[rgb]{0,0,0}$C_1$}%
}}}}
\put(6181,-1796){\makebox(0,0)[lb]{\smash{{\SetFigFont{12}{14.4}{\familydefault}{\mddefault}{\updefault}{\color[rgb]{0,0,0}$C_4$}%
}}}}
\put(2391,-2616){\makebox(0,0)[lb]{\smash{{\SetFigFont{12}{14.4}{\familydefault}{\mddefault}{\updefault}{\color[rgb]{0,0,0}$F_0$}%
}}}}
\put(3511,-3041){\makebox(0,0)[lb]{\smash{{\SetFigFont{12}{14.4}{\familydefault}{\mddefault}{\updefault}{\color[rgb]{0,0,0}$O$}%
}}}}
\put(2019,-1993){\makebox(0,0)[lb]{\smash{{\SetFigFont{12}{14.4}{\familydefault}{\mddefault}{\updefault}{\color[rgb]{0,0,0}$D_0$}%
}}}}
\put(2646,-2931){\makebox(0,0)[lb]{\smash{{\SetFigFont{12}{14.4}{\familydefault}{\mddefault}{\updefault}{\color[rgb]{0,0,0}$N_1$}%
}}}}
\put(3561,-3631){\makebox(0,0)[lb]{\smash{{\SetFigFont{12}{14.4}{\familydefault}{\mddefault}{\updefault}{\color[rgb]{0,0,0}$D_1$}%
}}}}
\put(4331,-3256){\makebox(0,0)[lb]{\smash{{\SetFigFont{12}{14.4}{\familydefault}{\mddefault}{\updefault}{\color[rgb]{0,0,0}$F_1$}%
}}}}
\put(4806,-4348){\makebox(0,0)[lb]{\smash{{\SetFigFont{12}{14.4}{\familydefault}{\mddefault}{\updefault}{\color[rgb]{0,0,0}$D_2$}%
}}}}
\put(3981,-1981){\makebox(0,0)[lb]{\smash{{\SetFigFont{12}{14.4}{\familydefault}{\mddefault}{\updefault}{\color[rgb]{0,0,0}$C_2$}%
}}}}
\put(5171,-2178){\makebox(0,0)[lb]{\smash{{\SetFigFont{12}{14.4}{\familydefault}{\mddefault}{\updefault}{\color[rgb]{0,0,0}$C_3$}%
}}}}
\end{picture}%

%% file: transversal.pdf_t
\begin{picture}(0,0)%
\includegraphics{transversal.pdf}%
\end{picture}%
\setlength{\unitlength}{3947sp}%
\begingroup\makeatletter\ifx\SetFigFont\undefined%
\gdef\SetFigFont#1#2#3#4#5{%
  \reset@font\fontsize{#1}{#2pt}%
  \fontfamily{#3}\fontseries{#4}\fontshape{#5}%
  \selectfont}%
\fi\endgroup%
\begin{picture}(6504,4684)(69,-4078)
\put(2082,-3927){\makebox(0,0)[lb]{\smash{{\SetFigFont{12}{14.4}{\familydefault}{\mddefault}{\updefault}{\color[rgb]{0,0,0}$C_1(0)$}%
}}}}
\put(3537,-3812){\makebox(0,0)[lb]{\smash{{\SetFigFont{12}{14.4}{\familydefault}{\mddefault}{\updefault}{\color[rgb]{0,0,0}$C_2(0)$}%
}}}}
\put(5652,-3275){\makebox(0,0)[lb]{\smash{{\SetFigFont{12}{14.4}{\familydefault}{\mddefault}{\updefault}{\color[rgb]{0,0,0}$C_3(0)$}%
}}}}
\put(2303,-213){\makebox(0,0)[lb]{\smash{{\SetFigFont{12}{14.4}{\familydefault}{\mddefault}{\updefault}{\color[rgb]{0,0,0}$w_2^+$}%
}}}}
\put(1775,-43){\makebox(0,0)[lb]{\smash{{\SetFigFont{12}{14.4}{\familydefault}{\mddefault}{\updefault}{\color[rgb]{0,0,0}$w_2^-$}%
}}}}
\put(1422,-803){\makebox(0,0)[lb]{\smash{{\SetFigFont{12}{14.4}{\familydefault}{\mddefault}{\updefault}{\color[rgb]{0,0,0}$w_1^-$}%
}}}}
\put(2138,-2907){\makebox(0,0)[lb]{\smash{{\SetFigFont{12}{14.4}{\familydefault}{\mddefault}{\updefault}{\color[rgb]{0,0,0}$z_1(0)$}%
}}}}
\put(4568,-2619){\makebox(0,0)[lb]{\smash{{\SetFigFont{12}{14.4}{\familydefault}{\mddefault}{\updefault}{\color[rgb]{0,0,0}$z_3(0)$}%
}}}}
\put(2694,-2825){\makebox(0,0)[lb]{\smash{{\SetFigFont{12}{14.4}{\familydefault}{\mddefault}{\updefault}{\color[rgb]{0,0,0}$z_2(0)$}%
}}}}
\put(1606,-2998){\makebox(0,0)[lb]{\smash{{\SetFigFont{12}{14.4}{\familydefault}{\mddefault}{\updefault}{\color[rgb]{0,0,0}$z_0(0)$}%
}}}}
\put(1272,-1599){\makebox(0,0)[lb]{\smash{{\SetFigFont{12}{14.4}{\familydefault}{\mddefault}{\updefault}{\color[rgb]{0,0,0}$w_0^-$}%
}}}}
\put(1820,-1427){\makebox(0,0)[lb]{\smash{{\SetFigFont{12}{14.4}{\familydefault}{\mddefault}{\updefault}{\color[rgb]{0,0,0}$w_0^+$}%
}}}}
\put(1986,-748){\makebox(0,0)[lb]{\smash{{\SetFigFont{12}{14.4}{\familydefault}{\mddefault}{\updefault}{\color[rgb]{0,0,0}$w_1^+$}%
}}}}
\put( 84,-3196){\makebox(0,0)[lb]{\smash{{\SetFigFont{12}{14.4}{\familydefault}{\mddefault}{\updefault}{\color[rgb]{0,0,0}$\gamma(0)$}%
}}}}
\put(1224,-4000){\makebox(0,0)[lb]{\smash{{\SetFigFont{12}{14.4}{\familydefault}{\mddefault}{\updefault}{\color[rgb]{0,0,0}$C_0(0)$}%
}}}}
\end{picture}%

%% file: bipods.pdf_t
\begin{picture}(0,0)%
\includegraphics{bipods.pdf}%
\end{picture}%
\setlength{\unitlength}{3947sp}%
\begingroup\makeatletter\ifx\SetFigFont\undefined%
\gdef\SetFigFont#1#2#3#4#5{%
  \reset@font\fontsize{#1}{#2pt}%
  \fontfamily{#3}\fontseries{#4}\fontshape{#5}%
  \selectfont}%
\fi\endgroup%
\begin{picture}(4063,3277)(430,-2434)
\put(3248,-2356){\makebox(0,0)[lb]{\smash{{\SetFigFont{12}{14.4}{\familydefault}{\mddefault}{\updefault}{\color[rgb]{0,0,0}$q_0'$}%
}}}}
\put(1351,-811){\makebox(0,0)[lb]{\smash{{\SetFigFont{12}{14.4}{\familydefault}{\mddefault}{\updefault}{\color[rgb]{0,0,0}$\eta_p$}%
}}}}
\put(3548,-886){\makebox(0,0)[lb]{\smash{{\SetFigFont{12}{14.4}{\familydefault}{\mddefault}{\updefault}{\color[rgb]{0,0,0}$\eta_q$}%
}}}}
\put(3488,134){\makebox(0,0)[lb]{\smash{{\SetFigFont{12}{14.4}{\familydefault}{\mddefault}{\updefault}{\color[rgb]{0,0,0}$q_1'$}%
}}}}
\put(2806,-2004){\makebox(0,0)[lb]{\smash{{\SetFigFont{12}{14.4}{\familydefault}{\mddefault}{\updefault}{\color[rgb]{0,0,0}$v_0'$}%
}}}}
\put(1426,-2288){\makebox(0,0)[lb]{\smash{{\SetFigFont{12}{14.4}{\familydefault}{\mddefault}{\updefault}{\color[rgb]{0,0,0}$\hat p_0$}%
}}}}
\put(4478,-1081){\makebox(0,0)[lb]{\smash{{\SetFigFont{12}{14.4}{\familydefault}{\mddefault}{\updefault}{\color[rgb]{0,0,0}$q$}%
}}}}
\put(1583,-1846){\makebox(0,0)[lb]{\smash{{\SetFigFont{12}{14.4}{\familydefault}{\mddefault}{\updefault}{\color[rgb]{0,0,0}$p_0'$}%
}}}}
\put(3572,531){\makebox(0,0)[lb]{\smash{{\SetFigFont{12}{14.4}{\familydefault}{\mddefault}{\updefault}{\color[rgb]{0,0,0}$\hat q_1$}%
}}}}
\put(1523,269){\makebox(0,0)[lb]{\smash{{\SetFigFont{12}{14.4}{\familydefault}{\mddefault}{\updefault}{\color[rgb]{0,0,0}$\hat p_1$}%
}}}}
\put(1426,651){\makebox(0,0)[lb]{\smash{{\SetFigFont{12}{14.4}{\familydefault}{\mddefault}{\updefault}{\color[rgb]{0,0,0}$p_1'$}%
}}}}
\put(1967,660){\makebox(0,0)[lb]{\smash{{\SetFigFont{12}{14.4}{\familydefault}{\mddefault}{\updefault}{\color[rgb]{0,0,0}$u_1'$}%
}}}}
\put(593,-552){\makebox(0,0)[lb]{\smash{{\SetFigFont{12}{14.4}{\familydefault}{\mddefault}{\updefault}{\color[rgb]{0,0,0}$\omega u$}%
}}}}
\put(445,-991){\makebox(0,0)[lb]{\smash{{\SetFigFont{12}{14.4}{\familydefault}{\mddefault}{\updefault}{\color[rgb]{0,0,0}$p$}%
}}}}
\put(4280,-578){\makebox(0,0)[lb]{\smash{{\SetFigFont{12}{14.4}{\familydefault}{\mddefault}{\updefault}{\color[rgb]{0,0,0}$\overline \omega v$}%
}}}}
\put(4152,-1504){\makebox(0,0)[lb]{\smash{{\SetFigFont{12}{14.4}{\familydefault}{\mddefault}{\updefault}{\color[rgb]{0,0,0}$v$}%
}}}}
\put(687,-1441){\makebox(0,0)[lb]{\smash{{\SetFigFont{12}{14.4}{\familydefault}{\mddefault}{\updefault}{\color[rgb]{0,0,0}$u$}%
}}}}
\put(3286,-1965){\makebox(0,0)[lb]{\smash{{\SetFigFont{12}{14.4}{\familydefault}{\mddefault}{\updefault}{\color[rgb]{0,0,0}$\hat q_0$}%
}}}}
\end{picture}%